\documentclass[11pt,twoside]{article}

\usepackage[english]{babel}

\usepackage{amsmath}
\usepackage{amssymb}
\usepackage{mathrsfs}
\usepackage{amsthm}

\usepackage{indentfirst}
\usepackage{color}
\usepackage{txfonts}

\usepackage{anysize}

\usepackage[colorlinks=true,
  linkcolor=blue,
  citecolor=red,
  urlcolor=magenta,
  backref=page
  ]{hyperref}

\allowdisplaybreaks

\newtheorem{theorem}{Theorem}[section]
\newtheorem{lemma}[theorem]{Lemma}
\newtheorem{corollary}[theorem]{Corollary}
\newtheorem{proposition}[theorem]{Proposition}

\theoremstyle{definition}
\newtheorem{remark}[theorem]{Remark}
\newtheorem{definition}[theorem]{Definition}

\numberwithin{equation}{section}

\begin{document}

\title{\bf\Large  Variable Muckenhoupt $A_\infty$ Weights
\footnotetext{\hspace{-0.35cm} 2020 {\it
Mathematics Subject Classification}. Primary 46E30; Secondary 47A56, 15A15,
46E40, 42B99. \endgraf
{\it Key words and phrases}. Muckenhoupt weight, matrix weight,
variable Lebesgue space, reverse H\"older's inequality, minimal operator,
weight dimension.
}}
\date{}
\author{}
\author{Dachun Yang\footnote{Corresponding author,
E-mail: \texttt{dcyang@bnu.edu.cn}/{\color{red}\today}/Final version.},\ \
Wen Yuan and Zongze Zeng}
\maketitle

\vspace{-0.7cm}

\begin{center}
\begin{minipage}{13cm}
{\small {\bf Abstract:}\quad
In this article, with introducing concepts of
variable scalar $\mathcal{A}_{p(\cdot),\infty}$ weights
and variable matrix $\mathscr{A}_{p(\cdot),\infty}$ weights,
we seek a comprehensive theory of $A_\infty$ weights
within the framework of variable exponent spaces.
We first show that a weight belongs to $\mathcal{A}_{p(\cdot),\infty}$
if and only if its $p(\cdot)$-th power is an $A_\infty$ weight.
Using this, we characterize the $\mathcal{A}_{p(\cdot),\infty}$ condition
by the minimal operator. Then we establish the reverse H\"older's
inequality for $\mathcal{A}_{p(\cdot),\infty}$ weights
in variable Lebesgue spaces with explicit constants
and, combining this with the previously established relationship
between $\mathcal{A}_{p(\cdot),\infty}$ weights and $A_\infty$ weights,
we prove that, for any weight $w$,
the reverse H\"older's inequality holds in variable Lebesgue spaces
if and only if $w$ is an $\mathcal{A}_{p(\cdot),\infty}$ weight.
For the matrix $\mathscr{A}_{p(\cdot),\infty}$ weights,
we first show the existence of the reducing operators
for matrix $\mathscr{A}_{p(\cdot),\infty}$ weights
and then, combining the matrix $\mathscr{A}_{p(\cdot),\infty}$ weights
with the scalar $\mathcal{A}_{p(\cdot),\infty}$ weights,
we establish the reverse H\"older's inequality for $\mathscr{A}_{p(\cdot),\infty}$ weights
in variable Lebesgue spaces.
Finally, for further applications to
variable matrix-weighted function spaces,
we introduce the upper and the lower dimensions for $\mathscr{A}_{p(\cdot),\infty}$ weights
and use these concepts to establish the sharp estimate involving reducing operators.
}
\end{minipage}
\end{center}

\vspace{0.2cm}

\tableofcontents

\section{Introduction}

Throughout this article, we work in $\mathbb{R}^n$,
which serves as our default underlying space
unless otherwise specified.

The Muckenhoupt weights were introduced by Muckenhoupt \cite{m72} to characterize the
boundedness of the Hardy--Littlewood maximal operator on weighted Lebesgue spaces.
From then on, the theory of Muckenhoupt weights has found wide applications in harmonic analysis (see, for example, \cite{cf74,hmw73,h12,hp13}) and
partial differential equations (see, for example, \cite{bhs11,dk18,gl86}).
On the other hand, the study of variable function spaces can be traced
to Orlicz \cite{o31} in 1931
and, driven by   studies of  partial differential equations
with   non-standard growth, it  has achieved  substantial advancements from 1990s;
see, for example, \cite{am02,am02 1,am05,r00,r04,cf13}. In 2011, to study
the boundedness of the Hardy--Littlewood maximal operator on variable Lebesgue spaces,
Cruz-Uribe et al. \cite{cdh11} introduced the concept of variable
Muckenhoupt $\mathcal{A}_{p(\cdot)}$ weights
(see Definition \ref{def Ap 1}) and proved  that the boundedness
of the Hardy--Littlewood maximal operator
on weighted variable Lebesgue spaces $L^{p(\cdot)}_w$
(see Definition \ref{def weight Leb} for its definition)
is equivalent to the condition that $w\in \mathcal{A}_{p(\cdot)}$; see also
  \cite{cfn12} for a related result on the weak-type norm inequality.
We also refer to \cite{cw17,cc22,cp24,cr24,cs23,wx22,wgx24} for more studies about variable weights
and variable exponent weighted function spaces.

The study of matrix weights can be tracked back to the work of
Wiener and Masani \cite{wm58} on
the prediction theory for multivariate stochastic processes,
in which they introduced the matrix-weighted Lebesgue space
$L^2_W$, where $W$ is a matrix weight.
In 1990s, on $\mathbb{R}$
Nazarov and Treil \cite{nt96}, Treil and Volberg \cite{tv97},
and Volberg  \cite{v97}
generalized the scalar Muckenhoupt $A_p$ weights
to matrix $A_p$ weights acting on vector-valued functions
mapping $\mathbb{R}$ to $\mathbb{C}^m$ with $m\in\mathbb{N}$.
From then on, a lot of attention have been paid to the theory
of matrix weights and we refer to \cite{b01,bc22,dhl20,
g03,dptv24,kn24,llor23,llor24,nptv17,nr18}
for their applications to the boundedness
of operators,  to \cite{n12,n25-2} for their applications to
multipliers, to \cite{bx24a,bx24b,bhyy23,bhyy23 2,bhyy23 3, fr21,
r03,lyy24a,rou04,wyy23} for their applications to matrix-weighted
Besov--Triebel--Lizorkin spaces, and to
\cite{bgx25,bx24c,bcyy24,cyy25,dly21,lyy24b,n25,n24,wgx25a,wgx25b} for
their applications to other matrix-weighted function spaces.
We also refer to survey articles \cite{byyyz25,c25}
for more recent results related to matrix weights.
In particular, recently Cruz-Uribe and Penrod in \cite{cp23}
introduced the concept of variable matrix $\mathscr{A}_{p(\cdot)}$ weights
with the variable exponent $p(\cdot)$
and established the reverse H\"older's inequality for
matrix $\mathscr{A}_{p(\cdot)}$ weights later in \cite{cp24}.
Very recently, Nieraeth and Penrod \cite{np25} proved
the boundedness of Calder\'on--Zygmund operators and
Christ--Goldberg maximal operators from \cite{cg01}
on the variable matrix $\mathscr{A}_{p(\cdot)}$ weighted Lebesgue spaces.
We refer to \cite{cs25} for more studies about variable matrix weights.

Compared with the relationship between matrix $A_p$ weights and scalar $A_p$
weights, the difference  between matrix  $A_\infty$ weights and scalar $A_\infty$
weights is more distinct.
By splitting the matrix $A_\infty$ condition into a family of conditions
depending on $p$, Volberg  \cite{v97} introduced the concept
of matrix $A_{p,\infty}$ weights with $p\in (0,\infty)$
as an analogue of the scalar $A_\infty$ weights,
which, in the scalar-valued case, for any $p\in (0,\infty)$, $A_{p,\infty}=A_\infty$, while
this no longer holds true for matrix-valued weights.
Bu et. al. \cite{bhyy-a} later found several equivalent and
simpler characterizations of matrix $A_{p,\infty}$ weights;
see also \cite{bcyy24,bhyy24,byyz25,yymz25} and the survey
article \cite{byyyz25} for more studies about matrix $A_{p,\infty}$
weights and their applications.

Based on the aforementioned results, it is natural to ask
a comprehensive theory of
$A_\infty$ weights within the framework of variable exponent spaces,
encompassing both scalar-valued and matrix-valued weights.
This article gives an affirmative answer to this question.
Inspired by the definition of matrix $A_{p,\infty}$ weights,
for any variable exponent $p(\cdot) \in \mathcal{P}_0$
(see Section \ref{sec weight} for the definition of $\mathcal{P}_0$),
we introduce the following concepts of scalar $\mathcal{A}_{p(\cdot),\infty}$ weights
and matrix $\mathscr{A}_{p(\cdot),\infty}$ weights.

Recall that a non-negative measurable function on $\mathbb{R}^n$
that is positive almost everywhere is called a \emph{scalar weight},
and see also Definition \ref{def matrix} for the definition
of matrix weights mapping $\mathbb{R}^n$
to $\mathbb{C}^m$.
\begin{definition}
\begin{itemize}
\item[(i)] A scalar weight $w$ on $\mathbb{R}^n$ is called
an \emph{$\mathcal{A}_{p(\cdot),\infty}$ weight} if
\begin{align}\label{def Apinfty}
[w]_{\mathcal{A}_{p(\cdot),\infty}} := \sup_Q \frac{1}{\|\mathbf{1}_Q\|_{L^{p(\cdot)}}} \left\|w\mathbf{1}_Q\right\|_{L^{p(\cdot)}}
\exp\left( \fint_Q \log\left( w^{-1}(y) \right)\,dy \right) < \infty,
\end{align}
where the supremum is taken over all cubes $Q$ in $\mathbb{R}^n$.

\item[(ii)] A matrix weight $W$ mapping $\mathbb{R}^n$
to $\mathbb{C}^m$ is called a
\emph{matrix $\mathscr{A}_{p(\cdot),\infty}$ weight} if
\begin{align}\label{def matrix Apinfty}
[W]_{\mathscr{A}_{p(\cdot),\infty}} := \sup_{Q} \exp\left( \fint_Q \log\left( \frac{1}{\|\mathbf{1}_Q\|_{L^{p(\cdot)}}}
 \left\|\,\left\| W(\cdot)W^{-1}(y) \right\|\mathbf{1}_Q(\cdot)\right\|_{L^{p(\cdot)}}
 \right)\,dy \right) < \infty,
\end{align}
where the supremum is taken over all cubes $Q$ in $\mathbb{R}^n$ and
$\|\cdot\|$ denotes the operator norm of
matrices mapping $\mathbb{R}^n$ to $\mathbb{C}^m$ [see \eqref{def norm matrix}].
\end{itemize}
\end{definition}

It is obvious that \eqref{def matrix Apinfty}
in the case $m=1$ reduces to \eqref{def Apinfty}.
Thus, the $\mathcal{A}_{p(\cdot),\infty}$ weight is a special case of
the matrix $\mathscr{A}_{p(\cdot),\infty}$ weight when $m=1$.

In this article, we establish some fundamental properties of
both $\mathcal{A}_{p(\cdot),\infty}$ weights and $\mathscr{A}_{p(\cdot),\infty}$ weights.
Observe that the $p(\cdot)$-th power of $\mathcal{A}_{p(\cdot),\infty}$ weights
naturally appears in the definition of weighted variable Lebesgue spaces
(see Definition \ref{def weight Leb}) and hence plays a vital role
in the studies related to $\mathcal{A}_{p(\cdot),\infty}$ weights.
Indeed, we prove that a scalar weight belongs to $\mathcal{A}_{p(\cdot),\infty}$
if and only if its $p(\cdot)$-th power
is a scalar $A_\infty$ weight.
Using this, we characterize the $\mathcal{A}_{p(\cdot),\infty}$ condition
by the minimal operator [see \eqref{def minimal} for the definition
of the minimal operator]. Based on these, we further establish the
reverse H\"older's inequality for $\mathcal{A}_{p(\cdot),\infty}$ weights
in variable Lebesgue spaces with explicit constants
and, combining this with the aforementioned relationship between
$\mathcal{A}_{p(\cdot),\infty}$ weights and $A_\infty$ weights,
we prove that, for any weight $w$,
the reverse H\"older's inequality holds in variable Lebesgue spaces
if and only if $w$ is an $\mathcal{A}_{p(\cdot),\infty}$ weight.
For the matrix weight case, we first introduce the reducing operators related to matrix $\mathscr{A}_{p(\cdot),\infty}$ weights.
Then, by connecting the matrix $\mathscr{A}_{p(\cdot),\infty}$
with the scalar $\mathcal{A}_{p(\cdot),\infty}$ weights,
we establish the reverse H\"older's inequality for $\mathscr{A}_{p(\cdot),\infty}$ weights,
which can further be used to characterize the $\mathscr{A}_{p(\cdot),\infty}$ weights.
Finally, for further applications to variable matrix-weighted function spaces,
we introduce the upper and the lower dimensions for $\mathscr{A}_{p(\cdot),\infty}$ weights
and use these concepts to establish a sharp estimate involving reducing operators.

The reason for  introducing  $\mathcal{A}_{p(\cdot),\infty}$ weights lies in that
the constants in their related reverse H\"older's inequality depend only on the weight constant, not on the involved  weight itself. This is not the case for the $A_\infty$ weights.
Indeed, since we show that any $A_\infty$ weight must be the $p(\cdot)$-th power of
one $\mathcal{A}_{p(\cdot),\infty}$ weight,
we can also establish a reverse H\"older's inequality for any $A_\infty$ weight $\mathbb{W}$
in $\widetilde{L}^{p(\cdot)}_{\mathbb{W}}$ [see \eqref{reverse Holder WW} for more details].
However, the constants in such a reverse H\"older's inequality depend
on the weight $\mathbb{W}$ itself, not only on $[\mathbb{W}]_{A_\infty}$;
see Remark \ref{rem reverse Holder Apinfty} for more details.
Also, although $w\in \mathcal{A}_{p(\cdot),\infty}$ if and only if $w^{p(\cdot)}\in A_\infty$,
in the proof of the reverse H\"older's inequality for $\mathcal{A}_{p(\cdot),\infty}$ weights
we can not  directly use the known reverse H\"older's inequality for $A_\infty$ weights
(see, for example, \cite{hp13}) because, when $p(\cdot)$ is not a constant exponent,
we can not compare $[w]_{\mathcal{A}_{p(\cdot),\infty}}$ with $[w^{p(\cdot)}]_{A_\infty}$.
To overcome this difficulty,
we make full use of the equivalent expressions of $A_\infty$ weights
and  the properties of variable Lebesgue spaces.

We point out that these fundamental properties of $\mathscr{A}_{p(\cdot),\infty}$
obtained in this article play a key role in developing a theory of variable matrix-weighted
Besov spaces in \cite{yyz25} and will also surely be crucial tools in
establishing a corresponding theory of variable matrix-weighted
Triebel--Lizorkin spaces in a forthcoming article.

The organization of the reminder of this article is as follows.

In Section \ref{sec weight},
we study the properties of $\mathcal{A}_{p(\cdot),\infty}$ weights.
We first recall some basic properties of variable Lebesgue spaces
and give an equivalent characterization of $\mathcal{A}_{p(\cdot),\infty}$ when $p_- > 1$
(see Theorem \ref{w reverse}).
Then, in Subsection \ref{sec Apinfty A},
we establish the equivalent relationship between
$\mathcal{A}_{p(\cdot),\infty}$ weights and $A_\infty$ weights
in the sense of ignoring the $p(\cdot)$-th power
(see Theorem \ref{Apinfty A}).
To achieve this goal, through leveraging the equivalent characterizations
of $A_\infty$ weights,
we give an explicit quantitative estimate about the $p(\cdot)$-th power
$w^{p(\cdot)}$ of $w\in \mathcal{A}_{p(\cdot),\infty}$, which is proved to be
an $A_\infty$ weight (see Proposition \ref{reverse Holder w1})
and, moreover, we show that the $\frac{1}{p(\cdot)}$-th power
of $A_\infty$ weights belongs to $\mathcal{A}_{p(\cdot),\infty}$
(see Proposition \ref{Apinfty Ainfty}).
In Subsection \ref{sec Apinfty 1}, applying Theorem \ref{Apinfty A},
we establish the proper inclusion relations among $\mathcal{A}_{p(\cdot)}$, $\mathcal{A}_{p(\cdot),\infty}$,
and $A_\infty$ (see Proposition \ref{Ap Apinfty 2})
and characterize the $\mathcal{A}_{p(\cdot),\infty}$ condition by the minimal operator
(see Theorem \ref{inf Apinfty}).
In Subsection \ref{sec Apinfty},
we first establish the reverse H\"older's inequality
for $\mathcal{A}_{p(\cdot),\infty}$ weights
in variable Lebesgue spaces
and determine   the dependence property of
the constants of the reverse H\"older's inequality
(see Theorem \ref{reverse Holder Apinfty}).
Finally,  applying Theorem \ref{Apinfty A},
we further find that any weight $w$ satisfying the reverse H\"older's inequality
in variable Lebesgue spaces belongs to  $\mathcal{A}_{p(\cdot),\infty}$
(see Theorem \ref{Apinfty Holder}), which implies the equivalence between
the   $\mathcal{A}_{p(\cdot),\infty}$ condition and its related reverse H\"older's inequality.

In Section \ref{sec matrix}, we focus on
matrix $\mathscr{A}_{p(\cdot),\infty}$ weights.
We first introduce the reducing operators for matrix $\mathscr{A}_{p(\cdot),\infty}$ weights
(see Proposition \ref{ext redu}).
Then, in Subsection \ref{sec proApinfty}, we extend some scalar results to
the matrix case.
As an extension of Proposition \ref{Ap Apinfty 2},
we give the proper inclusion relations among $\mathscr{A}_{p(\cdot)}$, $\mathscr{A}_{p(\cdot),\infty}$,
and $A_{1,\infty}$ (see Proposition \ref{Ap Apinfty}).
Then, as an application of Theorem \ref{reverse Holder Apinfty},
we obtain the reverse H\"older's inequality for $\mathscr{A}_{p(\cdot),\infty}$ weights
in variable Lebesgue spaces (see Theorem \ref{reverse Holder M}).
Additionally, as a  matrix-valued variant of Theorem \ref{w reverse},
we give an equivalent characterization of $\mathscr{A}_{p(\cdot),\infty}$ weights
when $p_- > 1$ (see Proposition \ref{WM reverse 2}).
Finally, in Subsection \ref{sec Apdimesnion},
for further applications of matrix weights to developing
the corresponding theory of function spaces,
we introduce and study the upper and the lower dimensions
of $\mathscr{A}_{p(\cdot),\infty}$ weights
(see Proposition \ref{dim ext}).

In the end, we make some conventions on symbols.
Let $\mathbb{Z}$ be the collection of all integers,
$\mathbb{N}:= \{1,2,\dots\},$
and $\mathbb{Z}_+:=\mathbb{N}\cup\{0\}$.
For any measurable set $E$ in $\mathbb{R}^n$,
denote by the \emph{symbol $\mathscr{M}(E)$} the set
of all measurable functions on $E$ and, when $E = \mathbb{R}^n$,
we simply write $\mathscr{M}(\mathbb{R}^n)$ as $\mathscr{M}$.
In addition, we use the symbol $L^p_{\rm loc}$ with $p\in (0,\infty)$ to denote
the set of all locally $p$-integrable functions on $\mathbb{R}^n$.
A \emph{cube} $Q$ in $\mathbb{R}^n$ always has finite edge length
and its edges are always assumed to be parallel to coordinate axes,
but $Q$ is not necessary to be open or closed.
The symbol $Q(x,l)$ with $x \in \mathbb{R}^n$ and $l\in (0,\infty)$
denotes the cube centered at the point $x$ with the edge length $l$.
If $E$ is a measurable set in $\mathbb{R}^n$,
then we denote by $\mathbf{1}_E$ its \emph{characteristic function}
and, for any bounded measurable set $E\subset \mathbb{R}^n$ with $|E| \neq 0$
and for any $f\in L^1_{\rm loc}$,
let
$$ \fint_E f(x)\,dx := \frac{1}{|E|} \int_E f(x)\,dx. $$
For any non-empty set $E$ in $\mathbb{R}^n$ and for any $x\in \mathbb{R}^n \setminus E$,
the symbol ${\mathop\mathrm{\,dist\,}}(x,E)$ is defined to be the \emph{distance} from $x$ to $E$,
that is, ${\mathop\mathrm{\,dist\,}}(x,E) := \inf_{y\in E} \left|x-y\right|$
and, moreover, for any sets $E$ and $F$ in $\mathbb{R}^n$,
\begin{align}\label{def d}
{\mathop\mathrm{\,dist\,}}(E,F) := \sup_{x\in E} {\mathop\mathrm{\,dist\,}}(x,F).
\end{align}
For any $y\in (0,\infty)$, let
\begin{align}\label{def log plus}
\log_+(y) := \max\left\{\log(y),0\right\}.
\end{align}
For any $p\in [1,\infty]$, let $p'$ be its conjugate index,
that is, $\frac1p+\frac{1}{p'} = 1$.
We always use $C$ to denote a positive constant independent of
the main parameters involved, but it may vary from line to line.
We also use $C_{\alpha,\beta,\dots}$ to denote a positive constant depending on
the indicated parameters $\alpha,\beta,\dots$.
The symbol $f\lesssim g$ means $f\leq  Cg$
and, if $f\lesssim g\lesssim f$, we then write $f\sim g$.
Finally, in all proofs we consistently retain the symbols
introduced in the original theorem (or related statement).

\section{Variable Scalar Weights}\label{sec weight}
In this section, we focus on the properties of $\mathcal{A}_{p(\cdot),\infty}$ weights.
We first give two equivalent expressions of $\mathcal{A}_{p(\cdot),\infty}$ weights.
After that, in Subsection \ref{sec Apinfty A},
we establish the equivalence between $\mathcal{A}_{p(\cdot),\infty}$ and $A_\infty$ weights
and then, in Subsection \ref{sec Apinfty 1},
as an application of the equivalence obtained in Subsection \ref{sec Apinfty A},
we characterize the $\mathcal{A}_{p(\cdot),\infty}$ weights by the minimal operator.
Finally, in Subsection \ref{sec Apinfty},
we discuss about the reverse H\"older's inequality for $\mathcal{A}_{p(\cdot),\infty}$ weights
in variable Lebesgue spaces.

We begin with some basic concepts of variable Lebesgue spaces.
A measurable  function $p:  \mathbb{R}^n\to (0,\infty]$ is called an \emph{exponent function}.
We use the \emph{symbol $\mathcal{P}$} to denote the set of all exponent functions $p:  \mathbb{R}^n \to [1,\infty]$,
and we use the \emph{symbol $\mathcal{P}_0$} to denote the set of all exponent functions $p:  \mathbb{R}^n \to (0,\infty]$
satisfying $\mathop{\rm{ess}\inf}_{x\in \mathbb{R}^n} p(x) > 0$.
For any $p(\cdot)\in \mathcal{P}_0$ and any set $E$ in $\mathbb{R}^n$,
let
$$ p_+(E) := \mathop{\rm{ess}\sup}_{x\in E} p(x)\ \ \text{and}\ \  p_-(E) := \mathop{\rm{ess}\inf}_{x\in E} p(x); $$
moreover, write $p_+:=p_+(\mathbb{R}^n)$  and $p_-:=p_-(\mathbb{R}^n)$.

We now  recall the definition of variable Lebesgue spaces
(see, for instance, \cite[Definition 2.16]{cf13}).
\begin{definition}
The \emph{variable Lebesgue space $L^{p(\cdot)}$} associated
with the exponent function $p:  \mathbb{R}^n \to (0,\infty)$
is defined to be the set of all $f \in \mathscr{M}$
such that
$$ \|f\|_{L^{p(\cdot)}} := \inf\left\{ \lambda\in (0,\infty):  \rho_{L^{p(\cdot)}}\left(\frac{f}{\lambda}\right) \leq 1 \right\}<\infty, $$
where $\rho_{L^{p(\cdot)}}$ is the \emph{variable exponent modular}
defined by setting, for any $f \in \mathscr{M}$,
$$\rho_{L^{p(\cdot)}} (f) :=  \int_{\mathbb{R}^n} \left|f(x)\right|^{p(x)}\,dx. $$
\end{definition}

The following log-H\"older continuous condition of variable exponents
is a widely used assumption for variable exponents (see, for instance, \cite[Definition 2.2]{cf13}).

\begin{definition}
A measurable real-valued  function $r$ on $\mathbb{R}^n$ is said to be
\emph{locally log-H\"older continuous},
denoted by $r(\cdot)  \in LH_0$,
if there exists a positive constant $C_0$
such that, for any $x,y\in\mathbb{R}^n$ with $|x-y| < \frac12$,
\begin{align}\label{clogp 1}
|r(x)-r(y)| \leq -\frac{C_0}{\log(|x-y|)}.
\end{align}
The function $r$ is said to be \emph{log-H\"older continuous at infinity},
denoted by $r(\cdot)  \in LH_\infty$,
if there exist positive constants $r_\infty$ and $C_{\infty}$
such that, for any $x\in \mathbb{R}^n$,
\begin{align}\label{clogp infty}
|r(x)-r_\infty| \leq \frac{C_\infty}{\log(e+|x|)}.
\end{align}
Furthermore, $r$ is said to be \emph{globally log-H\"older continuous},
denoted by $r(\cdot) \in LH$,
if $r$ is both locally log-H\"older continuous
and log-H\"older continuous at infinity.
\end{definition}
\begin{remark}\label{rem p}
From \cite[Proposition 2.3]{cf13},
it follows that, if $r(\cdot) \in \mathcal{P}_0$ with $r(\cdot)\in LH$,
then $\frac{1}{r(\cdot)}\in LH$.
\end{remark}

The following result is an equivalent expression of $\mathcal{A}_{p(\cdot),\infty}$ weights
(see, for instance, \cite[Proposition 7.3.2]{g14} for the related result on scalar $A_\infty$ weights).

\begin{proposition}\label{eq infty}
Let $p(\cdot)\in \mathcal{P}_0$ with $p(\cdot) \in LH$
and let $w$ be a weight function.
If $\log_+( w^{-1} ) \in L^1_{\rm loc}$,
where $\log_+$ is the same as in \eqref{def log plus},
then
\begin{align}\label{eq infty 1}
\left[w\right]_{\mathcal{A}_{p(\cdot),\infty}}
&\sim \sup_{Q}
\sup_{H\in \mathcal{F}_{Q,w}} \frac{\|w\mathbf{1}_Q \|_{L^{p(\cdot)}}}{\| wH\mathbf{1}_Q \|_{L^{p(\cdot)}}}  \exp\left( \fint_{Q} \log\left(H(y)\right)\,dy \right),
\end{align}
where the supremum is taken over all cubes $Q$ in $\mathbb{R}^n$,
the positive equivalence constants depend only on $p(\cdot)$,
and
\begin{align*}
\mathcal{F}_{Q,w} := &\left\{ H\in \mathscr{M}:
\left\|wH\mathbf{1}_Q \right\|_{L^{p(\cdot)}} \neq 0, \log H \in L^1(Q) \right\}. \nonumber
\end{align*}
\end{proposition}

If a positive constant $C$ depends on some parameters associated with $p(\cdot)$
or, more precisely, on some of  $\{p_-,p_+,p_\infty,C_0,C_\infty\}$,
then $C$ is said to \emph{depend on $p(\cdot)$},
denoted by $C_{p(\cdot)}$.
To prove Proposition \ref{eq infty},
we need some basic properties of variable Lebesgue spaces.
The first one is H\"older's inequality in
variable Lebesgue spaces as follows, which is exactly \cite[Theorem 2.26]{cf13}.

\begin{lemma}\label{Holder}
Let $p(\cdot) \in \mathcal{P}$.
If $f\in L^{p(\cdot)}$ and $g\in L^{p'(\cdot)}$,
where $p'(\cdot) \in \mathcal{P}$ is an exponent function
such that, for almost every $x\in\mathbb{R}^n$,
$ \frac{1}{p(x)} + \frac{1}{p'(x)} = 1 $,
then $fg\in L^1$ and
there exists a positive constant $C_{p(\cdot)}$, depending only on $p(\cdot)$,
such that
$$ \int_{\mathbb{R}^n} \left| f(x)g(x) \right|\,dx \leq C_{p(\cdot)} \|f\|_{L^{p(\cdot)}} \|g\|_{L^{p'(\cdot)}}. $$
\end{lemma}
In what follows, for any $p(\cdot) \in \mathcal{P}_0$ and any cube $Q$,
let $p_Q := [\fint_Q \frac{dx}{p(x)}]^{-1}$. Clearly,
when $p(\cdot)\equiv p$ is a constant, we have $p_Q = p$.
The following lemma is precisely \cite[Theorem 4.5.7]{dhr17}.

\begin{lemma}\label{est Q}
Let $p(\cdot) \in \mathcal{P}$ with $p(\cdot) \in LH$.
Then, for any cube $Q$ in $\mathbb{R}^n$,
$$ \left| Q \right|^{\frac{1}{p_Q}}
\sim \left\| \mathbf{1}_Q \right\|_{L^{p(\cdot)}}
\ \ \text{and}\ \
 \left| Q \right|^{\frac{1}{p'_Q}}
\sim \left\| \mathbf{1}_Q \right\|_{L^{p'(\cdot)}}, $$
where the positive equivalence constants depend only on $p(\cdot)$ and $n$.
\end{lemma}
\begin{remark}\label{rem est Q}
By \cite[Proposition 3.8]{cr24},
we find that, for any $p(\cdot) \in \mathcal{P}$ with $p(\cdot) \in LH$ and any cube $Q$ in $\mathbb{R}^n$,
$\frac{1}{6} |Q|^{\frac{1}{p_Q}} \leq \|\mathbf{1}_Q\|_{L^{p(\cdot)}}$.
\end{remark}
From Lemmas \ref{Holder} and \ref{est Q},
we immediately infer the following result;
we omit the details here.
\begin{lemma}\label{est fQ}
Let $p(\cdot) \in \mathcal{P}$ with $p(\cdot) \in LH$.
Then there exists a positive constant $C_{p(\cdot),n}$,
depending only on $n$ and $p(\cdot)$,
such that, for any $f\in \mathscr{M}$ and any cube $Q$ in $\mathbb{R}^n$,
\begin{align*}
\fint_Q \left|f(x)\right|\,dx \leq C_{p(\cdot),n}\frac{1}{\|\mathbf{1}_Q\|_{L^{p(\cdot)}}} \left\| f \mathbf{1}_Q \right\|_{L^{p(\cdot)}}.
\end{align*}
\end{lemma}
The following result is the convexification of variable Lebesgue spaces
(see, for instance, \cite[Proposition 2.18]{cf13} and \cite[Lemma 3.2.6]{dhr17}).
\begin{lemma}\label{con f}
Let $p(\cdot) \in\mathcal{P}_0$ with $p_+ < \infty$.
Then, for any $r \in (0,\infty)$ and any $f\in \mathscr{M}$,
$ \|f\|_{L^{rp(\cdot)}} = \| |f|^r \|^\frac1r_{L^{p(\cdot)}} $.
\end{lemma}
Next, we give the proof of Proposition \ref{eq infty}.

\begin{proof}[Proof of Proposition \ref{eq infty}]
Notice that, if we fix $H := w^{-1}$,
then, by the assumption $\log_+(w^{-1}) \in L^1_{\rm loc}$,
we conclude $H\in \mathcal{F}_{Q,w}$ for any cube $Q$
and the right-hand side, without both suprema, of \eqref{eq infty 1}
with this special $H$ coincides with the right-hand side, without the supremum, of \eqref{def Apinfty}.
From this and \eqref{def Apinfty}, we deduce that
\begin{align*}
\left[w\right]_{\mathcal{A}_{p(\cdot),\infty}}
\leq \sup_{Q}
\sup_{H\in \mathcal{F}_{Q,w}} \frac{\|w\mathbf{1}_Q \|_{L^{p(\cdot)}}}{\| wH\mathbf{1}_Q \|_{L^{p(\cdot)}}}  \exp\left( \fint_{Q} \log\left(H(y)\right)\,dy \right).
\end{align*}

Now, we show the converse inequality. Let $r := \min\{1,p_-\}$.
From the definition of $[w]_{\mathcal{A}_{p(\cdot),\infty}}$, Jensen's inequality,
and Lemmas \ref{con f} and \ref{est fQ} with $\frac{p(\cdot)}{r} \in \mathcal{P}$,
it follows that, for any cube $Q$ in $\mathbb{R}^n$ and any $H\in \mathcal{F}_{Q,w}$,
\begin{align*}
\exp\left( \fint_{Q} \log\left( H(y) \right)\,dy \right)
& \leq \left[\exp\left( \fint_{Q} \log\left( w(y)^rH(y)^r \right)\,dy \right) \right]^\frac1r
  \exp\left( \fint_{Q} \log\left( w^{-1}(y) \right)\,dy \right) \\
&\lesssim  \left[\frac{1}{\|\mathbf{1}_Q\|_{L^{\frac{p(\cdot)}{r}}}} \left\|w^r H^r \mathbf{1}_Q \right\|_{L^{\frac{p(\cdot)}{r}}}\right]^\frac1r \exp\left( \fint_{Q} \log\left( w^{-1}(y) \right)\,dy \right)\\
& = \frac{1}{\|\mathbf{1}_Q\|_{L^{p(\cdot)}}} \left\|wH\mathbf{1}_Q \right\|_{L^{p(\cdot)}} \exp\left( \fint_{Q} \log\left( w^{-1}(y) \right)\,dy \right),
\end{align*}
which, together with dividing both sides by $\frac{\|wH\mathbf{1}_Q \|_{L^{p(\cdot)}}}{\|w\mathbf{1}_Q\|_{L^{p(\cdot)}}} $
and taking the supremum over all cubes $Q$,
further implies that
\begin{align*}
\frac{\|w\mathbf{1}_Q \|_{L^{p(\cdot)}}}{\| wH\mathbf{1}_Q \|_{L^{p(\cdot)}}}  \exp\left( \fint_{Q} \log\left(H(y)\right)\,dy \right)\lesssim \left[w\right]_{\mathcal{A}_{p(\cdot),\infty}} .
\end{align*}
Hence, combining this with taking the supremum over all functions $H \in \mathcal{F}_{Q,w}$
and all cubes $Q$,
we conclude that
\begin{align*}
\sup_{Q} \sup_{H\in \mathcal{F}_{Q,w}} \frac{\|w\mathbf{1}_Q \|_{L^{p(\cdot)}}}{\| wH\mathbf{1}_Q \|_{L^{p(\cdot)}}}  \exp\left( \fint_{Q} \log\left(H(y)\right)\,dy \right)
\lesssim \left[w\right]_{\mathcal{A}_{p(\cdot),\infty}}.
\end{align*}
This finishes the proof of Proposition \ref{eq infty}.
\end{proof}

Observe that,
by the definition of $\mathcal{A}_{p(\cdot),\infty}$ and Jensen's inequality,
for any $w \in \mathcal{A}_{p(\cdot),\infty}$ and any cube $Q$ in $\mathbb{R}^n$,
we have
\begin{align}\label{wM 2}
\frac{1}{\|\mathbf{1}_Q\|_{L^{p(\cdot)}}} \left\| w\mathbf{1}_Q \right\|_{L^{p(\cdot)}} \leq [w]_{\mathcal{A}_{p(\cdot),\infty}} \fint_Q w(x)\,dx.
\end{align}
Inspired by this observation,
we introduce the following concept of $\mathcal{A}_{p(\cdot),\infty}^\ast$ weights,
which is proved to be equivalent with $\mathcal{A}_{p(\cdot),\infty}$ under some additional conditions.
\begin{definition}
Let $p(\cdot) \in \mathcal{P}$.
A scalar weight $w$ is called an \emph{$\mathcal{A}_{p(\cdot),\infty}^\ast$ weight}
if
\begin{align*}
[w]_{\mathcal{A}_{p(\cdot),\infty}^\ast} := \sup_Q \frac{1}{\|\mathbf{1}_Q\|_{L^{p(\cdot)}}} \left\|w\mathbf{1}_Q\right\|_{L^{p(\cdot)}} \left[ \fint_Q w(x)\,dx \right]^{-1} < \infty,
\end{align*}
where the supremum is taken over all cubes in $\mathbb{R}^n$.
\end{definition}
\begin{remark}
When $p(\cdot) \equiv p \in (1,\infty)$ is a constant exponent,
it is obvious that $\mathcal{A}_{p,\infty}^\ast$ reduces to
the reverse H\"older class $RH_{p}$.
\end{remark}
Next, we recall the definition of Muckenhoupt $A_\infty$ weights
(see, for example, \cite{g14}).
\begin{definition}
A scalar weight $w$ is called an \emph{$A_\infty$-weight} if
\begin{align*}
[w]_{A_\infty} := \sup_{Q}\fint_Q w(x)\,dx\, \exp\left( \fint_Q \log\left( w^{-1}(x) \right) \,dx\right) <\infty,
\end{align*}
where the supremum is taken over all cubes in $\mathbb{R}^n$.
\end{definition}
From \eqref{wM 2}, it follows immediately that
$\mathcal{A}_{p(\cdot),\infty}\subset \mathcal{A}^\ast_{p(\cdot),\infty}$
and $ [\cdot]_{\mathcal{A}_{p(\cdot),\infty}^\ast} \leq [\cdot]_{\mathcal{A}_{p(\cdot),\infty}} $,
but the reverse inclusion is not always true.
Indeed, if $p(\cdot):= 1$, then \eqref{wM 2} holds for any weight function,
which implies that $\mathcal{A}_{p(\cdot),\infty}^\ast$ is strictly
larger than $\mathcal{A}_{1,\infty} = A_\infty$.
The following result shows the equivalence between
$\mathcal{A}_{p(\cdot),\infty}$ and $\mathcal{A}_{p(\cdot),\infty}^\ast$
under some additional conditions.
\begin{theorem}\label{w reverse}
Let $p(\cdot) \in \mathcal{P}$ with $p(\cdot) \in LH$ and $p_- > 1$.
Then $\mathcal{A}_{p(\cdot),\infty} = \mathcal{A}^\ast_{p(\cdot),\infty}$.
\end{theorem}
\begin{proof}
By the just above discussion,
to prove this theorem,
it suffices to show $\mathcal{A}_{p(\cdot),\infty}^\ast \subset \mathcal{A}_{p(\cdot),\infty}$.
To this end, let $w \in \mathcal{A}_{p(\cdot),\infty}^\ast$.
Then, from Lemmas \ref{Holder}, \ref{est fQ},
and \ref{con f} and from the definition of $\mathcal{A}_{p(\cdot),\infty}^\ast$,
it follows that, for any cube $Q$ in $\mathbb{R}^n$,
\begin{align*}
\left[\fint_{Q} w^{p_-}(x)\,dx \right]^{\frac{1}{p_-}}
\lesssim \left[\frac{1}{\|\mathbf{1}_Q\|_{L^{\frac{p(\cdot)}{p_-}}}} \left\|w^{p_-}\mathbf{1}_Q\right\|_{L^{\frac{p(\cdot)}{p_-}}}\right]^{\frac{1}{p_-}}
= \frac{1}{\|\mathbf{1}_Q\|_{L^{p(\cdot)}}} \left\|w\mathbf{1}_Q\right\|_{L^{p(\cdot)}}
\lesssim \fint_Q w(x)\,dx.
\end{align*}
Using this and the fact $p_- > 1$,
we find that $ w \in RH_{p_-} $
and hence, by \cite[Theorem 7.3.3]{g14},
we obtain $w \in A_\infty$.
Combining this with the assumption $w \in \mathcal{A}_{p(\cdot),\infty}^\ast$,
we conclude that, for any cube $Q$ in $\mathbb{R}^n$,
\begin{align*}
\frac{1}{\|\mathbf{1}_Q\|_{L^{p(\cdot)}}} \left\|w\mathbf{1}_Q\right\|_{L^{p(\cdot)}} \exp\left( \fint_Q \log\left( w^{-1}(x) \right)\,dx \right)
\lesssim \fint_Q w(x)\,dx \exp\left( \fint_Q \log\left( w^{-1}(x) \right)\,dx \right)
\leq [w]_{A_\infty},
\end{align*}
which further implies that $w \in \mathcal{A}_{p(\cdot),\infty}$.
This finishes the proof of Theorem \ref{w reverse}.
\end{proof}

\subsection{Relationship Between $\mathcal{A}_{p(\cdot),\infty}$
and $A_\infty$}\label{sec Apinfty A}

In this subsection, we establish the following relationship
between $\mathcal{A}_{p(\cdot),\infty}$ and $A_\infty$.
\begin{theorem}\label{Apinfty A}
Let $p(\cdot)\in\mathcal{P}_0$ with $p(\cdot) \in LH$.
Then $w\in \mathcal{A}_{p(\cdot),\infty}$ if and only if
$\mathbb{W} := [w(\cdot)]^{p(\cdot)} \in A_\infty. $
\end{theorem}

To make the proof clear,
we first break the proof of Theorem \ref{Apinfty A}, under the additional assumption $p(\cdot) \in \mathcal{P}$,  into
the necessity part (see Proposition \ref{reverse Holder w1})
and the sufficiency part (see Proposition \ref{Apinfty Ainfty}).
Then, using the convexification property of $\mathcal{A}_{p(\cdot),\infty}$ weights (see Lemma \ref{wr}),
we give the proof of Theorem \ref{Apinfty A} for the general case $p(\cdot)\in\mathcal{P}_0$.

Here, and thereafter, we write
\begin{align}\label{def Q0}
Q_0 := Q(\mathbf{0},2e).
\end{align} We begin with the proof of  the necessity when $p(\cdot) \in \mathcal{P}$,
which is also vital for establishing the reverse H\"older's inequality for $\mathbb{W}$.

\begin{proposition}\label{reverse Holder w1}
Let $p(\cdot) \in \mathcal{P}$ with $p(\cdot) \in LH$ and
let $w \in \mathcal{A}_{p(\cdot),\infty}$.
Then $\mathbb{W} := [w(\cdot)]^{p(\cdot)} \in A_\infty$ and there exist positive constants $C$ and $A$,
depending only on $p(\cdot)$ and $n$, such that,
for any cube $Q$ in $\mathbb{R}^n$ and for any measurable set $ E\subset Q$
with $|E| > (1-\alpha)|Q|$,
one has
\begin{align*}
\mathbb{W}(Q) < C \left( 1 - 2\alpha^{1-\frac1\delta} \right)^{-p_+} [w]_{\mathcal{A}_{p(\cdot),\infty}}^{A}
L_w\mathbb{W}(E),
\end{align*}
where $\delta := 1 + \frac{1}{\tau_n C_{p(\cdot),n}[w]_{\mathcal{A}_{p(\cdot),\infty}}}$
with the constant $\tau_n$ depending only on $n$,
$\alpha \in (0,2^{-\frac{\delta}{\delta-1}})$ is any given constant,
and
\begin{align}\label{def Lw}
L_w :=  \max\left\{ \left\|w\mathbf{1}_{Q_0}\right\|^{p_--p_+}_{L^{p(\cdot)}}, \left\|w\mathbf{1}_{Q_0}\right\|^{p_+ - p_-}_{L^{p(\cdot)}} \right\}.
\end{align}
\end{proposition}
Before giving the proof of Proposition \ref{reverse Holder w1},
we first recall some necessary tools.
The following lemma is exactly \cite[Corollary 2.23]{cf13}.
\begin{lemma}\label{rhof 1}
Let $p(\cdot) \in \mathcal{P}$ with $p_+ < \infty$.
For any $f\in \mathscr{M}$,
if $\|f\|_{L^{p(\cdot)}} \in (1,\infty)$,
then
$$\left[\rho_{p(\cdot)}(f)\right]^\frac{1}{p_+} \leq \|f\|_{L^{p(\cdot)}} \leq \left[\rho_{p(\cdot)}(f)\right]^\frac{1}{p_-};$$
conversely, if $\|f\|_{L^{p(\cdot)}} \in [0,1]$,
then
$$\left[\rho_{p(\cdot)}(f)\right]^\frac{1}{p_-} \leq \|f\|_{L^{p(\cdot)}} \leq \left[\rho_{p(\cdot)}(f)\right]^\frac{1}{p_+}.$$
\end{lemma}

The following lemma is precisely \cite[Lemma 3.24]{cf13}
and \cite[Lemma 4.1.6 and Corollary 4.5.9]{dhr17};
in particular,  Lemma \ref{weight bound}{\rm (i)} is  called the \emph{Diening condition}
(see \cite[Lemma 2.7]{cp24}).
\begin{lemma}\label{weight bound}
Let $p(\cdot) \in \mathcal{P}$ with $p(\cdot) \in LH$.
Then the following statements hold:
\begin{itemize}
\item[{\rm (i)}] There exists a positive constant $C_D$, depending only on $p(\cdot)$ and $n$,
such that, for any cube $Q$ in $\mathbb{R}^n$,
$|Q|^{p_-(Q)-p_+(Q)} \leq C_D$.
Moreover, if $|Q|\leq 1$,
then, for any $x \in Q$,
$$ C_D^{-\frac{1}{p_-}}|Q|^{-1} \leq |Q|^{-\frac{p(x)}{p_Q}} \leq C_D^{\frac{1}{p_-}}|Q|^{-1}. $$
\item[{\rm (ii)}] For any cube $Q$ with $|Q|\geq 1$,
$|Q|^{\frac{1}{p_\infty}} \sim \|\mathbf{1}_Q\|_{L^{p(\cdot)}},$
where the positive equivalence constants depend only on $p(\cdot)$ and $n$.
\end{itemize}
\end{lemma}
\begin{remark}
Indeed, in Lemma \ref{weight bound}{\rm (i)},
for any $p(\cdot) \in \mathcal{P}$ with $p(\cdot) \in LH$,
we can take
$$C_D := \max\left\{ \left(2\sqrt n\right)^{n(p_+ - p_-)},\exp\left(C_0\left[1+\log_2\sqrt n\right]\right) \right\}, $$
where $C_0$ is the same as in \eqref{clogp 1}
and, moreover, it is obvious that, when $p(\cdot)$ is a constant exponent,
we have $C_D = 1$.
\end{remark}

The following lemma shows $\mathcal{A}_{p(\cdot),\infty} \subset A_\infty$.
\begin{lemma}\label{Ap Apinfty 1}
If $p(\cdot) \in \mathcal{P}$ with $p(\cdot) \in LH$,
then, for any $w \in \mathcal{A}_{p(\cdot),\infty}$,
$1 \leq [w]_{A_\infty} \leq C_{p(\cdot),n} [w]_{\mathcal{A}_{p(\cdot),\infty}}$,
where $C_{p(\cdot),n}$ is as in Lemma \ref{est fQ}.
\end{lemma}
\begin{proof}
For any $w \in \mathcal{A}_{p(\cdot),\infty}$,
using Lemma \ref{est fQ},
we obtain, for any cube $Q$ in $\mathbb{R}^n$,
\begin{align*}
\fint_Q w(x)\,dx \exp\left( \fint_Q \log\left(w^{-1}(x)\right)\,dx \right)
\leq C_{p(\cdot),n} \frac{1}{\|\mathbf{1}_Q\|_{L^{p(\cdot)}}} \left\|w\mathbf{1}_Q\right\|_{L^{p(\cdot)}} \exp\left( \fint_Q \log\left(w^{-1}(x)\right)\,dx \right),
\end{align*}
which, combined with the definition of $\mathcal{A}_{p(\cdot),\infty}$ and $A_\infty$,
further implies that
$$[w]_{A_\infty} \leq C_{p(\cdot),n}[w]_{\mathcal{A}_{p(\cdot),\infty}}.$$
Notice that, by \cite[Proposition 7.3.2(4)]{g14},
$[w]_{A_\infty} \geq 1$.
This finishes the proof of Lemma \ref{Ap Apinfty 1}.
\end{proof}

The following doubling property in \eqref{ww lambdaQ}
of $A_\infty$ weights  is  well known
(see, for instance, \cite[Exercise 7.3.1{\rm (ii)}]{g14}).
In what follows, for any cube $Q$ in $\mathbb{R}^n$,
we use $l(Q)$ to denote its \emph{edge length}.
\begin{lemma}\label{lambdaQ Q}
Let $w\in A_\infty$.
Then, for any cube $Q$ in $\mathbb{R}^n$ and any $\lambda \in (1,\infty)$,
\begin{align}\label{ww lambdaQ}
w(\lambda Q) \leq \left( 2\lambda \right)^{2^n(1+\log_2 [w]_{A_\infty})} w(Q).
\end{align}
Moreover, for any cubes $Q$ and $R$ in $\mathbb{R}^n$ with $Q \subset R$,
\begin{align}\label{ww RQ}
\left[ \frac{l(Q)}{l(R)} \right]^{2^n(1 + \log [w]_{A_\infty})}
\leq 2^{2^{n + 1}} [w]_{A_\infty}^{2^{n + 1}}
\frac{w(Q)}{w(R)}.
\end{align}
\end{lemma}
Using Lemmas \ref{Ap Apinfty 1} and \ref{lambdaQ Q}
and also \eqref{wM 2},
we obtain the following result.
\begin{lemma}\label{EB 1}
Let $p(\cdot) \in \mathcal{P}$ with $p(\cdot) \in LH$.
Then, for any $w\in \mathcal{A}_{p(\cdot),\infty}$, any cube $Q$ in $\mathbb{R}^n$,
and any $\lambda \in (1,\infty)$,
\begin{align}\label{EB eq 4}
\lambda^{-2^n[1 + \log_2 (C_{p(\cdot),n}[w]_{\mathcal{A}_{p(\cdot),\infty}})]}
\leq 2^{2^n} C^{2^n +1}_{p(\cdot)} [w]^{2^n + 1}_{\mathcal{A}_{p(\cdot),\infty}} \frac{\|w\mathbf{1}_Q\|_{L^{p(\cdot)}}}{\|w\mathbf{1}_{\lambda Q}\|_{L^{p(\cdot)}}};
\end{align}
moreover, for any cube $Q$ in $\mathbb{R}^n$ and any cube $R\subset Q$,
\begin{align*}
\left[ \frac{l(Q)}{l(R)} \right]^{2^n[1 + \log_2 (C_{p(\cdot),n}[w]_{\mathcal{A}_{p(\cdot),\infty}})]}
\leq 2^{2^n + 1} C^{2^{n+1} + 2}_{p(\cdot)} [w]^{2^{n+1} + 2}_{\mathcal{A}_{p(\cdot),\infty}}
\frac{\|w\mathbf{1}_Q\|_{L^{p(\cdot)}}}{\|w\mathbf{1}_R\|_{L^{p(\cdot)}}}.
\end{align*}
\end{lemma}

\begin{proof}
By Lemmas \ref{Holder} and \ref{est Q}
and the obvious fact that $Q\subset \lambda Q$,
we have
\begin{align*}
w(Q) \leq C_{p(\cdot)} \left\| w\mathbf{1}_Q \right\|_{L^{p(\cdot)}} \left\| \mathbf{1}_{\lambda Q} \right\|_{L^{p'(\cdot)}}
\leq C_{p(\cdot),n} \left\| w\mathbf{1}_Q \right\|_{L^{p(\cdot)}} \frac{|\lambda Q|}{\| \mathbf{1}_{\lambda Q} \|_{L^{p(\cdot)}}},
\end{align*}
which, with dividing $w(\lambda Q)$ on the both sides,
further implies that
\begin{align}\label{EB eq 2}
\frac{w(Q)}{w(\lambda Q)}
\leq C_{p(\cdot),n} \frac{\| w\mathbf{1}_Q \|_{L^{p(\cdot)}}}{w(\lambda Q)} \frac{|\lambda Q|}{\| \mathbf{1}_{\lambda Q} \|_{L^{p(\cdot)}}}
= C_{p(\cdot),n} \frac{\|w\mathbf{1}_{Q}\|_{L^{p(\cdot)}}}{\|w\mathbf{1}_{\lambda Q}\|_{L^{p(\cdot)}}}
\frac{|\lambda Q|}{w(\lambda Q)} \frac{\|w\mathbf{1}_{\lambda Q}\|_{L^{p(\cdot)}} }{\|\mathbf{1}_{\lambda Q}\|_{L^{p(\cdot)}}}.
\end{align}
From \eqref{wM 2}, we infer that
$\frac{1}{\|\mathbf{1}_{\lambda Q}\|_{L^{p(\cdot)}}}\|w\mathbf{1}_{\lambda Q}\|_{L^{p(\cdot)}}
\leq [w]_{\mathcal{A}_{p(\cdot),\infty}} \fint_{\lambda Q} w(x)\,dx, $
which, together with \eqref{EB eq 2}, implies that
\begin{align}\label{EB eq 5}
\frac{w(Q)}{w(\lambda Q)} \leq C_{p(\cdot),n} [w]_{\mathcal{A}_{p(\cdot),\infty}} \frac{\|w\mathbf{1}_Q\|_{L^{p(\cdot)}}}{\|w\mathbf{1}_{\lambda Q}\|_{L^{p(\cdot)}}}.
\end{align}
Combining this with Lemmas \ref{lambdaQ Q} and \ref{Ap Apinfty 1},
we conclude that
\begin{align*}
\left(2\lambda\right)^{-2^n(1+\log_2 (C_{p(\cdot),n}[w]_{\mathcal{A}_{p(\cdot),\infty}}))}
\leq \left(2\lambda\right)^{-2^n(1+\log_2 [w]_{A_\infty})}
\leq \frac{w(Q)}{w(\lambda Q)}
\leq C_{p(\cdot),n} [w]_{\mathcal{A}_{p(\cdot),\infty}} \frac{\|w\mathbf{1}_Q\|_{L^{p(\cdot)}}}{\|w\mathbf{1}_{\lambda Q}\|_{L^{p(\cdot)}}},
\end{align*}
which further implies that
\begin{align*}
\lambda^{-2^n(1+\log_2 (C_{p(\cdot),n}[w]_{\mathcal{A}_{p(\cdot),\infty}}))}
\leq 2^{2^n} C^{2^n +1}_{p(\cdot)} [w]^{2^n + 1}_{\mathcal{A}_{p(\cdot),\infty}} \frac{\|w\mathbf{1}_Q\|_{L^{p(\cdot)}}}{\|w\mathbf{1}_{\lambda Q}\|_{L^{p(\cdot)}}}.
\end{align*}
This finishes the proof of \eqref{EB eq 4}.

Moreover, for any cubes $Q$ and $R$ with $Q\subset R$,
similar to the previous proof about \eqref{EB eq 4},
by Lemma \ref{Ap Apinfty 1}, \eqref{ww RQ},
and \eqref{EB eq 5} with $\lambda Q$ replaced by $R$,
we obtain
\begin{align*}
\left[ \frac{l(Q)}{l(R)} \right]^{2^n[1 + \log_2 (C_{p(\cdot),n}[w]_{\mathcal{A}_{p(\cdot),\infty}})]}
\leq 2^{2^n + 1} C^{2^{n+1} + 2}_{p(\cdot)} [w]^{2^{n+1} + 2}_{\mathcal{A}_{p(\cdot),\infty}}
\frac{\|w\mathbf{1}_Q\|_{L^{p(\cdot)}}}{\|w\mathbf{1}_R\|_{L^{p(\cdot)}}},
\end{align*}
which completes the proof of Lemma \ref{EB 1}.
\end{proof}

The following lemma is an essential estimate about $\|w\mathbf{1}_Q\|_{L^{p(\cdot)}}^{p_-(Q) - p_+(Q)}$
(see, for instance, \cite[Lemma 3.3]{cfn12} for a similar result about $\mathcal{A}_{p(\cdot)}$ weights).
\begin{lemma}\label{Bp bound}
Let $p(\cdot)\in \mathcal{P}$ with $p(\cdot) \in LH$ and let $w\in \mathcal{A}_{p(\cdot),\infty}$.
Then there exist positive constants $C$ and $A$,
depending only on $p(\cdot)$ and $n$, such that,
for any cube $Q$ in $\mathbb{R}^n$,
\begin{align}\label{EB eq}
\left\|w\mathbf{1}_Q\right\|_{L^{p(\cdot)}}^{p_-(Q) - p_+(Q)} \leq C L_w [w]^{A}_{\mathcal{A}_{p(\cdot),\infty}},
\end{align}
where $L_w$ is the same as in \eqref{def Lw}.
\end{lemma}
\begin{proof}
From the fact that $p_-(Q) - p_+(Q) \leq 0$,
we deduce that,
if $\|w\mathbf{1}_Q\|_{L^{p(\cdot)}} \geq 1$,
then it is obvious that
$\|w\mathbf{1}_Q\|_{L^{p(\cdot)}}^{p_-(Q) - p_+(Q)} \leq 1. $
Hence, we only need to consider the case $\|w\mathbf{1}_Q\|_{L^{p(\cdot)}} < 1$.

Now, fix a cube $Q := Q(y_0,l)$ with $y_0 \in \mathbb{R}^n$ and $l \in (0,\infty)$
and let $Q_0$ be the same as in \eqref{def Q0}.
To make the proof clear,
we begin with the assumption $\| w\mathbf{1}_{Q_0} \|_{L^{p(\cdot)}} = 1$.
During the following discussions,
we divide the proof into the following four cases:
$l\leq 2e$ and ${\mathop\mathrm{\,dist\,}}(Q,Q_0) \leq 2e$,
$l > 2e$ and ${\mathop\mathrm{\,dist\,}}(Q,Q_0) \leq l$,
$l\leq 2e$ and ${\mathop\mathrm{\,dist\,}}(Q,Q_0) > 2e$,
or $l > 2e$ and ${\mathop\mathrm{\,dist\,}}(Q,Q_0) > l$,
where ${\mathop\mathrm{\,dist\,}}(Q,Q_0)$ is the same as in \eqref{def d}.

We first consider the case where $l\leq 2e$ and ${\mathop\mathrm{\,dist\,}}(Q,Q_0) \leq 2e$.
In this case, let $\widetilde{Q} := 5Q_0$ and hence
it is obvious that $Q\subset \widetilde{Q}$.
Thus, by this and Lemma \ref{EB 1} with $Q := Q$ and $R := \widetilde{Q}$,
we obtain
\begin{align*}
\frac{\|w\mathbf{1}_{\widetilde{Q}}\|_{L^{p(\cdot)}}}{\|w\mathbf{1}_Q\|_{L^{p(\cdot)}}}
\lesssim \left\{C_{p(\cdot),n}[w]_{\mathcal{A}_{p(\cdot),\infty}}\right\}^{2^{n+1} + 2}
\left[ \frac{l(\widetilde{Q})}{l(Q)} \right]^{2^n[1 + \log_2 (C_{p(\cdot),n}[w]_{\mathcal{A}_{p(\cdot),\infty}})]},
\end{align*}
where $C_{p(\cdot),n}$ is the same as in Lemma \ref{est fQ},
which further implies that
\begin{align}\label{EB eq 7}
\left\{C_{p(\cdot),n}[w]_{\mathcal{A}_{p(\cdot),\infty}}\right\}^{-2^{n+1} - 2}
\left[\frac{l}{10e}\right]^{2^n[1 + \log_2 (C_{p(\cdot),n}[w]_{\mathcal{A}_{p(\cdot),\infty}})]}
&\lesssim \left\|w\mathbf{1}_Q\right\|_{L^{p(\cdot)}}.
\end{align}
Notice that, by Lemma \ref{Ap Apinfty 1},
we have
\begin{align}\label{EB eq 8}
1 \leq C_{p(\cdot),n}[w]_{\mathcal{A}_{p(\cdot),\infty}}.
\end{align}
Thus, using this and \eqref{EB eq 7} with raising both sides to the power $p_-(Q) - p_+(Q)$
and using Lemmas \ref{weight bound} {\rm (i)},
we find that
\begin{align*}
&\left\|w\mathbf{1}_Q\right\|_{L^{p(\cdot)}}^{p_-(Q)-p_+(Q)}\\
&\quad \lesssim \left\{C_{p(\cdot),n}[w]_{\mathcal{A}_{p(\cdot),\infty}}\right\}^{(2^{n+1} + 2)(p_+(Q) - p_-(Q))}
\left(\frac{l}{10e}\right)^{2^n(p_+(Q) - p_-(Q))\{1 + \log_2 (C_{p(\cdot),n}[w]_{\mathcal{A}_{p(\cdot),\infty}})\}}\\
&\quad \leq \left\{C_{p(\cdot),n}[w]_{\mathcal{A}_{p(\cdot),\infty}}\right\}^{(2^{n+1} + 2)(p_+ - p_-)}
(10e)^{2^n(p_+ - p_-)\{1 + \log_2 (C_{p(\cdot),n}[w]_{\mathcal{A}_{p(\cdot),\infty}})\}}
C_D^{\frac{2^n}{n}\{1 + \log_2 (C_{p(\cdot),n}[w]_{\mathcal{A}_{p(\cdot),\infty}})\}}\\
&\quad \sim [w]_{\mathcal{A}_{p(\cdot),\infty}}^{(p_+ - p_-)(2^{n+1} + 2 + 2^n\log_2 (10e)) + \frac{2^n}{n}\log_2(C_D)},
\end{align*}
where $C_D$ is the same as in Lemma \ref{weight bound}.
This finishes the proof of the case where $l \leq 2e$ and ${\mathop\mathrm{\,dist\,}}(Q,Q_0) \leq 2e$.

Next, we consider the case where $l > 2e$ and ${\mathop\mathrm{\,dist\,}}(Q,Q_0) \leq l$.
Let $\widetilde{Q} := 5Q$.
Then $Q_0 \subset \widetilde{Q}$.
From this and Lemma \ref{EB 1} with $Q := Q$ and $R := \widetilde{Q}$,
we infer that
\begin{align*}
\frac{\|w\mathbf{1}_{\widetilde{Q}}\|_{L^{p(\cdot)}}}{\|w\mathbf{1}_Q\|_{L^{p(\cdot)}}}
&\lesssim \left\{C_{p(\cdot),n}[w]_{\mathcal{A}_{p(\cdot),\infty}}\right\}^{2^{n+1} + 2}
\left[ \frac{l(\widetilde{Q})}{l(Q)} \right]^{2^n[1 + \log_2 (C_{p(\cdot),n}[w]_{\mathcal{A}_{p(\cdot),\infty}})]}\\
&\sim \left\{C_{p(\cdot),n}[w]_{\mathcal{A}_{p(\cdot),\infty}}\right\}^{2^{n+1} + 2 + 2^n\log_2 5},
\end{align*}
which further implies that
$$\left\{C_{p(\cdot),n}[w]_{\mathcal{A}_{p(\cdot),\infty}}\right\}^{-2^{n+1} - 2 - 2^n\log_2 5} \left\|w\mathbf{1}_{\widetilde{Q}}\right\|_{L^{p(\cdot)}}
\lesssim \left\|w\mathbf{1}_Q\right\|_{L^{p(\cdot)}}.$$
By this with raising both sides to the power $p_-(Q)-p_+(Q)$ and \eqref{EB eq 8},
we conclude that
\begin{align*}
\left\|w\mathbf{1}_Q\right\|_{L^{p(\cdot)}}^{p_-(Q)-p_+(Q)}
&\lesssim \left\{C_{p(\cdot),n}[w]_{\mathcal{A}_{p(\cdot),\infty}}\right\}^{ (2^{n+1} + 2 + 2^n \log_25)(p_+(Q)-p_-(Q))}
\left\|w\mathbf{1}_{\widetilde{Q}}\right\|_{L^{p(\cdot)}}^{p_-(Q) - p_+(Q)}\\
&\leq \left\{C_{p(\cdot),n}[w]_{\mathcal{A}_{p(\cdot),\infty}}\right\}^{ (2^{n+1} + 2 + 2^n \log_25)(p_+-p_-)}
\left\|w\mathbf{1}_{Q_0}\right\|_{L^{p(\cdot)}}^{p_-(Q) - p_+(Q)}\\
&\sim [w]_{\mathcal{A}_{p(\cdot),\infty}}^{ (2^{n+1} + 2 + 2^n \log_25)(p_+-p_-)},
\end{align*}
which completes the proof of the case where $l > 2e$ and ${\mathop\mathrm{\,dist\,}}(Q,Q_0) \leq l$.

Finally, we consider the last two cases:
$l\leq 2e$ and ${\mathop\mathrm{\,dist\,}}(Q,Q_0) > 2e$ or
$l > 2e$ and
$${\mathop\mathrm{\,dist\,}}(Q,Q_0) > l.$$
Let $d := {\mathop\mathrm{\,dist\,}}(Q,Q_0)$ and $\widetilde{Q} := Q(\mathbf{0}, 4d)$.
It is easy to find that $Q \subset \widetilde{Q}$.
By this and Lemma \ref{EB 1} with $Q := Q$ and $R := \widetilde{Q}$,
we find that
\begin{align*}
\frac{\|w\mathbf{1}_{\widetilde{Q}}\|_{L^{p(\cdot)}}}{\|w\mathbf{1}_Q\|_{L^{p(\cdot)}}}
&\lesssim  \left\{C_{p(\cdot),n}[w]_{\mathcal{A}_{p(\cdot),\infty}}\right\}^{2^{n+1} + 2}
\left[ \frac{4d}{l} \right]^{2^n[1 + \log_2 (C_{p(\cdot),n}[w]_{\mathcal{A}_{p(\cdot),\infty}})]},
\end{align*}
which further implies that
\begin{align*}
\left\{C_{p(\cdot),n}[w]_{\mathcal{A}_{p(\cdot),\infty}}\right\}^{-(2^{n+1} + 2)}
\left[ \frac{l}{4d} \right]^{2^n[1 + \log_2 (C_{p(\cdot),n}[w]_{\mathcal{A}_{p(\cdot),\infty}})]}
\left\|w\mathbf{1}_{\widetilde{Q}}\right\|_{L^{p(\cdot)}}
\lesssim \left\|w\mathbf{1}_Q\right\|_{L^{p(\cdot)}}.
\end{align*}
Hence, using this with raising both sides to the power $p_-(Q)-p_+(Q)$
and using Lemma \ref{weight bound}(i), \eqref{EB eq 8},
and the assumption $\|w\mathbf{1}_{Q_0}\|_{L^{p(\cdot)}} = 1$,
we conclude that
\begin{align}\label{EB eq 3}
\left\|w\mathbf{1}_Q\right\|_{L^{p(\cdot)}}^{p_-(Q)-p_+(Q)}
&\lesssim \left\|w\mathbf{1}_{\widetilde{Q}}\right\|_{L^{p(\cdot)}}^{p_-(Q) - p_+(Q)}\left\{C_{p(\cdot),n}[w]_{\mathcal{A}_{p(\cdot),\infty}}\right\}^{[p_+(Q)-p_-(Q)](2^{n+1} + 2)}\nonumber\\
&\quad \times \left(\frac{4l}{d}\right)^{2^n[p_-(Q) - p_+(Q)][1 + \log_2(C_{p(\cdot),n} [w]_{\mathcal{A}_{p(\cdot),\infty}})] } \nonumber\\
&\leq \left\|w\mathbf{1}_{Q_0}\right\|_{L^{p(\cdot)}}^{p_- - p_+} \left\{C_{p(\cdot),n}[w]_{\mathcal{A}_{p(\cdot),\infty}}\right\}^{(p_+-p_-)(2^{n+1} + 2)}
 C_D^{\frac{2^n}{n}[1 + \log_2(C_{p(\cdot),n} [w]_{\mathcal{A}_{p(\cdot),\infty}}) ]} \nonumber\\
&\quad \times 4^{2^n(p_+-p_-)[1 + \log_2(C_{p(\cdot),n} [w]_{\mathcal{A}_{p(\cdot),\infty}}) ]} d^{2^n[p_+(Q) - p_-(Q)][1 + \log_2 (C_{p(\cdot),n}[w]_{\mathcal{A}_{p(\cdot),\infty}})]} \nonumber\\
&\lesssim [w]^{(p_+-p_-)(2^{n+2} + 2) + \frac{1}{n}2^n\log_2 C_D}_{\mathcal{A}_{p(\cdot),\infty}} d^{2^n[p_+(Q) - p_-(Q)][1 + \log_2 (C_{p(\cdot),n}[w]_{\mathcal{A}_{p(\cdot),\infty}})]}.
\end{align}
Since $p(\cdot) \in LH$,
we deduce that $p(\cdot)$ is continuous and hence
there exist points $x,y \in \overline{Q}$
such that $p(x) = p_-(Q)$ and $p(y) = p_+(Q)$.
Let $d_Q := {\mathop\mathrm{\,dist\,}}(\mathbf{0},Q)$.
Then, by $\mathbf{0} \in Q_0$,
we conclude that $d < d_Q$ and $|x|, |y| \geq d_Q$.
Thus, using this and the assumption $p(\cdot) \in LH$,
we obtain
$$ p_+(Q)-p_-(Q) = \left| p_+(Q)-p_-(Q)  \right|
\leq \left| p(x) - p_\infty \right| + \left| p(y) - p_\infty \right|
\leq \frac{2C_\infty}{\log(e + d_Q)}, $$
where $C_\infty$ and $p_\infty$ are the same as in \eqref{clogp infty}.
From this and the fact $d < d_Q$,
we further infer that
\begin{align*}
d^{p_+(Q)-p_-(Q)}
\leq d_Q^{\frac{2C_\infty}{\log(e + d_Q)}}
\leq (e+d_Q)^{\frac{2C_\infty}{\log(e + d_Q)}}
 = e^{\frac{2C_\infty\log(e + d_Q)}{\log(e + d_Q)}} = e^{2C_\infty},
\end{align*}
which, combined with \eqref{EB eq 3},
further implies that
\begin{align*}
\left\|w\mathbf{1}_Q\right\|_{L^{p(\cdot)}}^{p_-(Q)-p_+(Q)}
&\lesssim  [w]^{(p_+-p_-)(2^{n+2} + 2) + \frac{1}{n}2^n\log_2 C_D}_{\mathcal{A}_{p(\cdot),\infty}}
e^{2C_\infty \log_2 [w]_{\mathcal{A}_{p(\cdot),\infty}}}\\
& = [w]^{(p_+-p_-)(2^{n+2} + 2) + \frac{1}{n}2^n\log_2 C_D + 2C_\infty \log_2 e}_{\mathcal{A}_{p(\cdot),\infty}}.
\end{align*}
This finishes the proof of \eqref{EB eq} in the case $\|w\mathbf{1}_{Q_0}\|_{L^{p(\cdot)}} = 1$,
that is, there exists a positive constant $A$, depending only on $n$ and $p(\cdot)$,
such that, for any cube $Q$ in $\mathbb{R}^n$,
\begin{align}\label{EB eq 6}
\left\|w\mathbf{1}_Q\right\|_{L^{p(\cdot)}} \lesssim [w]_{\mathcal{A}_{p(\cdot),\infty}}^A.
\end{align}

In the final step, we consider the case $\|w\mathbf{1}_{Q_0}\|_{L^{p(\cdot)}} \neq 1$.
Since $w \in \mathcal{A}_{p(\cdot),\infty}$, from Lemma \ref{Ap Apinfty 1},
it follows that $w \in A_\infty$ and
$0 < w(Q_0) \lesssim \|w\mathbf{1}_{Q_0}\|_{L^{p(\cdot)}} \|\mathbf{1}_{Q_0}\|_{L^{p,(\cdot)}}$,
which further implies that $\|w\mathbf{1}_{Q_0}\|_{L^{p(\cdot)}} > 0$.
Using this and letting $v(x) := \frac{w(x)}{\|w\mathbf{1}_{Q_0}\|_{L^{p(\cdot)}}}$,
we easily find that $\|v\mathbf{1}_{Q_0}\|_{L^{p(\cdot)}} = 1$.
Notice that, for any cube $Q$ in $\mathbb{R}^n$,
$$ \frac{1}{\|\mathbf{1}_{Q}\|_{L^{p(\cdot)}}} \left\|v\mathbf{1}_Q\right\|_{L^{p(\cdot)}} \exp\left( \fint_Q \log v^{-1}(x)\,dx \right)
= \frac{1}{\|\mathbf{1}_{Q}\|_{L^{p(\cdot)}}} \left\|w\mathbf{1}_Q\right\|_{L^{p(\cdot)}} \exp\left( \fint_Q \log w^{-1}(x)\,dx \right), $$
which, together with the definition of $\mathcal{A}_{p(\cdot),\infty}$,
further implies that
\begin{align}\label{EB eq 13}
 [v]_{\mathcal{A}_{p(\cdot),\infty}} = [w]_{\mathcal{A}_{p(\cdot),\infty}} .
\end{align}
By this and the just proven conclusion that \eqref{EB eq 6} holds with $w$ replaced by $v$,
we conclude that, for any cube $Q$ in $\mathbb{R}^n$,
\begin{align*}
\left\| w \mathbf{1}_Q \right\|_{L^{p(\cdot)}}^{p_-(Q) - p_+(Q)}
&= \left\| w \mathbf{1}_{Q_0} \right\|_{L^{p(\cdot)}}^{p_+(Q) - p_-(Q)}
\left\| v \mathbf{1}_{Q} \right\|_{L^{p(\cdot)}}^{p_-(Q) - p_+(Q)} \\
& \lesssim \max\left\{ \left\| w \mathbf{1}_{Q_0} \right\|_{L^{p(\cdot)}}^{p_- - p_+}, \left\| w \mathbf{1}_{Q_0} \right\|_{L^{p(\cdot)}}^{p_+ - p_-} \right\} [v]_{\mathcal{A}_{p(\cdot),\infty}}^{A}
 = L_w [w]_{\mathcal{A}_{p(\cdot),\infty}}^{A}.
\end{align*}
This finishes the proof of Lemma \ref{Bp bound}.
\end{proof}
The following lemma is exactly \cite[Proposition 2.21]{cf13}.
\begin{lemma}\label{rho f norm}
Let $p(\cdot) \in \mathcal{P}$.
If $f\in L^{p(\cdot)}$ and $\|f\|_{L^{p(\cdot)}} \neq 0$,
then $\rho_{p(\cdot)}(\frac{f}{\|f\|_{L^{p(\cdot)}}}) = 1$.
\end{lemma}
The following lemma is precisely \cite[Lemma 2.10]{cc22}
with $X = \mathbb{R}^n$.
\begin{lemma}\label{rs f}
Let $G$ be a measurable set in $\mathbb{R}^n$
and $r(\cdot), s(\cdot) \in \mathcal{P}_0$ satisfy, for any $y\in G$,
$$ \left| r(y) - s(y) \right| \leq \frac{C_0}{\log(e + |y|)}, $$
where $C_0$ is a positive constant.
Then, for any $t\in (0,\infty)$, any measure $\mu$,
and any $f\in \mathscr{M}$ with $|f| \leq 1$ on $G$,
\begin{align*}
\int_{G} \left|f(y)\right|^{s(y)}\,d\mu(y) \leq e^{ntC_0} \int_{G}\left| f(y) \right|^{r(y)}\,d\mu(y)
+ \int_G \frac{1}{(e + |y|)^{nts_-(G)}}\,d\mu(y).
\end{align*}
\end{lemma}
The following lemma is exactly \cite[Lemma 2.12]{cp24},
which is a consequence  of Lemma \ref{rs f}.

\begin{lemma}\label{rs f 1}
Let $p(\cdot) \in \mathcal{P}_0$ with $p(\cdot) \in LH$.
Then, for any measure $\mu$, any measurable set $E$ with $\mu(E)<\infty$,
any $t\in (0,\infty)$,
and any $f\in \mathscr{M}$ with $|f| \leq 1$ on $E$,
\begin{align*}
\int_{E} \left|f(y)\right|^{p(y)}\,d\mu(y) \leq e^{ntC_\infty} \int_{E}\left| f(y) \right|^{p_\infty}\,d\mu(y)
+ \int_E \frac{1}{(e + |y|)^{ntp_-}}\,d\mu(y)
\end{align*}
and
\begin{align*}
\int_{E} \left|f(y)\right|^{p_\infty}\,d\mu(y) \leq e^{ntC_\infty} \int_{E}\left| f(y) \right|^{p(y)}\,d\mu(y)
+ \int_E \frac{1}{(e + |y|)^{ntp_-}}\,d\mu(y),
\end{align*}
where $C_\infty$ is the same as in \eqref{clogp infty}.
\end{lemma}

Based on the above lemmas, we have the following estimate.

\begin{lemma}\label{EB 2}
Let $p(\cdot) \in \mathcal{P}$ with $p(\cdot) \in LH$
and $w\in\mathcal{A}_{p(\cdot),\infty}$.
If $\|w\mathbf{1}_{Q_0}\|_{L^{p(\cdot)}} = 1$,
where $Q_0$ is the same as in \eqref{def Q0},
then there exists a positive constant $A$, depending only on $p(\cdot)$ and $n$,
such that, for any measurable set $E$ in $\mathbb{R}^n$ with $\|w\mathbf{1}_E\|_{L^{p(\cdot)}} \in [1,\infty)$,
\begin{align}\label{WWinfty}
C^{-1} [w]_{\mathcal{A}_{p(\cdot),\infty}}^{-A} \mathbb{W}(E)^{\frac{1}{p_\infty}}
\leq \|w\mathbf{1}_E\|_{L^{p(\cdot)}}
\leq C [w]_{\mathcal{A}_{p(\cdot),\infty}}^{A} \mathbb{W}(E)^{\frac{1}{p_\infty}},
\end{align}
where $\mathbb{W}(\cdot):=[w(\cdot)]^{p(\cdot)}$.
\end{lemma}

\begin{proof}
We begin with the proof of the first inequality of \eqref{WWinfty}.
In what follows, for simplicity of presentation,
we let $C_w := 2^{2^n} C^{2^n +1}_{p(\cdot)} [w]^{2^n + 1}_{\mathcal{A}_{p(\cdot),\infty}}$
and $r_w := \frac{2^n}{n}[1 + \log_2 (C_{p(\cdot),n}[w]_{\mathcal{A}_{p(\cdot),\infty}})]$.
By Lemma \ref{Ap Apinfty 1}, we find that $C_w \geq 1$.
From Lemma \ref{EB 1},
it follows that, for any cube $Q$ in $\mathbb{R}^n$
and for any $\lambda \in (1,\infty)$,
\begin{align}\label{w doubling 1}
\left\| w\mathbf{1}_{\lambda Q} \right\|_{L^{p(\cdot)}} \leq C_w \lambda^{r_w} \left\| w\mathbf{1}_{Q} \right\|_{L^{p(\cdot)}}.
\end{align}
By Lemma \ref{rs f 1} with $ f := \|w\mathbf{1}_E\|^{-1}_{L^{p(\cdot)}} $ and $\mu := \mathbb{W}$,
we easily find that
\begin{align}\label{WWinfty eq 1}
\left\|w\mathbf{1}_E\right\|^{-p_\infty}_{L^{p(\cdot)}} \mathbb{W}(E) =
\int_E \frac{\mathbb{W}(x)}{\|w\mathbf{1}_E\|^{p_\infty}_{L^{p(\cdot)}}} \,dx
\leq e^{ntC_\infty} \int_E \frac{\mathbb{W}(x)}{\|w\mathbf{1}_E\|^{p(x)}_{L^{p(\cdot)}}} \,dx
+ \int_E \frac{\mathbb{W}(x)}{(e + |x|)^{ntp_-}} \,dx.
\end{align}
Notice that, from Lemma \ref{rho f norm},
it follows immediately that
\begin{align}\label{WWinfty eq 4}
1 = \rho_{L^{p(\cdot)}}\left( \frac{w\mathbf{1}_E}{\|w\mathbf{1}_E\|_{L^{p(\cdot)}}} \right)
= \int_E \frac{\mathbb{W}(x)}{\|w\mathbf{1}_E\|^{p(x)}_{L^{p(\cdot)}}} \,dx.
\end{align}
Now, let $Q_0$ be the same as in \eqref{def Q0} and
$Q_k := Q(\mathbf{0},2e^{k+1})$ for any $k\in\mathbb{N}$.
Then, for any $k \in \mathbb{N}$ and any $x\in Q_{k} \setminus Q_{k-1}$,
we have $|x| \leq e^k$ and hence
\begin{align}\label{WWinfty eq 11}
(e + |x|)^{-ntp_-} \leq e^{-kntp_-},\ \ \forall\, t\in(0,\infty).
\end{align}
By the assumption $\|w\mathbf{1}_{Q_0}\|_{L^{p(\cdot)}} = 1$
and the fact that $ Q_0\subset Q_k $ for any $k\in\mathbb{N}$,
we conclude that $ \|w\mathbf{1}_{Q_k}\|_{L^{p(\cdot)}}\geq 1 $
for any $k\in\mathbb{Z}_+$.
Using this, \eqref{WWinfty eq 11}, and Lemma \ref{rhof 1},
we obtain
\begin{align}\label{WWinfty eq 2}
\int_E \frac{\mathbb{W}(x)}{(e + |x|)^{ntp_-}} \,dx
&\leq \int_{Q_0} \frac{\mathbb{W}(x)}{(e + |x|)^{ntp_-}} \,dx + \sum_{k = 1}^\infty \int_{Q_k\setminus Q_{k-1}} \frac{\mathbb{W}(x)}{(e + |x|)^{ntp_-}}\,dx\nonumber\\
&\leq e^{-ntp_-}\mathbb{W}(Q_0) + \sum_{k = 1}^\infty e^{-kntp_-} \mathbb{W}(Q_k)\nonumber\\
&\leq e^{-ntp_-}\left\|w\mathbf{1}_{Q_0}\right\|^{p_+}_{L^{p(\cdot)}}
+ \sum_{k = 1}^\infty e^{-kntp_-} \left\|w\mathbf{1}_{Q_k}\right\|^{p_+}_{L^{p(\cdot)}},
\end{align}
where $t \in (0,\infty)$ is a constant which will be determined later.
Then, by \eqref{w doubling 1} and the assumption $\|w\mathbf{1}_{Q_0}\|_{L^{p(\cdot)}} = 1$,
we find that, for any $k\in\mathbb{Z}_+$,
\begin{align*}
\left\|w\mathbf{1}_{Q_k}\right\|_{L^{p(\cdot)}} \leq C_w e^{knr_w} \left\|w\mathbf{1}_{Q_0}\right\|_{L^{p(\cdot)}} = C_w e^{nkr_w}.
\end{align*}
Thus, from this and \eqref{WWinfty eq 2},
we deduce that
\begin{align}\label{WWinfty eq 3}
\int_E \frac{\mathbb{W}(x)}{(e + |x|)^{ntp_-}} \,dx
\leq e^{-ntp_-} + C_w^{p_+} \sum_{k = 1}^\infty  e^{-k(ntp_- - nr_wp_+)},
\end{align}
and it is obvious that the summation above absolutely converges if $t > \frac{r_w p_+}{p_-}$.
Next, let
$$t_1 := \frac{r_w p_+}{p_-} + \frac{1}{np_-} \log\left( 2C_w^{p_+} \right).$$
Since $C_w \geq 1$,
it follows that $t_1 > \frac{r_w p_+}{p_-}$.
Hence, by the summation formula for geometric series with taking $t := t_1$,
we find that
\begin{align*}
C_w^{p_+} \sum_{k = 1}^\infty  e^{-k(nt_1p_- - nr_wp_+)}
= C_w^{p_+} \sum_{k = 1}^\infty e^{-k\log\left( 2C_w^{p_+} \right)}
= C_w^{p_+} \frac{e^{-\log\left( 2C_w^{p_+} \right)} }{ 1 - e^{-\log\left( 2C_w^{p_+} \right)} }
= \frac{C_w^{p_+}}{2C_w^{p_+} - 1} \leq 1.
\end{align*}
Now, combining both this and \eqref{WWinfty eq 3} with $t := t_1$,
we conclude that
\begin{align*}
\int_E \frac{\mathbb{W}(x)}{(e + |x|)^{nt_1p_-}} \,dx
\leq e^{-nt_1p_-} + 1 < 1 + 1 = 2.
\end{align*}
From this, \eqref{WWinfty eq 1}, and \eqref{WWinfty eq 4} with $t := t_1$,
it follows that
\begin{align*}
\left\|w\mathbf{1}_E\right\|^{-p_\infty}_{L^{p(\cdot)}} \mathbb{W}(E)
\leq e^{\frac{n r_w p_+ C_\infty}{p_-}}\left(2C_w^{p_+}\right)^\frac{C_\infty}{p_-} + 2
\leq 3 e^{\frac{n r_w p_+ C_\infty}{p_-}}\left(2C_w^{p_+}\right)^\frac{C_\infty}{p_-},
\end{align*}
where the final inequality comes from the fact
$ 1 < e^{\frac{n r_w p_+ C_\infty}{p_-}}\left(2C_w^{p_+}\right)^\frac{C_\infty}{p_-} $.
Hence, by this with rearranging the terms,
we conclude that
\begin{align*}
3^{-\frac{1}{p_\infty}}e^{-\frac{n r_w p_+ C_\infty}{p_-p_\infty}}\left(2C_w^{p_+}\right)^{-\frac{C_\infty}{p_-p_\infty}}\mathbb{W}(E)^\frac{1}{p_\infty}
\leq \left\|w\mathbf{1}_E\right\|_{L^{p(\cdot)}}.
\end{align*}
This finishes the estimation of the first part of \eqref{WWinfty}.

Next, we consider the second inequality of \eqref{WWinfty}.
From Lemmas \ref{rho f norm} and \ref{rs f 1}
with $ f := \|w\mathbf{1}_E\|^{-1}_{L^{p(\cdot)}} $ and $\mu := \mathbb{W}$,
it follows that
\begin{align}\label{WWinfty eq 5}
1 & = \rho_{L^{p(\cdot)}}\left( \frac{w\mathbf{1}_E}{\|w\mathbf{1}_E\|_{L^{p(\cdot)}}} \right)
= \int_E \frac{\mathbb{W}(x)}{\|w\mathbf{1}_E\|^{p(x)}_{L^{p(\cdot)}}} \,dx
\leq e^{ntC_\infty}  \int_E \frac{\mathbb{W}(x)}{\|w\mathbf{1}_E\|^{p_\infty}_{L^{p(\cdot)}}} \,dx
+ \int_E \frac{\mathbb{W}(x)}{(e + |x|)^{ntp_-}} \,dx\nonumber\\
&= e^{ntC_\infty} \left\|w\mathbf{1}_E\right\|^{-p_\infty}_{L^{p(\cdot)}} \mathbb{W}(E)
+ \int_E \frac{\mathbb{W}(x)}{(e + |x|)^{ntp_-}} \,dx,
\end{align}
where $t\in (0,\infty)$ is a constant which will be determined later.

Now, let
$$t_2 := \frac{r_w p_+}{p_-} + \frac{1}{np_-} \log\left( 5C_w^{p_+} \right). $$
Then, by \eqref{WWinfty eq 3} with taking $t := t_2$
and by the summation formula for geometric series
and the known fact $C_w \geq 1$,
we find that
\begin{align*}
\int_E \frac{\mathbb{W}(x)}{(e + |x|)^{nt_2p_-}} \,dx
& \leq e^{-nt_2p_-} + C_w^{p_+} \sum_{k = 1}^\infty  e^{-k(nt_2p_- - nr_wp_+)}\\
&= e^{-nr_wp_+}e^{-\log(5C_w^{p_+})} + C_w^{p_+} \sum_{k = 1}^\infty e^{-k\log\left( 5C_w^{p_+} \right)}\\
&\leq \frac{1}{5C_w^{p_+}} + C_w^{p_+} \frac{e^{-\log\left( 5C_w^{p_+} \right)} }{ 1 - e^{-\log\left( 5C_w^{p_+} \right)} }
\leq \frac{1}{5} + \frac{C_w^{p_+}}{5C_w^{p_+} - 1} < \frac15 + \frac14 < \frac12.
\end{align*}
Hence, from this and \eqref{WWinfty eq 5} with $t := t_2$,
we infer that
\begin{align*}
1 < e^{\frac{n r_w p_+ C_\infty}{p_-}}\left(5C_w^{p_+}\right)^\frac{C_\infty}{p_-} \left\|w\mathbf{1}_E\right\|^{-p_\infty}_{L^{p(\cdot)}} \mathbb{W}(E) + \frac12,
\end{align*}
which further implies that
\begin{align*}
\left\|w\mathbf{1}_E\right\|_{L^{p(\cdot)}}
\leq 2e^{\frac{n r_w p_+ C_\infty}{p_-p_\infty}}\left(5C_w^{p_+}\right)^{\frac{C_\infty}{p_-p_\infty}}
\mathbb{W}(E)^\frac{1}{p_\infty}.
\end{align*}
This finishes the estimation of the second part of \eqref{WWinfty},
which completes the proof of Lemma \ref{EB 2}.
\end{proof}

\begin{remark}\label{rem EB 2}
Notice that, in the proof of Lemma \ref{EB 2},
the condition $w \in \mathcal{A}_{p(\cdot),\infty}$ is only used to obtain \eqref{w doubling 1}.
Thus, by this observation, we conclude that, for any $p(\cdot)\in\mathcal{P}$ with $p(\cdot) \in LH$
and for any weight function $w$,
if $w$ satisfies \eqref{w doubling 1} and $\|w\mathbf{1}_{Q_0}\|_{L^{p(\cdot)}} = 1$,
then, for any measurable bounded set $E$ with $\|w\mathbf{1}_{E}\|_{L^{p(\cdot)}} \geq 1$,
we have $ \|w\mathbf{1}_{E}\|_{L^{p(\cdot)}} \sim \mathbb{W}(E)^{\frac{1}{p_\infty}} $,
where $\mathbb{W} := [w(\cdot)]^{p(\cdot)}$ and the positive equivalence constants are independent of $E$.
\end{remark}
The following is an estimate about $\|w\mathbf{1}_E\|_{L^{p(\cdot)}}$
with $w\in \mathcal{A}_{p(\cdot),\infty}$.
\begin{lemma}\label{EB 3}
Let $p(\cdot) \in \mathcal{P}$ with $p(\cdot) \in LH$
and let $w \in \mathcal{A}_{p(\cdot),\infty}$.
Then there exist positive constants $C$ and $A$,
depending only on $p(\cdot)$ and $n$, such that,
for any cube $Q$ in $\mathbb{R}^n$ and for any measurable set $E\subset Q$,
\begin{align}\label{EBeq 7}
\frac{\|w\mathbf{1}_E\|_{L^{p(\cdot)}}}{\|w\mathbf{1}_{Q}\|_{L^{p(\cdot)}}}
\leq C [w]_{\mathcal{A}_{p(\cdot),\infty}}^{A}
L_w^{\frac{1}{p_+}} \left[\frac{\mathbb{W}(E)}{\mathbb{W}(Q)}\right]^{\frac{1}{p_+}},
\end{align}
where $\mathbb{W} := [w(\cdot)]^{p(\cdot)}$ and $L_w$ is the same as in \eqref{def Lw}.
\end{lemma}
\begin{proof}
Assume first $\|w\mathbf{1}_{Q_0}\|_{L^{p(\cdot)}} = 1$.
We firstly consider the case $\|w\mathbf{1}_Q\|_{L^{p(\cdot)}} \leq 1$.
In this case, using Lemma \ref{rhof 1} with the fact
$\|w\mathbf{1}_E\|_{L^{p(\cdot)}} \leq \|w\mathbf{1}_Q\|_{L^{p(\cdot)}}\leq 1$,
we find that $\|w\mathbf{1}_E\|_{L^{p(\cdot)}} \leq \mathbb{W}(E)^{\frac{1}{p_+(Q)}}$
and $\mathbb{W}(Q)^{\frac{1}{p_-(Q)}} \leq \|w\mathbf{1}_Q\|_{L^{p(\cdot)}}$.
By this and Lemma \ref{Bp bound},
we conclude that there exists a positive constant $A$, depending only on $p(\cdot)$ and $n$,
such that
\begin{align}\label{EBeq 8}
\frac{\|w\mathbf{1}_E\|_{L^{p(\cdot)}}}{\|w\mathbf{1}_{Q}\|_{L^{p(\cdot)}}}
 &= \|w\mathbf{1}_{Q}\|^{\frac{p_-(Q) - p_+(Q)}{p_+(Q)}}_{L^{p(\cdot)}} \frac{\|w\mathbf{1}_E\|_{L^{p(\cdot)}}}{\|w\mathbf{1}_{Q}\|^{\frac{p_-(Q)}{p_+(Q)}}_{L^{p(\cdot)}}}
 \lesssim [w]^{\frac{A}{p_+(Q)}}_{\mathcal{A}_{p(\cdot),\infty}}
\left[\frac{\mathbb{W}(E)}{\mathbb{W}(Q)}\right]^{\frac{1}{p_+(Q)}}.
\end{align}
Since Lemma \ref{Ap Apinfty 1} with $w \in \mathcal{A}_{p(\cdot),\infty}$,
we deduce that $1 \leq C_{p(\cdot),n}[w]_{\mathcal{A}_{p(\cdot),\infty}}$.
Hence, using this, \eqref{EBeq 8}, and the facts $p_+(Q) \leq p_+$ and $\frac{\mathbb{W}(E)}{\mathbb{W}(Q)} \leq 1$,
we obtain
\begin{align*}
\frac{\|w\mathbf{1}_E\|_{L^{p(\cdot)}}}{\|w\mathbf{1}_{Q}\|_{L^{p(\cdot)}}}  \leq [w]^{\frac{A}{p_+}}_{\mathcal{A}_{p(\cdot),\infty}}
\left[\frac{\mathbb{W}(E)}{\mathbb{W}(Q)}\right]^{\frac{1}{p_+}}.
\end{align*}
This finishes the proof for the case $\|w\mathbf{1}_Q\|_{L^{p(\cdot)}} \leq 1$.

Secondly, we consider the case
$\|w\mathbf{1}_E\|_{L^{p(\cdot)}} \leq 1 \leq \|w\mathbf{1}_Q\|_{L^{p(\cdot)}}$.
Then it follows from this and Lemma \ref{rhof 1} that
$\|w\mathbf{1}_E\|_{L^{p(\cdot)}} \leq \mathbb{W}(E)^{\frac{1}{p_+(Q)}}$
and $\mathbb{W}(Q)^{\frac{1}{p_+(Q)}} \leq \|w\mathbf{1}_Q\|_{L^{p(\cdot)}}$.
Hence, using this, we conclude that
\begin{align*}
\frac{\|w\mathbf{1}_E\|_{L^{p(\cdot)}}}{\|w\mathbf{1}_{Q}\|_{L^{p(\cdot)}}} \leq \left[\frac{\mathbb{W}(E)}{\mathbb{W}(Q)}\right]^{\frac{1}{p_+(Q)}}
\leq \left[\frac{\mathbb{W}(E)}{\mathbb{W}(Q)}\right]^{\frac{1}{p_+}},
\end{align*}
which completes the proof for the case $ \|w\mathbf{1}_E\|_{L^{p(\cdot)}} \leq 1 \leq \|w\mathbf{1}_Q\|_{L^{p(\cdot)}} $.

Finally, we consider the case $\|w\mathbf{1}_E\|_{L^{p(\cdot)}}\geq 1$.
By this and Lemma \ref{EB 2},
we find that
\begin{align}\label{EBeq 1}
\frac{\|w\mathbf{1}_E\|_{L^{p(\cdot)}}}{\|w\mathbf{1}_{Q}\|_{L^{p(\cdot)}}}
\lesssim [w]_{\mathcal{A}_{p(\cdot),\infty}}^{2A} \left[\frac{\mathbb{W}(E)}{\mathbb{W}(Q)}\right]^{\frac{1}{p_\infty}}.
\end{align}
Notice that $\frac{1}{p_+} \leq \frac{1}{p_\infty}$.
Using this, \eqref{EBeq 1}, and the fact $ \frac{\mathbb{W}(E)}{\mathbb{W}(Q)}\leq 1 $,
we conclude that
\begin{align*}
\frac{\|w\mathbf{1}_E\|_{L^{p(\cdot)}}}{\|w\mathbf{1}_{Q}\|_{L^{p(\cdot)}}}
\lesssim [w]_{\mathcal{A}_{p(\cdot),\infty}}^{2A} \left[\frac{\mathbb{W}(E)}{\mathbb{W}(Q)}\right]^{\frac{1}{p_+}}.
\end{align*}
This finishes the proof of \eqref{EBeq 7} in the case $\|w\mathbf{1}_E\|_{L^{p(\cdot)}} \geq 1$.
Namely, there exists a positive constant $A$, depending only on $p(\cdot)$ and $n$,
such that, for any cube $Q$ in $\mathbb{R}^n$ and any measurable set $E\subset Q$,
\begin{align}\label{EBeq 6}
\frac{\|w\mathbf{1}_E\|_{L^{p(\cdot)}}}{\|w\mathbf{1}_{Q}\|_{L^{p(\cdot)}}}
\lesssim [w]_{\mathcal{A}_{p(\cdot),\infty}}^{A} \left[\frac{\mathbb{W}(E)}{\mathbb{W}(Q)}\right]^{\frac{1}{p_+}}.
\end{align}

In the last step,
we consider the case $\|w\mathbf{1}_{Q_0}\|_{L^{p(\cdot)}} \neq 1$.
Similar to the proof of Lemma \ref{Bp bound},
since $w \in \mathcal{A}_{p(\cdot),\infty}$,
it follows that $\|w\mathbf{1}_{Q_0}\|_{L^{p(\cdot)}} > 0$.
Next, letting $v := \frac{w}{\|w\mathbf{1}_{Q_0}\|_{L^{p(\cdot)}}}$,
we have $ \|v\mathbf{1}_{Q_0}\|_{L^{p(\cdot)}} = 1 $
and, by \eqref{EB eq 13},
we obtain immediately $[v]_{\mathcal{A}_{p(\cdot),\infty}} = [w]_{\mathcal{A}_{p(\cdot),\infty}}$.
Thus, using this and the just proven conclusion that \eqref{EBeq 6}
holds with $w$ replaced by $v$
and letting $\mathbb{W}_0 := [v(\cdot)]^{p(\cdot)}$,
we obtain
\begin{align}\label{EBeq 2}
\frac{\|v\mathbf{1}_E\|_{L^{p(\cdot)}}}{\|v\mathbf{1}_{Q}\|_{L^{p(\cdot)}}}
\lesssim [v]_{\mathcal{A}_{p(\cdot),\infty}}^{A}
\left[\frac{\mathbb{W}_0(E)}{\mathbb{W}_0(Q)}\right]^{\frac{1}{p_+}}.
\end{align}
If $\|w\mathbf{1}_{Q_0}\|_{L^{p(\cdot)}} < 1$,
then
$ \|w\mathbf{1}_{Q_0}\|_{L^{p(\cdot)}}^{-p_-} \leq \|w\mathbf{1}_{Q_0}\|_{L^{p(\cdot)}}^{-p(x)} \leq \|w\mathbf{1}_{Q_0}\|_{L^{p(\cdot)}}^{-p_+} $
for any $x\in\mathbb{R}^n$.
Hence, combining this with \eqref{EBeq 2},
we find that
\begin{align}\label{EBeq 3}
\frac{\mathbb{W}_0(E)}{\mathbb{W}_0(Q)} = \frac{\int_E v(x)^{p(x)}\,dx}{\int_Q v(x)^{p(x)}\,dx}
\leq \frac{ \|w\mathbf{1}_{Q_0}\|^{-p_+}_{L^{p(\cdot)}} \int_E w(x)^{p(x)}\,dx}{\|w\mathbf{1}_{Q_0}\|^{-p_-}_{L^{p(\cdot)}}\int_Q w(x)^{p(x)}\,dx}
=  \left\|w\mathbf{1}_{Q_0}\right\|^{p_--p_+}_{L^{p(\cdot)}} \frac{\mathbb{W}(E)}{\mathbb{W}(Q)}.
\end{align}
Conversely,
if $\|w\mathbf{1}_{Q_0}\|_{L^{p(\cdot)}} \geq 1$,
then $ \|w\mathbf{1}_{Q_0}\|_{L^{p(\cdot)}}^{-p_+} \leq \|w\mathbf{1}_{Q_0}\|_{L^{p(\cdot)}}^{-p(x)} \leq \|w\mathbf{1}_{Q_0}\|_{L^{p(\cdot)}}^{-p_-}$ for any $x\in\mathbb{R}^n$.
From this and \eqref{EBeq 2},
we infer that
\begin{align*}
\frac{\mathbb{W}_0(E)}{\mathbb{W}_0(Q)} = \frac{\int_E v(x)^{p(x)}\,dx}{\int_Q v(x)^{p(x)}\,dx}
\leq \frac{ \|w\mathbf{1}_{Q_0}\|^{-p_-}_{L^{p(\cdot)}} \int_E w(x)^{p(x)}\,dx}{\|w\mathbf{1}_{Q_0}\|^{-p_+}_{L^{p(\cdot)}}\int_Q w(x)^{p(x)}\,dx}
=  \left\|w\mathbf{1}_{Q_0}\right\|^{p_+-p_-}_{L^{p(\cdot)}} \frac{\mathbb{W}(E)}{\mathbb{W}(Q)}.
\end{align*}
Thus, combining this, \eqref{EBeq 2}, and \eqref{EBeq 3},
we conclude that
\begin{align*}
\frac{\|w\mathbf{1}_E\|_{L^{p(\cdot)}}}{\|w\mathbf{1}_{Q}\|_{L^{p(\cdot)}}}
& = \frac{\|v\mathbf{1}_E\|_{L^{p(\cdot)}}}{\|v\mathbf{1}_{Q}\|_{L^{p(\cdot)}}}
 \lesssim [w]_{\mathcal{A}_{p(\cdot),\infty}}^{A}
\max\left\{ \left\|w\mathbf{1}_{Q_0}\right\|^{p_--p_+}_{L^{p(\cdot)}}, \left\|w\mathbf{1}_{Q_0}\right\|^{p_+ - p_-}_{L^{p(\cdot)}} \right\}^\frac{1}{p_+}
\left[\frac{\mathbb{W}(E)}{\mathbb{W}(Q)}\right]^{\frac{1}{p_+}}. \nonumber
\end{align*}
This finishes the proof of Lemma \ref{EB 3}.
\end{proof}
Now, we recall the reverse H\"older's inequality for $A_\infty$ weights
(see, for instance, \cite[Theorem 2.3]{hp13}).

\begin{lemma}\label{reverse Holder}
If $w\in A_\infty$,
then, for any $r \in [1, 1 + \frac{1}{\tau_n [w]_{A_\infty}} ]$
with $\tau_n$ being a positive constant depending only on $n$ and for
any cube $Q$ in $\mathbb{R}^n$,
\begin{align*}
\left[ \fint_Q w^r(x)\,dx \right]^{\frac1r} \leq 2 \fint_Q w(x)\,dx.
\end{align*}
\end{lemma}

Next, we are ready to prove Proposition \ref{reverse Holder w1}.
\begin{proof}[Proof of Proposition \ref{reverse Holder w1}]
From Lemma \ref{Ap Apinfty 1} with $w \in \mathcal{A}_{p(\cdot),\infty}$,
it follows immediately that $w \in A_\infty$ and
$[w]_{A_\infty} \leq C_{p(\cdot),n} [w]_{\mathcal{A}_{p(\cdot),\infty}}$.
Hence, by this and Lemma \ref{reverse Holder},
we obtain, for any cube $Q$ in $\mathbb{R}^n$,
\begin{align*}
\left[ \fint_Q w(x)^\delta\,dx \right]^\frac1\delta \leq 2 \fint_Q w(x)\,dx,
\end{align*}
where $\delta:= 1 + \frac{1}{\tau_n C_{p(\cdot),n}[w]_{\mathcal{A}_{p(\cdot),\infty}}}$.
Using this and H\"older's inequality with the obvious fact $\delta \geq 1$,
we find that, for any cube $Q$ in $\mathbb{R}^n$ and any measurable set $E\subset Q$ with $|E| > 0$,
\begin{align*}
w(E) &= |E| \fint_E w(x)\,dx \leq |E| \left[\fint_E w(x)^\delta \,dx\right]^\frac1\delta
= |E|^{1 - \frac{1}{\delta}} |Q|^{\frac{1}{\delta}} \left[\fint_Q w(x)^\delta \,dx\right]^\frac1\delta\\
&\leq 2|E|^{1 - \frac{1}{\delta}} |Q|^{\frac{1}{\delta}} \fint_Q w(x) \,dx
= 2\left( \frac{|E|}{|Q|} \right)^{1-\frac{1}{\delta}} w(Q),
\end{align*}
which further implies that
\begin{align}\label{wEQ EQ}
\frac{w(E)}{w(Q)} \leq 2\left( \frac{|E|}{|Q|} \right)^{1-\frac{1}{\delta}}.
\end{align}

Now, let $\alpha \in (0,\infty)$ be a constant which will be determined later.
From \eqref{wEQ EQ}, it follows that,
for any cube $Q$ in $\mathbb{R}^n$ and any measurable set $E\subset Q$
with $|E| \leq \alpha |Q|$,
we have $\frac{w(E)}{w(Q)} \leq 2\alpha^{1-\frac1\delta} $,
which further implies that $w(E) \leq 2\alpha^{1-\frac1\delta} w(Q)$.
From this and the arbitrariness of $E$,
via replacing $E$ by $Q\setminus E$,
we deduce that, for any cube $Q$ in $\mathbb{R}^n$ and any measurable set $E\subset Q$
with $ |E| > (1 - \alpha) |Q| $,
\begin{align}\label{wEQ EQ 1}
1 - 2\alpha^{1-\frac1\delta} < \frac{w(E)}{w(Q)}.
\end{align}
Similar to the proof of \eqref{EB eq 2},
using Lemma \ref{est fQ} and \eqref{wM 2},
we obtain
\begin{align*}
\frac{w(E)}{w(Q)} \leq C_{p(\cdot),n}\frac{ \|w\mathbf{1}_E\|_{L^{p(\cdot)}}}{\|w\mathbf{1}_Q\|_{L^{p(\cdot)}}}
\frac{|Q|}{w(Q)} \frac{\|w\mathbf{1}_Q\|_{L^{p(\cdot)}}}{\|\mathbf{1}_Q\|_{L^{p(\cdot)}}}
\leq C_{p(\cdot),n}[w]_{\mathcal{A}_{p(\cdot),\infty}} \frac{ \|w\mathbf{1}_E\|_{L^{p(\cdot)}}}{\|w\mathbf{1}_Q\|_{L^{p(\cdot)}}},
\end{align*}
which, combined with Lemma \ref{EB 3},
further implies that there exists a positive constant $A$,
depending only on $p(\cdot)$ and $n$,
such that
\begin{align*}
\frac{w(E)}{w(Q)} \lesssim [w]_{\mathcal{A}_{p(\cdot),\infty}}^{A}
L_w^{\frac{1}{p_+}}\left[\frac{\mathbb{W}(E)}{\mathbb{W}(Q)}\right]^{\frac{1}{p_+}},
\end{align*}
where $\mathbb{W} := [w(\cdot)]^{p(\cdot)}$ and $L_w$ is the same as in Lemma \ref{EB 3}.
Thus, by this and \eqref{wEQ EQ 1},
we conclude that,
for any cube $Q$ in $\mathbb{R}^n$ and any measurable set $E\subset Q$ with $ |E| > (1 - \alpha) |Q| $,
\begin{align*}
1 - 2\alpha^{1-\frac1\delta} \lesssim [w]_{\mathcal{A}_{p(\cdot),\infty}}^{A}
L_w^{\frac{1}{p_+}}\left[\frac{\mathbb{W}(E)}{\mathbb{W}(Q)}\right]^{\frac{1}{p_+}},
\end{align*}
which, together with \cite[Lemma 7.3.3]{g14}, further implies that $\mathbb{W} \in A_\infty$.
This finishes the proof of Proposition \ref{reverse Holder w1}.
\end{proof}
Next, we show the sufficiency of Theorem \ref{Apinfty A} in the following proposition.
\begin{proposition}\label{Apinfty Ainfty}
Let $p(\cdot)\in\mathcal{P}$ with $p(\cdot) \in LH$.
Then, for any $\mathbb{W} \in A_\infty$,
$w := [\mathbb{W}(\cdot)]^{\frac{1}{p(\cdot)}} \in \mathcal{A}_{p(\cdot),\infty}$;
moreover, there exist positive constants $C$ and $A$, depending only on $p(\cdot)$ and $n$,
such that
\begin{align}\label{Apinfty Ainfty 1}
[w]_{\mathcal{A}_{p(\cdot),\infty}} \leq C \widetilde{L_w}
[\mathbb{W}]_{A_\infty}^A,
\end{align}
where
$$ \widetilde{L_w} := \max\left\{ \left[\mathbb{W}\left(Q_0\right)\right]^{\frac{1}{p_-} - \frac{1}{p_+}},
\left[\mathbb{W}\left(Q_0\right)\right]^{\frac{1}{p_+} - \frac{1}{p_-}} \right\}  $$
and $Q_0$ is the same as in \eqref{def Q0}.
\end{proposition}

\begin{remark}
Let $p(\cdot) \in \mathcal{P}$ with $p(\cdot) \in LH$.
As we can see, $\widetilde{L_w}$ in \eqref{Apinfty Ainfty 1} depends on both $W$ and $p(\cdot)$.
Indeed,   there does not exist any measurable positive function $\varphi$
such that, for any $\mathbb{W} \in A_\infty$,
$ [w]_{\mathcal{A}_{p(\cdot),\infty}} \lesssim \varphi([\mathbb{W}]_{A_\infty}) $,
where the implicit positive constant is independent of $\mathbb{W}$.

We show this by giving an example.
Let $\mathbb{W}_k := k$ for any $k\in\mathbb{N}$.
Then it is obvious that $\mathbb{W}_k \in A_\infty$
and $[\mathbb{W}_k]_{A_\infty} = 1$ for any $k\in\mathbb{N}$.
Now, assume
\begin{align*}
p(x) :=
\begin{cases}
\displaystyle 1 &\text{if}\ x\in (-\infty,-1)\cup (1,\infty),\\
\displaystyle 3 + 2x &\text{if}\  x\in [-1,0],\\
\displaystyle 3 - 2x &\text{if}\  x\in (0,1]
\end{cases}
\end{align*}
and $w_k := [\mathbb{W}_k(\cdot)]^{\frac{1}{p(\cdot)}}$ for any $k\in\mathbb{N}$.
If there exists a positive function $\varphi$ such that
$ [w]_{\mathcal{A}_{p(\cdot),\infty}} \lesssim \varphi([\mathbb{W}]_{A_\infty}) $,
then, from the fact that $ [\mathbb{W}_k]_{A_\infty} = 1 $ for any $k\in\mathbb{N}$,
it follows immediately that
\begin{align}\label{rem ap eq 5}
\sup_k [w_k]_{\mathcal{A}_{p(\cdot),\infty}} < \infty.
\end{align}
However, this is not true.
Indeed, letting $Q := (0,1]$,
then, by the definition of $\mathcal{A}_{p(\cdot),\infty}$,
we obtain, for any $k\in\mathbb{N}$,
\begin{align}\label{rem ap eq 1}
[w_k]_{\mathcal{A}_{p(\cdot),\infty}}
\geq \frac{1}{\|\mathbf{1}_Q\|_{L^{p(\cdot)}}} \left\|w_k\mathbf{1}_Q\right\|_{L^{p(\cdot)}} \exp\left( \fint_Q \log w_k^{-1}(x) \,dx \right).
\end{align}
Using Lemma \ref{rhof 1} with the fact $\rho_{L^{p(\cdot)}}(\mathbf{1}_Q) = |Q| = 1$,
we find that
\begin{align}\label{rem ap eq 2}
\|\mathbf{1}_Q\|_{L^{p(\cdot)}} = 1;
\end{align}
moreover, for any $k\in\mathbb{N}$, we obtain
\begin{align*}
\exp\left( \fint_Q \log w_k^{-1}(x) \,dx \right)
= \exp\left( -\log k \int_0^1 \frac{1}{3-2x}\,dx \right) = k^{-\frac{\log 3}{2}}.
\end{align*}
Thus, from these and \eqref{rem ap eq 2},
we infer that, for any $k\in\mathbb{N}$,
\begin{align}\label{rem ap eq 4}
\frac{1}{\|\mathbf{1}_Q\|_{L^{p(\cdot)}}} \left\|w_k\mathbf{1}_Q\right\|_{L^{p(\cdot)}} \exp\left( \fint_Q \log w_k^{-1}(x) \,dx \right)
 =  \left\| k^{\frac{1}{p(\cdot)}-\frac{\log 3}{2}} \mathbf{1}_Q\right\|_{L^{p(\cdot)}}.
\end{align}
By the definition of $\rho_{L^{p(\cdot)}}$,
we have, for any $k\in\mathbb{N}$,
\begin{align*}
\rho_{L^{p(\cdot)}}\left( k^{\frac{1}{p(\cdot)}-\frac{\log 3}{2}}  \mathbf{1}_Q \right)
= \int_0^1 k^{1 - \frac{\log 3}{2}(3-2x)}\,dx = \frac{k^{1 + \log 3} - 1}{3\log(3) k^{\frac{3\log 3}{2}}  \log k }.
\end{align*}
Since $ \log 3 + 1 > \frac32 \log3 $,
it follows that $\rho_{L^{p(\cdot)}}( k^{\frac{1}{p(\cdot)}-\frac{\log 3}{2}}  \mathbf{1}_Q ) \to\infty$
as $k\to \infty$.
Using this and Lemma \ref{rhof 1},
we conclude that the right-hand side of \eqref{rem ap eq 4}
tends to infinity as $k\to \infty$,
which, combined with \eqref{rem ap eq 1}, yields
$[w_k]_{\mathcal{A}_{p(\cdot),\infty}} \to \infty $ as $k \to \infty$.
This contradicts \eqref{rem ap eq 5} and hence finishes the proof of the above claim.
\end{remark}

Before giving the proof of Proposition \ref{Apinfty Ainfty},
we recall some basic properties of $A_\infty$ weights.
The following result is an estimate about the $p_-(Q) - p_+(Q)$ power of $\mathbb{W}(Q)$
for any $\mathbb{W} \in A_\infty$ and any cube $Q$ in $\mathbb{R}^n$,
which is similar to Lemma \ref{Bp bound}.
\begin{lemma}\label{ww Bp}
Let $p(\cdot) \in \mathcal{P}$ with $p(\cdot) \in LH$
and $\mathbb{W} \in A_\infty$.
Then there exist positive constants $C$ and $A$, depending only on $p(\cdot)$ and $n$,
such that, for any cube $Q$ in $\mathbb{R}^n$,
\begin{align}\label{ww B1}
\mathbb{W}(Q)^{p_-(Q) - p_+(Q)} \leq C [\mathbb{W}]_{A_\infty}^A.
\end{align}
Moreover, if $\mathbb{W}(Q) \leq 1$,
then there further exists a positive constant $C$, depending only on $p(\cdot)$ and $n$,
such that, for any $x\in Q$,
\begin{align}\label{ww B2}
C^{-1} \left[\mathbb{W}\right]_{A_\infty}^{-\frac{A}{p_-}} \mathbb{W}(Q)^{-1}
\leq \mathbb{W}(Q)^{-\frac{p(x)}{p_Q}}
\leq C \left[\mathbb{W}\right]_{A_\infty}^{\frac{A}{p_-}} \mathbb{W}(Q)^{-1}.
\end{align}
\end{lemma}
\begin{proof}
Since the proof of \eqref{ww B1} is a slight modification
of the proof of Lemma \ref{Bp bound} with Lemma \ref{EB 1} replaced
by Lemma \ref{lambdaQ Q},
we omit the details here and only show \eqref{ww B2}.
To this end, by \eqref{ww B1}, the assumption $\mathbb{W}(Q) \leq 1$, and the fact that
$ p_-(Q) - p_+(Q) \leq - |p(x) - p(y)| $ for any $x,y\in Q$,
we find that there exists a positive constant $A$, depending only on $p(\cdot)$ and $n$,
such that $ \mathbb{W}(Q)^{-|p(x) - p(y)|} \leq \mathbb{W}(Q)^{p_-(Q) - p_+(Q)} \lesssim [\mathbb{W}]_{A_\infty}^A $.
Using this, Jensen's inequality, and the assumption $\mathbb{W}(Q) \leq 1$,
we conclude that, for any $x \in Q$,
\begin{align}\label{ww BB eq 9}
\mathbb{W}(Q)^{-|\frac{p(x)}{p_Q} - 1|} & = \mathbb{W}(Q)^{-|\fint_Q \frac{p(x)-p(y)}{p(y)}\,dy|}
 \leq \mathbb{W}(Q)^{-\fint_Q |\frac{p(x) - p(y)}{p(y)}|\,dy}\nonumber\\
& \leq  \left[\fint_Q \mathbb{W}(Q)^{-|p(x) - p(y)|}\,dy\right]^{\frac{1}{p_-}}
\lesssim \left[\mathbb{W}\right]_{A_\infty}^{\frac{A}{p_-}}.
\end{align}
Notice that, by the assumption $\mathbb{W}(Q) \leq 1$,
$$\mathbb{W}(Q)^{-\frac{p(x)}{p_Q}} \leq \mathbb{W}(Q)^{-|\frac{p(x)}{p_Q} - 1|} \mathbb{W}(Q)^{-1}
\ \ \text{and}\ \
\mathbb{W}(Q)^{-1} \leq \mathbb{W}(Q)^{-|\frac{p(x)}{p_Q} - 1|} \mathbb{W}(Q)^{-\frac{p(x)}{p_Q}},  $$
which, combined with \eqref{ww BB eq 9}, further implies that
$ [\mathbb{W}]_{A_\infty}^{-\frac{A}{p_-}} \mathbb{W}(Q)^{-1}
\lesssim \mathbb{W}(Q)^{-\frac{p(x)}{p_Q}} \lesssim [\mathbb{W}]_{A_\infty}^{\frac{A}{p_-}} \mathbb{W}(Q)^{-1}. $
This finishes the proof of Lemma \ref{ww Bp}.
\end{proof}
The following lemma is some estimates about
$\|\mathbb{W}^{\frac{1}{p(\cdot)}} \mathbf{1}_Q \|_{L^{p(\cdot)}}$.
We note that this result can be found in \cite[Corollary 3.7]{dh08};
however, since we need the explicit expressions of related positive constants,
we give a detailed proof here.
\begin{lemma}\label{ww BB}
Let $p(\cdot) \in \mathcal{P}$ with $p(\cdot) \in LH$ and let $\mathbb{W} \in A_\infty$.
Then there exist positive constants $C$ and $A$,
depending only on $p(\cdot)$ and $n$,
such that, for any cube $Q$ in $\mathbb{R}^n$,
\begin{align}\label{ww BB eq 5}
C^{-1} \left[\mathbb{W}\right]^{-A}_{A_\infty} \widetilde{L_w}^{-1} \mathbb{W}(Q)^{\frac{1}{p_Q}}
\leq \left\| w \mathbf{1}_Q \right\|_{L^{p(\cdot)}} \leq C \left[\mathbb{W}\right]^A_{A_\infty}
\widetilde{L_w} \mathbb{W}(Q)^{\frac{1}{p_Q}},
\end{align}
where $w := \mathbb{W}(\cdot)^{\frac{1}{p(\cdot)}}$
and $\widetilde{L_w}$ is as in Proposition \ref{Apinfty Ainfty}.
\end{lemma}
\begin{proof}
We first prove \eqref{ww BB eq 5} under the assumption $ \mathbb{W}(Q_0) = 1 $, where $Q_0$ is the same as in \eqref{def Q0}.
If $\mathbb{W}(Q) \leq 1$,
then, by Lemma \ref{ww Bp}, we obtain
\begin{align}\label{ww BB eq 12}
\rho_{p(\cdot)} \left( \mathbb{W}(Q)^{-\frac{1}{p_Q}} w \mathbf{1}_Q \right)
& = \int_{Q} \mathbb{W}(Q)^{-\frac{p(x)}{p_Q}} \mathbb{W}(x) \,dx
\gtrsim \left[\mathbb{W}\right]_{A_\infty}^{-A} \int_Q \mathbb{W}(Q)^{-1} \mathbb{W}(x)\,dx = \left[\mathbb{W}\right]_{A_\infty}^{-A}
\end{align}
and
\begin{align*}
\rho_{p(\cdot)} \left( \mathbb{W}(Q)^{-\frac{1}{p_Q}} w \mathbf{1}_Q \right)
& = \int_{Q} \mathbb{W}(Q)^{-\frac{p(x)}{p_Q}} \mathbb{W}(x) \,dx
\lesssim \left[\mathbb{W}\right]_{A_\infty}^{A} \int_Q \mathbb{W}(Q)^{-1} \mathbb{W}(x)\,dx = \left[\mathbb{W}\right]_{A_\infty}^{A},
\end{align*}
where $A$ is a positive constant, depending only on $p(\cdot)$ and $n$.
Combining this, \eqref{ww BB eq 12}, and Lemma \ref{rhof 1},
we conclude that
$ [\mathbb{W}]_{A_\infty}^{-\frac{A}{p_+}} \lesssim \|\mathbb{W}(Q)^{-\frac{1}{p_Q}} w \mathbf{1}_Q\|_{L^{p(\cdot)}} \lesssim [\mathbb{W}]_{A_\infty}^{\frac{A}{p_+}} $,
which further implies the desired estimate
$[\mathbb{W}]_{A_\infty}^{-\frac{A}{p_+}} \mathbb{W}(Q)^{\frac{1}{p_Q}} \lesssim \| w \mathbf{1}_Q\|_{L^{p(\cdot)}} \lesssim [\mathbb{W}]_{A_\infty}^{\frac{A}{p_+}} \mathbb{W}(Q)^{\frac{1}{p_Q}}$.

Next, we consider the case $\mathbb{W}(Q) \geq 1$.
Let $Q := Q(x_0, l)$ with $x_0\in\mathbb{R}^n$ and $l \in (0,\infty)$
and we prove \eqref{ww BB eq 5} in the following two cases:
$\sqrt{n}l \leq |x_0|$ and $\sqrt{n}l > |x_0|$.
We first begin with the case $\sqrt{n}l \leq |x_0|$.
By the assumption $\sqrt{n}l \leq |x_0|$,
we find that $|x| \leq |x_0| + \sqrt{n}l \lesssim |x_0|$
and $|x| \geq |x_0| - \frac{\sqrt{n}}{2} l \gtrsim |x_0|$ for any $x \in Q$,
which further imply that, for any $x \in Q$, $|x| \sim |x_0|$.
Thus, using this, we obtain, for any $x \in Q$,
$$ \frac{1}{\log(e + |x|)} \sim \frac{1}{\log(e + |x_0|)} = \fint_Q \frac{1}{\log(e + |x_0|)}\,dy \sim \fint_Q \frac{1}{\log(e + |y|)}\,dy. $$
From this, Remark \ref{rem p}, and the definition of $p_Q$,
we deduce that
\begin{align*}
\left| p(x) - p_Q \right| & = \left| p(x) p_Q \right| \left|\frac{1}{p_Q} - \frac{1}{p(x)}\right|
\lesssim p_+^2 \left|\frac{1}{p_Q} - \frac{1}{p(x)}\right|
\lesssim \left| \frac{1}{p(x)} - \frac{1}{p_\infty} \right| + \left| \frac{1}{p_Q} - \frac{1}{p_\infty} \right|\\
&\leq \frac{1}{\log(e + |x|)} + \fint_Q \frac{1}{\log(e+|y|)} \,dy
\lesssim \frac{1}{\log(e + |x|)}, \nonumber
\end{align*}
which further implies that there exists a positive constant $C_0$,
depending only on $p(\cdot)$ and $n$,
such that
\begin{align*}
\left| p(x) - p_Q \right| \leq \frac{C_0}{\log(e + |x|)}.
\end{align*}
Then, by applying an argument similar to that used in the proof of Lemma \ref{EB 2}
via replacing Lemma \ref{rs f 1} by Lemma \ref{rs f},
we conclude that there exists a positive constant $A$,
depending only on $p(\cdot)$ and $n$, such that
$$\left[\mathbb{W}\right]_{A_\infty}^{-A} \mathbb{W}(Q)^{\frac{1}{p_Q}} \lesssim \left\| w\mathbf{1}_Q \right\|_{L^{p(\cdot)}} \lesssim \left[\mathbb{W}\right]_{A_\infty}^A \mathbb{W}(Q)^{\frac{1}{p_Q}}. $$
This finishes the proof of the case $\sqrt{n}l \leq |x_0|$.

Now, we consider the case $\sqrt{n}l > |x_0|$.
Let $\widetilde{Q} := Q(\mathbf{0}, l)$ and $S := Q(\mathbf{0}, 2|x_0| + l)$.
Then $ Q\subset S $ and $\widetilde{Q}\subset S$.
Using this, Lemma \ref{lambdaQ Q} with $\lambda = \frac{2|x_0| + l}{l}$,
and the assumption $ |x_0| < \sqrt{n}l $,
we conclude that
\begin{align*}
\frac{\mathbb{W}(Q)}{\mathbb{W}(\widetilde{Q})}
&= \frac{\mathbb{W}(Q)}{\mathbb{W}(S)}\frac{\mathbb{W}(S)}{\mathbb{W}(\widetilde{Q})}
\leq \frac{\mathbb{W}(Q)}{\mathbb{W}(S)} \left( 2\frac{2|x_0| + l}{l} \right)^{ 2^n(1 + \log_2 [\mathbb{W}]_{A_\infty}) }\\
&\lesssim \left( 2 + 4\sqrt{n} \right)^{2^n \log_2 [\mathbb{W}]_{A_\infty} }
 = \left[\mathbb{W}\right]_{A_\infty}^{ 2^n \log_2( 2 + 4\sqrt{n} ) },
\end{align*}
which further implies that
\begin{align}\label{QQ1}
\mathbb{W}(Q) \lesssim \left[\mathbb{W}\right]_{A_\infty}^{ 2^n \log_2( 2 + 4\sqrt{n} ) } \mathbb{W}(\widetilde{Q}),
\end{align}
where the implicit positive constant depends only on $p(\cdot)$ and $n$.
Observe that, with replacing \eqref{w doubling 1} by Lemma \ref{lambdaQ Q},
Lemma \ref{EB 2} still holds for $A_\infty$ weights.
Thus, there exists a positive constant $A$,
depending only on $p(\cdot)$ and $n$,
such that, for any cube $Q$ in $\mathbb{R}^n$ with $\mathbb{W}(Q) \geq 1$,
$$ [\mathbb{W}]^{-A}_{A_\infty}\mathbb{W}(Q)^{\frac{1}{p_\infty}} \lesssim \left\|w\mathbf{1}_Q\right\|_{L^{p(\cdot)}} \lesssim [\mathbb{W}]^A_{A_\infty} \mathbb{W}(Q)^{\frac{1}{p_\infty}}, $$
which further implies that
$[\mathbb{W}]^{-A}_{A_\infty} \mathbb{W}(Q)^{\frac{1}{p_Q}} \lesssim \mathbb{W}(Q)^{\frac{1}{p_Q} - \frac{1}{p_\infty}} \|w\mathbf{1}_Q\|_{L^{p(\cdot)}}$
and
\begin{align*}
\left\|w\mathbf{1}_Q\right\|_{L^{p(\cdot)}} \lesssim
 [\mathbb{W}]^A_{A_\infty} \mathbb{W}(Q)^{\frac{1}{p_Q}} \mathbb{W}(Q)^{\frac{1}{p_\infty} - \frac{1}{p_Q}}.
\end{align*}
Thus, using this, to prove \eqref{ww BB eq 5},
we only need to show that
\begin{align}\label{ww BB eq 6}
\mathbb{W}(Q)^{|\frac{1}{p_\infty} - \frac{1}{p_Q}|} \lesssim 1.
\end{align}
From \eqref{QQ1} with the assumption $\mathbb{W}(Q) \geq 1$
and the fact that
$\fint_Q \frac{1}{\log(e + |x|)}\,dx \leq \fint_{\widetilde{Q}}\frac{1}{\log(e + |x|)}\,dx, $
it follows that
\begin{align}\label{ww BB eq 7}
\mathbb{W}(Q)^{|\frac{1}{p_\infty} - \frac{1}{p_Q}|}
&\leq \mathbb{W}(Q)^{\fint_Q \frac{C_\infty}{\log(e + |x|)}\,dx}
\leq \mathbb{W}(Q)^{\fint_{\widetilde{Q}} \frac{C_\infty}{\log(e + |x|)}\,dx}\nonumber\\
&\lesssim \left\{\left[\mathbb{W}\right]_{A_\infty}^{ 2^n \log_2( 2 + 4\sqrt{n} ) }  \mathbb{W}\left(\widetilde{Q}\right)\right\}^{\fint_{\widetilde{Q}} \frac{C_\infty}{\log(e + |x|)}\,dx}\nonumber\\
& \leq \left[\mathbb{W}\right]_{A_\infty}^{C_\infty 2^n \log_2( 2 + 4\sqrt{n} ) } \exp\left( \log\left(\mathbb{W}\left(\widetilde{Q}\right)\right) \fint_{\widetilde{Q}} \frac{C_\infty}{\log(e + |x|)}\,dx \right),
\end{align}
where $C_\infty$ is the constant appearing in \eqref{clogp infty}
with  $p(\cdot)$ replaced by $\frac{1}{p(\cdot)}$.

If $ l \leq 2e $, then, by  the assumption $\mathbb{W}(Q_0) = 1$,
we find that $ \mathbb{W}(\widetilde{Q}) \leq \mathbb{W}(Q_0) \leq 1 $,
which, combined with \eqref{ww BB eq 7}, further implies that
$$ \mathbb{W}(Q)^{|\frac{1}{p_\infty} - \frac{1}{p_Q}|} \lesssim \left[\mathbb{W}\right]_{A_\infty}^{C_\infty 2^n \log_2( 2 + \sqrt{n} ) }. $$
Next, if $l > 2e$,
then, from Lemma \ref{lambdaQ Q} and the assumption $\mathbb{W}(Q_0) = 1$,
we infer that
\begin{align}\label{ww BB eq 10}
\log\left(\mathbb{W}\left(\widetilde{Q}\right)\right) \fint_{\widetilde{Q}} \frac{dx}{\log(e + |x|)}
& \leq 2^n\left( 1 + \log_2 [\mathbb{W}]_{A_\infty} \right) \log\left(\frac{l}{e}\right) \fint_{\widetilde{Q}} \frac{dx}{\log(e + |x|)}\nonumber\\
& \lesssim 2^n\left( 1 + \log_2 [\mathbb{W}]_{A_\infty} \right) \log\left(\frac{l}{e}\right) l^{-n} \int_0^{\frac{\sqrt{n}l}{2}} \frac{r^{n-1}}{\log(e + r)}\,dr \nonumber\\
& \sim \left( 1 + \log_2 [\mathbb{W}]_{A_\infty} \right) \log\left(\frac{l}{e}\right) \left(\frac{e}{l}\right)^n \int_0^{\frac{\sqrt{n}l}{4e}} \frac{r^{n-1}}{\log(e + 2er)}\,dr.
\end{align}
Using this with letting $t := \frac{l}{e}$ and using \eqref{ww BB eq 7} ,
to prove \eqref{ww BB eq 6},
it is sufficient to show that, for any $t \in (2,\infty)$,
\begin{align}\label{ww BB eq 8}
\frac{\log t}{t^n} \int_{0}^{\frac{\sqrt{n}t}{4}} \frac{r^{n-1}}{\log(e + 2e r)}\,dr \lesssim 1.
\end{align}
Noticing that the left-hand side of \eqref{ww BB eq 8} is continuous on $(2,\infty)$,
we obtain $\frac{\log t}{t^n} \int_{0}^{\frac{\sqrt{n} t}{4}} \frac{r^{n-1}}{\log(e + 2e r)}\,dr$
is bounded on $t \in (2, e^2]$.
Let $ g(t) := \int_{0}^{\frac{\sqrt{n} t}{4}} \frac{r^{n-1}}{\log(e + 2e r)}\,dr - c\frac{t^n}{\log t}$,
where $c$ is a positive constant which will be determined later.
Then the derivative of $g(t)$ is
$$ g'(t) = \left( \frac{\sqrt{n}}{4} \right)^{n} \frac{t^{n-1}}{\log(e + t)} \left\{\frac{\log t}{\log(e + 2e t)} - cn + \frac{c}{\log t}\right\}. $$
Clearly, if $c$ is large enough, then
$g'(t) \leq 0$ and $g(e^2) \leq 0$,
which further implies that, for any $t \in [e^2,\infty)$, $g(t) \leq 0$ and hence
$$\frac{\log t}{t^n} \int_{0}^{\frac{\sqrt{n} t}{4}} \frac{r^{n-1}}{\log(e + 2e r)}\,dr \lesssim 1. $$
Moreover, since \eqref{ww BB eq 8} is independent of $\mathbb{W}$,
it follows that the implicit positive constants in \eqref{ww BB eq 10} are independent of $\mathbb{W}$.
Hence, combining this, \eqref{ww BB eq 8}, \eqref{ww BB eq 7}, and \eqref{ww BB eq 10},
we find that there exists a positive constant $A$, depending only on $p(\cdot)$ and $n$,
such that
$$[\mathbb{W}]_{A_\infty}^{-A} \mathbb{W}(Q)^{\frac{1}{p_Q}} \lesssim \|w\mathbf{1}_Q\|_{L^{p(\cdot)}} \lesssim [\mathbb{W}]_{A_\infty}^A \mathbb{W}(Q)^{\frac{1}{p_Q}}. $$
This finishes the proof of \eqref{ww BB eq 5} in the case $\sqrt{n} > |x_0|$.

Finally, we consider the case $\mathbb{W}(Q_0) \neq 1$.
Since $\mathbb{W} \in A_\infty$, we deduce that $\mathbb{W}(Q_0) \neq 0$.
Now, letting $V := \frac{\mathbb{W}}{\mathbb{W}(Q_0)}$,
we obtain $V\in A_\infty$ and $ V(Q_0) = 1 $
and, moreover, by the definition of $A_\infty$,
we conclude $[V]_{A_\infty} = [\mathbb{W}]_{A_\infty}$.
Using this and the above proved result for the case $\mathbb{W}(Q_0) = 1$,
we find that
\begin{align}\label{ww BB eq 13}
\left\| \mathbb{W}(Q_0)^{-\frac{1}{p(\cdot)}}w\mathbf{1}_Q \right\|_{L^{p(\cdot)}}
&\gtrsim [V]_{A_\infty}^{-A} V(Q)^{\frac{1}{p_Q}}
 = [\mathbb{W}]_{A_\infty}^{-A} \mathbb{W}(Q_0)^{-\frac{1}{p_Q}} \mathbb{W}(Q)^{\frac{1}{p_Q}}
\end{align}
and
\begin{align}\label{ww BB eq 11}
\left\| \mathbb{W}(Q_0)^{-\frac{1}{p(\cdot)}}w\mathbf{1}_Q \right\|_{L^{p(\cdot)}}
&\lesssim [V]_{A_\infty}^A V(Q)^{\frac{1}{p_Q}}
 = [\mathbb{W}]_{A_\infty}^A \mathbb{W}(Q_0)^{-\frac{1}{p_Q}} \mathbb{W}(Q)^{\frac{1}{p_Q}}.
\end{align}
Notice, if $\mathbb{W}(Q_0) \leq 1$,
then
$$\mathbb{W}(Q_0)^{-\frac{1}{p_+}} \left\| w\mathbf{1}_Q\right\|_{L^{p(\cdot)}}
\leq \left\| \mathbb{W}(Q_0)^{-\frac{1}{p(\cdot)}}w\mathbf{1}_Q\right\|_{L^{p(\cdot)}}
\leq \mathbb{W}(Q_0)^{-\frac{1}{p_-}} \left\| w\mathbf{1}_Q\right\|_{L^{p(\cdot)}} $$
and $\mathbb{W}(Q_0)^{-\frac{1}{p_+}} \leq \mathbb{W}(Q_0)^{-\frac{1}{p_Q}} \leq \mathbb{W}(Q_0)^{-\frac{1}{p_-}} $.
Hence, from this with \eqref{ww BB eq 13} and \eqref{ww BB eq 11},
we infer that
\begin{align*}
[\mathbb{W}]_{A_\infty}^{-A} \mathbb{W}(Q_0)^{\frac{1}{p_-}-\frac{1}{p_+}} \mathbb{W}(Q)^{\frac{1}{p_Q}}
\| w\mathbf{1}_Q\|_{L^{p(\cdot)}} \lesssim [\mathbb{W}]_{A_\infty}^A \mathbb{W}(Q_0)^{\frac{1}{p_+}-\frac{1}{p_-}} \mathbb{W}(Q)^{\frac{1}{p_Q}}.
\end{align*}
Conversely, if $\mathbb{W}(Q_0) > 1$,
then
$$\mathbb{W}(Q_0)^{-\frac{1}{p_-}} \| w\mathbf{1}_Q\|_{L^{p(\cdot)}}
\leq \left\| \mathbb{W}(Q_0)^{-\frac{1}{p(\cdot)}}w\mathbf{1}_Q\right\|_{L^{p(\cdot)}}
\leq \mathbb{W}(Q_0)^{-\frac{1}{p_+}} \| w\mathbf{1}_Q\|_{L^{p(\cdot)}}$$
and $ \mathbb{W}(Q_0)^{-\frac{1}{p_-}} \leq \mathbb{W}(Q_0)^{-\frac{1}{p_Q}} \leq \mathbb{W}(Q_0)^{-\frac{1}{p_+}} $.
Thus, from this, \eqref{ww BB eq 13}, and \eqref{ww BB eq 11},
it follows that
\begin{align*}
[\mathbb{W}]_{A_\infty}^{-A} \mathbb{W}(Q_0)^{\frac{1}{p_+}-\frac{1}{p_-}} \mathbb{W}(Q)^{\frac{1}{p_Q}}
\lesssim \| w\mathbf{1}_Q\|_{L^{p(\cdot)}}
\lesssim [\mathbb{W}]_{A_\infty}^A \mathbb{W}(Q_0)^{\frac{1}{p_-}-\frac{1}{p_+}} \mathbb{W}(Q)^{\frac{1}{p_Q}}.
\end{align*}
This finishes the proof of Lemma \ref{ww BB}.
\end{proof}
Finally, we recall an equivalent characterization of $A_\infty$.
We first recall the concept of the space ${\rm BMO}$
(see, for instance, \cite[Chapter 3]{g14 1}).
\begin{definition}
The \emph{space ${\rm BMO}$} is defined to be the set of
all $f\in\mathscr{M}$ such that
\begin{align*}
\|f\|_{\rm BMO} := \sup_Q \fint_Q \left| f(x) - \fint_Q f(y)\,dy \right|\,dx < \infty,
\end{align*}
where the supremum is taken over all cubes $Q$ in $\mathbb{R}^n$.
\end{definition}
The following is an equivalent characterization of $A_\infty$ weights
in terms of ${\rm BMO}$ (see, for instance, \cite[Proposition 1.21]{hp13}).
\begin{lemma}\label{ww BMO}
For any scalar weight $\mathbb{W}$, $\mathbb{W} \in A_\infty$ if and only if $\log \mathbb{W} \in \rm{BMO}$;
moreover, $ \|\log \mathbb{W}\|_{\rm BMO} \leq \log(2e[\mathbb{W}]_{A_\infty})$.
\end{lemma}
Next, we give the proof of Proposition \ref{Apinfty Ainfty}.
\begin{proof}[Proof of Proposition \ref{Apinfty Ainfty}]
Let $\mathbb{W} \in A_\infty$
and $w := [\mathbb{W}(\cdot)]^{\frac{1}{p(\cdot)}}$.
For any cube $Q$ in $\mathbb{R}^n$,
let
$${\rm I}(Q) := \frac{1}{\|\mathbf{1}_Q\|_{L^{p(\cdot)}}} \left\|w\mathbf{1}_Q\right\|_{L^{p(\cdot)}}
\ \ \text{and}\ \
 {\rm II}(Q) := \exp\left( \fint_Q \frac{1}{p(y)} \log \mathbb{W}^{-1}(y) \,dy\right). $$
From Lemmas \ref{est Q} and \ref{ww BB},
we deduce that there exists a positive constant $A$,
depending only on $p(\cdot)$ and $n$, such that
\begin{align}\label{ww BB eq 4}
{\rm I}(Q) \lesssim [\mathbb{W}]_{A_\infty}^{A}
\max\left\{ \mathbb{W}(Q_0)^{\frac{1}{p_+} - \frac{1}{p_-}}, \mathbb{W}(Q_0)^{\frac{1}{p_-} - \frac{1}{p_+}} \right\}
 \left(\frac{\mathbb{W}(Q)}{|Q|}\right)^{\frac{1}{p_Q}}.
\end{align}

Now, by the definition of $p_Q$,
we obtain $ \fint_Q [\frac{1}{p(x)} - \frac{1}{p_Q}] \,dx= 0 $.
From this and Lemma \ref{ww BMO},
we infer that
\begin{align*}
{\rm II}(Q)
&= \exp\left( \fint_Q \frac{1}{p_Q} \log \mathbb{W}^{-1}(y) \,dy\right)\exp\left( \fint_Q \left[\frac{1}{p(y)}-\frac{1}{p_Q}\right] \log \mathbb{W}^{-1}(y) \,dy\right) \\
&= \exp\left( \fint_Q \frac{1}{p_Q} \log \mathbb{W}^{-1}(y) \,dy\right)\exp\left( \fint_Q \left[\frac{1}{p(y)}-\frac{1}{p_Q}\right] \left[\log \mathbb{W}^{-1}(y) - \left(\log \mathbb{W}^{-1}\right)_Q \right] \,dy\right) \\
&\quad \times \exp\left( \left(\log \mathbb{W}^{-1}\right)_Q \fint_Q \left[\frac{1}{p(y)}-\frac{1}{p_Q}\right]  \,dy\right) \\
&\leq \exp\left( \fint_Q \frac{1}{p_Q} \log \mathbb{W}^{-1}(y) \,dy\right)\exp\left( \fint_Q \left|\frac{1}{p(y)}-\frac{1}{p_Q}\right| \left|\log \mathbb{W}^{-1}(y) - \left(\log \mathbb{W}^{-1}\right)_Q \right| \,dy\right) \\
&\leq \exp\left( \fint_Q \frac{1}{p_Q} \log \mathbb{W}^{-1}(y) \,dy\right)\exp\left( \frac{2}{p_-} \fint_Q \left|\log \mathbb{W}^{-1}(y) - \left(\log \mathbb{W}^{-1}\right)_Q \right| \,dy\right) \\
&\leq \left( 2e[\mathbb{W}]_{A_\infty} \right)^{\frac{2}{p_-}} \exp\left( \fint_Q \frac{1}{p_Q} \log \mathbb{W}^{-1}(y) \,dy\right).
\end{align*}
Combining this with \eqref{ww BB eq 4},
we conclude that
\begin{align*}
{\rm I}(Q) \times {\rm II}(Q)
& \lesssim [\mathbb{W}]_{A_\infty}^{A + \frac{2}{p_-}}
\max\left\{ \mathbb{W}(Q_0)^{\frac{1}{p_+} - \frac{1}{p_-}}, \mathbb{W}(Q_0)^{\frac{1}{p_-} - \frac{1}{p_+}} \right\}\\
&\quad \times \left[ \fint_Q \mathbb{W}(x)\,dx\right]^{\frac{1}{p_Q}}  \left[ \exp\left( \fint_Q \log \mathbb{W}^{-1}(y) \,dy\right) \right]^\frac{1}{p_Q} \\
& \lesssim [\mathbb{W}]_{A_\infty}^{A + \frac{2}{p_-}}
\max\left\{ \mathbb{W}(Q_0)^{\frac{1}{p_+} - \frac{1}{p_-}}, \mathbb{W}(Q_0)^{\frac{1}{p_-} - \frac{1}{p_+}} \right\}
 [\mathbb{W}]_{A_\infty}^{\frac{1}{p_Q}},
\end{align*}
which, together with the definition of $\mathcal{A}_{p(\cdot),\infty}$,
further implies that $\mathbb{W}^{\frac{1}{p(\cdot)}} = w \in \mathcal{A}_{p(\cdot),\infty}$.
This finishes the proof of Proposition \ref{Apinfty Ainfty}.
\end{proof}

To give the proof of Theorem \ref{Apinfty A}, we also need
the following result on the convexification of $\mathcal{A}_{p(\cdot),\infty}$ weights.

\begin{lemma}\label{wr}
Let $p(\cdot) \in \mathcal{P}_0$.
Then, for any $r\in (0,\infty)$,
the weight $w \in \mathcal{A}_{p(\cdot),\infty}$
if and only if $w^r \in \mathcal{A}_{\frac{p(\cdot)}{r},\infty}$;
moreover, for any weight $w$,
$ [w]_{\mathcal{A}_{p(\cdot),\infty}} = [w^r]_{\mathcal{A}_{\frac{p(\cdot)}{r},\infty}}^{\frac1r}$.
\end{lemma}
\begin{proof}
Notice that,
by Lemma \ref{con f}, for  any cube $Q$ in $\mathbb{R}^n$,
\begin{align*}
&\frac{1}{\|\mathbf{1}_Q\|_{L^{p(\cdot)}}} \left\|w\mathbf{1}_Q\right\|_{L^{p(\cdot)}}
\exp\left( \fint_Q \log\left( w^{-1}(y) \right)\,dy \right)\\
&\quad = \left\{\frac{1}{\|\mathbf{1}_Q\|_{L^{\frac{p(\cdot)}{r}}}} \left\|w^r\mathbf{1}_Q\right\|_{L^{\frac{p(\cdot)}{r}}}
\exp\left( \fint_Q \log\left( w^{-r}(y) \right)\,dy \right)\right\}^\frac1r.
\end{align*}
This, combined with the definition of $ \mathcal{A}_{p(\cdot),\infty} $,
further implies that $w \in \mathcal{A}_{p(\cdot),\infty}$
if and only if $w^r \in\mathcal{A}_{\frac{p(\cdot)}{r},\infty} $
and $ [w]_{\mathcal{A}_{p(\cdot),\infty}} = [w^r]_{\mathcal{A}_{\frac{p(\cdot)}{r},\infty}}^{\frac1r}$.
This finishes the proof of Lemma \ref{wr}.
\end{proof}

Finally, we  prove Theorem \ref{Apinfty A}.

\begin{proof}[Proof of Theorem \ref{Apinfty A}]
We begin with assuming $p(\cdot) \in \mathcal{P}$.
It follows from Propositions \ref{reverse Holder w1} and \ref{Apinfty Ainfty}
that $ w\in\mathcal{A}_{p(\cdot),\infty} $ if and only if $w^{p(\cdot)} \in A_\infty$.

Next, we discuss about the case $p(\cdot)\in \mathcal{P}_0$.
Using Lemma \ref{wr},
we obtain $w \in \mathcal{A}_{p(\cdot),\infty}$ if and only if $w^{p_-} \in \mathcal{A}_{\frac{p(\cdot)}{p_-},\infty}$.
By the assumption that $p(\cdot) \in LH$,
we have $\frac{p(\cdot)}{p_-} \in \mathcal{P}$ and $\frac{p(\cdot)}{p_-} \in LH$.
From this and the previous discussion,
we deduce that $w^{p_-} \in \mathcal{A}_{\frac{p(\cdot)}{p_-},\infty}$
if and only if $\mathbb{W} := (w^{p_-})^{\frac{p(\cdot)}{p_-}} \in A_\infty $.
Hence, by this, we conclude that
$w \in \mathcal{A}_{p(\cdot),\infty}$ if and only if $\mathbb{W} \in A_\infty$.
This finishes the proof of Theorem \ref{Apinfty A}.
\end{proof}

\subsection{The Minimal Operator}\label{sec Apinfty 1}
In this subsection, as an application of Theorem \ref{Apinfty A},
we characterize the $\mathcal{A}_{p(\cdot),\infty}$ weights by the minimal operator.
Recall that the \emph{minimal operator $\mathfrak{m}$} was introduced by Cruz-Uribe and Neugebauer in \cite{cn95}
to  characterize the reverse H\"older class, which is defined by setting,
for any $f\in \mathscr{M}$ and $x \in \mathbb{R}^n$,
\begin{align}\label{def minimal}
\mathfrak{m}(f)(x) := \inf_{Q \ni x } \fint_Q \left|f(x)\right|\,dx,
\end{align}
where the infimum is taken over all cubes $Q$ containing $x$ in $\mathbb{R}^n$.

We now recall the concept  of $A_p$  weights
(see \cite{m72}).

\begin{definition}
Let $p\in [1,\infty)$.
A scalar weight $w$ on $\mathbb{R}^n$ is called an \emph{$A_p$-weight} if,
when $p\in (1,\infty)$,
\begin{align*}
[w]_{A_p} := \sup_{Q} \left[ \fint_Q w(x)\,dx \right]
\left\{ \fint_{Q} [w(x)]^{\frac{1}{1-p}}\,dx \right\}^{p-1}<\infty
\end{align*}
and, when $p = 1$,
\begin{align*}
[w]_{A_1} := \sup_{Q} \frac{1}{|Q|}\int_Q w(x)\,dx
\left[ \left\|w^{-1} \right\|_{L^\infty(Q)} \right]<\infty,
\end{align*}
where the suprema are taken over all cubes $Q$ in $\mathbb{R}^n$.
\end{definition}
\begin{remark}\label{rem Ap 1}
It is well known that $A_\infty = \bigcup_{p\in [1,\infty)}A_p$
and $A_p \subsetneqq A_\infty$ for any $p \in [1,\infty)$;
see \cite{cf74}.
\end{remark}

Next, we recall the concept of $\mathcal{A}_{p(\cdot)}$ weights
 introduced by Cruz-Uribe et al. \cite{cdh11}.

\begin{definition}\label{def Ap 1}
Let $p(\cdot) \in \mathcal{P}$. Then the \emph{weight class $\mathcal{A}_{p(\cdot)}$}
is defined to be the set of all scalar weights $w$ on $\mathbb{R}^n$
satisfying
\begin{align*}
[w]_{\mathcal{A}_{p(\cdot)}} := \sup_Q \frac{1}{|Q|} \left\|w\mathbf{1}_Q\right\|_{L^{p(\cdot)}}
\left\| w^{-1}\mathbf{1}_Q \right\|_{L^{p'(\cdot)}}< \infty,
\end{align*}
where the supremum is taken over all cubes $Q$ in $\mathbb{R}^n$.
\end{definition}

To give an equivalent description of $\mathcal{A}_{p(\cdot)}$ weights,
we recall the following concept (see \cite{cdh11}).
\begin{definition}
Let $p(\cdot) \in \mathcal{P}$.
Then \emph{weight class $A_{p(\cdot)}^{\dagger}$} is defined to be the set of all
scalar weights $w$ on $\mathbb{R}^n$ satisfying
\begin{align*}
[w]_{A_{p(\cdot)}^{\dagger}} := \sup_{Q} |Q|^{p_Q} \left\|w\mathbf{1}_Q\right\|_{L^1} \left\|w^{-1}\mathbf{1}_Q\right\|_{L^{\frac{p'(\cdot)}{p(\cdot)}}} < \infty,
\end{align*}
where the supremum is taken over all cubes $Q$ in $\mathbb{R}^n$.
\end{definition}

We have the following basic properties of $A^\dagger_{p(\cdot)}$,
where Lemma \ref{rem Ap}{\rm (i)} is precisely \cite[Lemma 17]{wgx24} and
Lemma \ref{rem Ap}{\rm (ii)} is a direct consequence of \cite[Theorems 1.1 and 1.3]{cdh11}
(see also \cite[p.\,187]{cf13}).

\begin{lemma}\label{rem Ap}
If $p(\cdot) \in \mathcal{P}$ with $p(\cdot) \in LH$ and $p_- > 1$,
then
\begin{itemize}
\item[{\rm (i)}]
$ A_{p_-} \subset A_{p(\cdot)}^{\dagger} \subset A_{p_+} $,
\item[{\rm (ii)}]
for any weight $w$,
$w \in \mathcal{A}_{p(\cdot)}$ if and only if $w^{p(\cdot)} \in A^{\dagger}_{p(\cdot)}$.
\end{itemize}
\end{lemma}

On the relationships  among
$\mathcal{A}_{p(\cdot)}$, $\mathcal{A}_{p(\cdot),\infty}$, and $A_\infty$,
we have the following result.

\begin{proposition}\label{Ap Apinfty 2}
Let $p(\cdot) \in \mathcal{P}$ with $p(\cdot) \in LH$ and $p_- > 1$.
Then $\mathcal{A}_{p(\cdot)} \subsetneqq \mathcal{A}_{p(\cdot),\infty} \subsetneqq A_\infty$.
\end{proposition}
\begin{proof}
By Lemma \ref{Ap Apinfty 1},
we obtain $\mathcal{A}_{p(\cdot),\infty} \subset A_\infty $.
Now, let $w \in \mathcal{A}_{p(\cdot)}$.
Then, using Jensen's inequality and Lemma \ref{est fQ},
we find that, for any cube $Q$ in $\mathbb{R}^n$,
\begin{align*}
&\frac{1}{\|\mathbf{1}_Q\|_{L^{p(\cdot)}}} \left\| w \mathbf{1}_Q \right\|_{L^{p(\cdot)}} \exp\left( \fint_Q \log\left( w^{-1}(x) \right)\,dx \right)\\
&\quad\leq \frac{1}{\|\mathbf{1}_Q\|_{L^{p(\cdot)}}} \left\| w \mathbf{1}_Q \right\|_{L^{p(\cdot)}}  \fint_Q  w^{-1}(x) \,dx
\lesssim \frac{1}{|Q|} \left\| w \mathbf{1}_Q \right\|_{L^{p(\cdot)}} \left\|w^{-1} \mathbf{1}_Q \right\|_{L^{p'(\cdot)}}
\leq [w]_{\mathcal{A}_{p(\cdot)}},
\end{align*}
which further implies that $w \in \mathcal{A}_{p(\cdot),\infty}$.
Thus, $\mathcal{A}_{p(\cdot)} \subset \mathcal{A}_{p(\cdot),\infty}$.

Next, we show that they are both proper contained.
We first prove $\mathcal{A}_{p(\cdot)} \subsetneqq \mathcal{A}_{p(\cdot),\infty}$.
From Remark \ref{rem Ap 1} and Lemma \ref{rem Ap}{\rm (i)},
we infer immediately that $A^{\dagger}_{p(\cdot)} \subsetneqq A_\infty$.
Using this, Theorem \ref{Apinfty A}, and Lemma \ref{rem Ap}{\rm (ii)},
we conclude that $\mathcal{A}_{p(\cdot)} \subsetneqq \mathcal{A}_{p(\cdot),\infty}$.

Finally, we prove $ \mathcal{A}_{p(\cdot),\infty} \subsetneqq A_\infty $.
Indeed, from \cite[Example 7.1.7]{g14},
it follows that the power function $|x|^a$ belongs to $A_\infty$
if and only if $a \in (-n,\infty)$.
Let $a_0 \in (-n, -\frac{n}{p_-})$
and $w(x) := |x|^{a_0}$ for any $x\in\mathbb{R}^n$.
Then $w \in A_\infty$.
However, since $a_0 \in (-n,-\frac{n}{p_-})$,
we deduce that $[w(x)]^{p(x)} \geq |x|^{-n}$ for any $|x| \in (0,1)$
and hence $[w(\cdot)]^{p(\cdot)}$ is not integrable over the cube $Q(\mathbf{0},1)$.
This further implies $[w(\cdot)]^{p(\cdot)} \notin A_\infty$.
Thus, using this and Theorem \ref{Apinfty A},
we conclude that $w \notin \mathcal{A}_{p(\cdot),\infty}$
and hence $ \mathcal{A}_{p(\cdot),\infty} \subsetneqq A_\infty $.
This finishes the proof of Proposition \ref{Ap Apinfty 2}.
\end{proof}

Now, we recall the concept of weighted variable Lebesgue spaces (see \cite{cdh11}).
\begin{definition}\label{def weight Leb}
Let $p(\cdot)\in \mathcal{P}_0$ and $w$ be any scalar weight on $\mathbb{R}^n$.
Then the \emph{weighted variable Lebesgue space $L^{p(\cdot)}_w$} is defined
to be the set of all $f\in \mathscr{M}$ such that
$\|f\|_{L^{p(\cdot)}_w} :=\|fw\|_{L^{p(\cdot)}} < \infty.$
\end{definition}

The main result of this subsection  is the following characterization of $\mathcal{A}_{p(\cdot),\infty}$ weights
in terms of the minimal operator.

\begin{theorem}\label{inf Apinfty}
Let $p(\cdot) \in \mathcal{P}$ with $p(\cdot) \in LH$ and $p_-> 1$.
Then, for any scalar weight $w$,
$ w \in \mathcal{A}_{p(\cdot),\infty} $ if and only if, for any $f\in \mathscr{M}$
satisfying $f^{-1} \in L^{p(\cdot)}_w$,
\begin{align}\label{eq inf Apinfty}
\left\| \left[\mathfrak{m}(f)\right]^{-1} \right\|_{L^{p(\cdot)}_w} \leq C \left\| f^{-1} \right\|_{L^{p(\cdot)}_w},
\end{align}
where $C$ is a positive constant independent of $f$.
\end{theorem}
\begin{remark}
We note that, when $p(\cdot)$ is a constant exponent,
the necessary of Theorem \ref{inf Apinfty} reduces to \cite[Theorem 3.1]{cn95}.
Moreover, in this case, since \cite[Theorme 3.1]{cn95} only shows $w \in A_\infty$
and $\mathcal{A}_{p,\infty} \subsetneqq A_\infty$,
it follows that Theorem \ref{inf Apinfty} is a little bit
stronger than \cite[Theorem 3.1]{cn95}.
\end{remark}

To prove Theorem \ref{inf Apinfty},
we   recall that the \emph{Hardy--Littlewood maximal operator $\mathcal{M}$}
is defined by setting, for any $f\in \mathscr{M}$ and $x \in \mathbb{R}^n$,
\begin{align*}
\mathcal{M}(f)(x) := \sup_{Q \ni x } \fint_Q \left|f(x)\right|\,dx,
\end{align*}
where the supremum is taken over all cubes $Q$ containing $x$.
The following result is the boundedness of the Hardy--Littlewood maximal operator on
weighted variable Lebesgue spaces (see \cite[Theorem 1.3]{cdh11}).
\begin{lemma}\label{M Ap}
Let $p(\cdot)\in \mathcal{P}$ with $p(\cdot) \in LH$ and $p_- > 1$.
Then, for any weight $w$,
$w \in \mathcal{A}_{p(\cdot)}$ if and only if the Hardy--Littlewood maximal operator $\mathcal{M}$
is bounded on $L^{p(\cdot)}_w$.
\end{lemma}
Next, we begin to prove Theorem \ref{inf Apinfty}.
\begin{proof}[Proof of Theorem \ref{inf Apinfty}]
We first consider the necessity.
Let $w \in \mathcal{A}_{p(\cdot)},\infty$.
By Theorem \ref{Apinfty A}, we find that $\mathbb{W} := [w(\cdot)]^{p(\cdot)} \in A_\infty$.
Using this and Remark \ref{rem Ap 1},
we conclude that there exists $r_0 \in [1,\infty)$
such that $\mathbb{W} \in A_{r_0p_-}$.
From this and Lemma \ref{rem Ap},
we infer that $ \mathbb{W} \in A_{r_0p(\cdot)}^{\dagger} $ and hence
$ w^{\frac{1}{r_0}} = [\mathbb{W}(\cdot)]^{\frac{1}{r_0p(\cdot)}} \in \mathcal{A}_{r_0p(\cdot)} $.

Now, let $s := r_0 + 1$.
Then, by H\"older's inequality,
we find that, for any $f\in \mathscr{M}$ and any cube $Q$ in $\mathbb{R}^n$,
\begin{align*}
1 &= \fint_Q 1\,dy = \fint_{Q} |f(y)|^{\frac1s} |f(y)|^{-\frac1s}\,dy
 \leq \left[\fint_{Q} |f(y)|\,dy\right]^\frac{1}{s} \left[ \fint_Q |f(y)|^{-\frac{s'}{s}}\,dy\right]^\frac{1}{s'}\\
& = \left[\fint_{Q} |f(y)|\,dy\right]^\frac{1}{s} \left[ \fint_Q |f(y)|^{-\frac{1}{r_0}}\,dy\right]^\frac{1}{s'},
\end{align*}
which, together with raising both sides to $s$ power, further implies that
\begin{align*}
 \left[\fint_{Q} |f(y)|\,dy\right]^{-1} \leq \left[ \fint_Q |f(y)|^{-\frac{1}{r_0}}\,dy\right]^{r_0}.
\end{align*}
Thus, for any $x\in \mathbb{R}^n$,
using this with taking the supremum over all cubes $Q$ containing $x$,
we conclude that
\begin{align*}
\frac{1}{\mathfrak{m}f(x)} = \sup_{Q} \left[\fint_{Q} |f(y)|\,dy\right]^{-1} \leq \left[ \mathcal{M}\left(|f|^{-\frac{1}{r_0}}\right)(x)\right]^{r_0}.
\end{align*}
From this and Lemmas \ref{con f} and \ref{M Ap},
it follows that
\begin{align*}
\left\| \frac{w}{\mathfrak{m}f} \right\|_{L^{p(\cdot)}}
&\leq \left\| \mathcal{M}\left(|f|^{-\frac{1}{r_0}}\right)^{r_0} w \right\|_{L^{p(\cdot)}}
 = \left\| \mathcal{M}\left(|f|^{-\frac{1}{r_0}}\right) w^{\frac{1}{r_0}} \right\|_{L^{\frac{p(\cdot)}{r_0}}}^{\frac{1}{r_0}}
 \lesssim \left\| |f|^{-\frac{1}{r_0}} w^{\frac{1}{r_0}} \right\|_{L^{\frac{p(\cdot)}{r_0}}}^{\frac{1}{r_0}}
 = \left\| \frac{w}{f} \right\|_{L^{p(\cdot)}}.
\end{align*}
This finishes the proof of the necessity.

Next, we prove the sufficiency.
Fix a cube $Q$ in $\mathbb{R}^n$ and let $f := \frac{w}{\mathbf{1}_Q}$.
Then $\frac{1}{f} \in L^{p(\cdot)}_w$.
Using this and the definition of the minimal operator,
we find that, for any $x \in Q$
\begin{align*}
\mathfrak{m}f(x) \leq \fint_Q w(x)\,dx.
\end{align*}
Combining this with \eqref{eq inf Apinfty},
we obtain
\begin{align*}
\left[ \fint_Q w(x)\,dx\right]^{-1} \left\| w \mathbf{1}_Q\right\|_{L^{p(\cdot)}}
\leq \left\| \frac{w}{\mathfrak{m}(f)} \right\|_{L^{p(\cdot)}} \lesssim \left\|\frac{w}{f}\right\|_{L^{p(\cdot)}}
 = \left\|\mathbf{1}_Q\right\|_{L^{p(\cdot)}},
\end{align*}
which further implies that
\begin{align*}
\frac{1}{\|\mathbf{1}_Q\|_{L^{p(\cdot)}}} \left\|w\mathbf{1}_Q\right\|_{L^{p(\cdot)}} \lesssim \fint_Q w(x)\,dx.
\end{align*}
From this and Theorem \ref{w reverse},
we deduce that $w \in \mathcal{A}_{p(\cdot),\infty}$.
This finishes the proof of the sufficiency
and hence Theorem \ref{inf Apinfty}.
\end{proof}

\subsection{Reverse H\"older's inequality for $\mathcal{A}_{p(\cdot),\infty}$}
\label{sec Apinfty}

In this subsection, we first establish the reverse H\"older's inequality
for $\mathcal{A}_{p(\cdot),\infty}$ weights in variable Lebesgue spaces (see Theorem \ref{reverse Holder Apinfty}),
and then show that the assertion that the reverse H\"older's inequality
for $w$ in variable Lebesgue spaces holds is equivalent to $w \in \mathcal{A}_{p(\cdot),\infty}$
(see Theorem \ref{Apinfty Holder}).

The reverse H\"older's inequality for $\mathcal{A}_{p(\cdot),\infty}$ weights
in variable Lebesgue spaces is as follows.

\begin{theorem}\label{reverse Holder Apinfty}
Let $p(\cdot)\in\mathcal{P}_0$ with $p(\cdot) \in LH$.
Then, for any $w\in \mathcal{A}_{p(\cdot),\infty}$,
there exist positive constants $C$, $C_1$, $C_2$,
$A$, and $A_1$,
depending only on $p(\cdot)$ and $n$,
such that, for any $r \in (1,r_w]$
with
$$ r_w := 1+ \frac{1}{C_1 [w]_{\mathcal{A}_{p(\cdot),\infty}}^{A_1} 2^{C_2 [w]_{\mathcal{A}_{p(\cdot),\infty}}}} $$
and for any cube $Q$ in $\mathbb{R}^n$,
\begin{align}\label{reverse Holder Apinfty 1}
\frac{1}{\|\mathbf{1}_{Q}\|_{L^{rp(\cdot)}}} \left\|w\mathbf{1}_Q\right\|_{L^{rp(\cdot)}} \leq C [w]_{\mathcal{A}_{p(\cdot),\infty}}^{A} \frac{1}{\|\mathbf{1}_{Q}\|_{L^{p(\cdot)}}} \left\|w\mathbf{1}_Q\right\|_{L^{p(\cdot)}}.
\end{align}
\end{theorem}
\begin{remark}\label{rem reverse Holder Apinfty}
We point out that the concept of $\mathcal{A}_{p(\cdot),\infty}$ weights
is essential for the reverse H\"older's inequality in variable Lebesgue spaces.
Indeed, although we can also establish a reverse H\"older's inequality in variable Lebesgue spaces
with just using the  $A_\infty$ weights,
but this is not our desired result.
To be more precise, for any given $p(\cdot) \in \mathcal{P}_0$ and for any weight $\mathbb{W}$,
we define an \emph{alternative weighted variable Lebesgue space
$\widetilde{L}^{p(\cdot)}_{\mathbb{W}}$}
(see \cite[(2.2)]{dh08} for more details),
which consists of all $f \in \mathscr{M}$
satisfying
\begin{align*}
\|f\|_{\widetilde{L}^{p(\cdot)}_{\mathbb{W}}}
:= \inf\left\{ \lambda \in (0,\infty):  \int_{\mathbb{R}^n} \left|\frac{f(x)}{\lambda}\right|^{p(x)} \mathbb{W}(x)\,dx < 1 \right\}
< \infty.
\end{align*}
Then, for any given $p(\cdot) \in \mathcal{P}_0$,
any weight $\mathbb{W}$,
and any $f\in \mathscr{M}$,
$\|f\|_{\widetilde{L}^{p(\cdot)}_{\mathbb{W}}} = \|f\|_{L^{p(\cdot)}_w}$,
where $w := [\mathbb{W}(\cdot)]^{\frac{1}{p(\cdot)}}$ and $L^{p(\cdot)}_w$ is the weighted variable Lebesgue space
in Definition \ref{def weight Leb}.
By this and Theorems \ref{Apinfty A} and \ref{reverse Holder Apinfty},
we find that, for any $\mathbb{W} \in A_\infty$,
there exist positive constants $C$ and $r\in(1,\infty)$,
depending on $n$, $p(\cdot)$, and $\mathbb{W}$,
such that, for any cube $Q$ in $\mathbb{R}^n$,
\begin{align}\label{reverse Holder WW}
\frac{1}{\|\mathbf{1}_Q\|_{L^{rp(\cdot)}}} \left\| \mathbf{1}_Q \right\|_{\widetilde{L}^{rp(\cdot)}_{\mathbb{W}^r}}
\leq C \frac{1}{\|\mathbf{1}_Q\|_{L^{p(\cdot)}}} \left\| \mathbf{1}_Q \right\|_{\widetilde{L}^{p(\cdot)}_{\mathbb{W}}}.
\end{align}
However, inequality \eqref{reverse Holder WW} is unsatisfactory
because  the constants $C$ and $r$ may depend not only on
the weight constant  $[\mathbb{W}]_{A_\infty}$, but also on $\mathbb{W}$ itself.
Indeed, if $p(\cdot)$ is not a constant exponent,
there may  not exist positive functions
$\varphi$ and $\phi$ such that
$C = \varphi([\mathbb{W}]_{A_\infty})$
and $r = \phi([\mathbb{W}]_{A_\infty})$.

To see this, we give an example on $\mathbb{R}$.
Let
\begin{align*}
p(x) :=
\begin{cases}
\displaystyle 2 & \text{if}\ x \in (-\infty,-1)\cup [1,\infty), \\
\displaystyle 1-x & \text{if}\ x \in [-1,0), \\
\displaystyle 1+x & \text{if}\ x \in [0,1).
\end{cases}
\end{align*}
Then $p(\cdot) \in  \mathcal{P}_0 \cap LH$.
Now, assume that \eqref{reverse Holder WW} holds  with $C = \varphi([\mathbb{W}]_{A_\infty})$
and $r = \phi([\mathbb{W}]_{A_\infty})$, where $\varphi$ and $\phi$ are some positive functions.
Then, letting $\mathbb{W}_k := k$ for any $k\in\mathbb{N}$,
it is easy to find that $\mathbb{W}_k \in A_\infty $
and $[\mathbb{W}_k]_{A_\infty} = 1$ for any $k\in\mathbb{N}$.
Applying \eqref{reverse Holder WW}  to $\mathbb{W}_k$, we find that, for any $k \in \mathbb{N}$ and any interval $Q$,
\begin{align}\label{reverse Holder WW 1}
\frac{1}{\|\mathbf{1}_Q\|_{L^{rp(\cdot)}}} \left\| \mathbf{1}_Q \right\|_{\widetilde{L}^{rp(\cdot)}_{\mathbb{W}_k^r}}
\leq C \frac{1}{\|\mathbf{1}_Q\|_{L^{p(\cdot)}}} \left\| \mathbf{1}_Q \right\|_{\widetilde{L}^{p(\cdot)}_{\mathbb{W}_k}}.
\end{align}
Fixing the cube $Q := [0,1]$ and writing $\lambda_k := \| \mathbf{1}_Q \|_{\widetilde{L}^{p(\cdot)}_{\mathbb{W}_k}}$
for any $k\in\mathbb{N}$,
by using \eqref{reverse Holder WW 1},
we obtain
$\| \lambda_k^{-1} \mathbf{1}_Q \|_{\widetilde{L}^{rp(\cdot)}_{\mathbb{W}_k}} \lesssim 1 $
for any $k\in\mathbb{N}$,
which, combined with Lemma \ref{rhof 1},
further implies that
\begin{align}\label{reverse Holder WW 2}
\sup_{k\in\mathbb{N}} \int_Q \lambda_k^{-rp(x)} [\mathbb{W}_k(x)]^r \,dx
= \sup_{k\in\mathbb{N}} \int_0^1 \lambda_k^{-r(x + 1)} k^r \,dx < \infty.
\end{align}
From Lemma \ref{rhof 1},
we infer that, for any $k \in\mathbb{N}$,
\begin{align*}
1 = \int_0^1 \lambda_k^{-x-1} k\,dx
 = \frac{k}{\lambda_k \log \lambda_k}\left(1-\frac{1}{\lambda_k}\right),
\end{align*}
which further implies that
$ k = \frac{\lambda_k^2 \log \lambda_k}{\lambda_k - 1}$.
Taking this into the integral in the right-hand side of \eqref{reverse Holder WW 2},
we find that, for any $k \in\mathbb{N}$,
\begin{align}\label{reverse Holder WW 3}
\int_0^1 \lambda_k^{-r(x + 1)} k^r \,dx
& = \int_0^1 \lambda_k^{-r(x + 1)} \frac{\lambda_k^{2r} (\log \lambda_k)^r}{(\lambda_k - 1)^r} \,dx
= \frac{\lambda_k^r (\log \lambda_k)^r}{(\lambda_k - 1)^r} \int_0^1 \lambda_k^{-rx}  \,dx \nonumber\\
& = \frac1r\left(\log \lambda_k\right)^{r - 1} \left(\frac{\lambda_k}{\lambda_k - 1}\right)^r
\left( 1 - \frac{1}{\lambda_k^r}\right).
\end{align}
Notice that, by Lemma \ref{rhof 1},
we have $\lambda_k \geq k^{\frac{1}{p_+}} \to \infty$
as $k \to \infty$.
Using this, the fact $r>1$, and \eqref{reverse Holder WW 3},
we conclude that
$$\sup_{k\in\mathbb{N}} \int_Q \lambda_k^{-rp(x)} [\mathbb{W}_k(x)]^r \,dx = \infty, $$
which contradicts \eqref{reverse Holder WW 2}
and shows that such $\varphi$ and $\psi$ do not exist.
This finishes the proof of the above claim.
\end{remark}
Before giving the proof of Theorem \ref{reverse Holder Apinfty},
we first establish the reverse H\"older's inequality for $\mathbb{W} := [w(\cdot)]^{p(\cdot)}$
with $w \in \mathcal{A}_{p(\cdot),\infty}$,
whose proof is based on the proof of \cite[Theorem 2.3]{hp13}
with \cite[Lemma 6.6]{hp13} replaced by Proposition \ref{reverse Holder w1}.
For the completeness of the article,
we give a proof here.
In what follows, $M_d$ denotes the \emph{dyadic maximal operator},
which is defined by setting, for any $f\in \mathscr{M}$ and $x\in\mathbb{R}^n$,
\begin{align*}
M_d(f)(x) := \sup_{Q \ni x } \fint_Q \left|f(x)\right|\,dx,
\end{align*}
where the supremum is taken over all dyadic cubes $Q$ in $\mathbb{R}^n$.
\begin{lemma}\label{reverse Holder w}
Let $p(\cdot)\in \mathcal{P}$ with $p(\cdot) \in LH$
and let $w\in\mathcal{A}_{p(\cdot),\infty}$.
Then $\mathbb{W} := [w(\cdot)]^{p(\cdot)}$ is an $A_\infty$ weight.
Moreover, there exists positive constants $C_1$, $C_2$, and $A_1$,
depending only on $n$ and $p(\cdot)$,
such that, for any $r \in (1, r_w]$ with
\begin{align}\label{def rw}
r_w := 1+ \frac{1}{C_1 L_w [w]_{\mathcal{A}_{p(\cdot),\infty}}^{A_1} 2^{C_2[w]_{\mathcal{A}_{p(\cdot),\infty}}}},
\end{align}
where $L_w$ is as in \eqref{def Lw} and $Q_0$ as in \eqref{def Q0}
and, for any cube $Q$ in $\mathbb{R}^n$,
\begin{align*}
\fint_Q \mathbb{W}(x)^{r}\,dx \leq 2\left[ \fint_Q \mathbb{W}(x)\,dx\right]^r.
\end{align*}
\end{lemma}
\begin{proof}
Fix a cube $Q$ in $\mathbb{R}^n$.
By the homogeneity of $\fint_Q \mathbb{W}(x)\,dx $,
we only need to consider the case $\fint_Q \mathbb{W}(x)\,dx = 1$.
Similar to the proof of \cite[Theorem 2.3]{hp13},
we obtain
\begin{align}\label{akw 3}
\int_Q \mathbb{W}(x)^{1 + \epsilon}\,dx
\leq \int_Q \left[M_d\left( \mathbb{W}\mathbf{1}_Q\right)(x)\right]^\epsilon\mathbb{W}(x)\,dx
\leq |Q| + \epsilon a^\epsilon \log a \sum_{k = 0}^\infty a^{k\epsilon}\mathbb{W}\left(\Omega_k\right),
\end{align}
where $a \in (1,\infty)$ and $\epsilon \in (0,\infty)$ are two constants to be determined later
and
$$\Omega_k :=\left\{x \in Q:  M_d(\mathbb{W}\mathbf{1}_Q)(x) > a^k\right\}.$$

Using the Calder\'on--Zygmund decomposition of $\mathbb{W}$ adapted to $Q$,
we find that,
for any $k\in\mathbb{Z}_+$,
there exists a family of maximal nonoverlapping dyadic cubes $\{Q_{k,j}\}_{k,j\in\mathbb{Z}_+}$
strictly contained in $Q$
for which $\Omega_k = \bigcup_{j\in\mathbb{Z}_+} Q_{k,j}$ and
$a^k < \fint_{Q_{k,j}} \mathbb{W}(x)\,dx \leq 2^n a^k$.
Then, by this and the disjointness of $\{Q_{k,j}\}_{j\in\mathbb{Z}_+}$ for any $k\in\mathbb{Z}_+$,
we obtain
\begin{align}\label{akw 1}
\sum_{k = 0}^\infty a^{k\epsilon}\mathbb{W}(\Omega_k)
= \sum_{k = 0}^\infty \sum_{j = 0}^\infty a^{k\epsilon}\mathbb{W}\left(Q_{k,j}\right)
\leq \sum_{k = 0}^\infty \sum_{j = 0}^\infty \left[ \fint_{Q_{k,j}} \mathbb{W}(x)\,dx \right]^{\epsilon}\mathbb{W}\left(Q_{k,j}\right).
\end{align}
For any $k,j\in\mathbb{Z}_+$,
let $E_{k,j} := Q_{k,j} \setminus \Omega_{k+1}$.
Then, similar to the estimation of \cite[(6.9)]{hp13},
we find that, for any $a \in (2^n,\infty)$ and $k,j\in\mathbb{Z}_+$,
$| Q_{k,j} | < \frac{a}{a-2^n} | E_{k,j} |$,
which further implies that
\begin{align}\label{QE a}
\left( 1 - 2^{n}a^{-1} \right)\left| Q_{k,j} \right| < \left| E_{k,j} \right|.
\end{align}
Next, by Proposition \ref{reverse Holder w1},
there exist positive constants $C$ and $A$,
depending only on $n$ and $p(\cdot)$,
such that, for any $\alpha \in (0,2^{-\frac{\delta}{\delta - 1}})$,
any cube $Q$ in $\mathbb{R}^n$, and any measurable set $E\subset Q$ with $|E| > (1-\alpha)|Q|$,
\begin{align}\label{wEQ EQ 3}
\mathbb{W}(Q) < C \left( 1 - 2\alpha^{1-\frac1\delta} \right)^{-p_+} [w]_{\mathcal{A}_{p(\cdot),\infty}}^{A}
L_w\mathbb{W}(E),
\end{align}
where $\delta$ and $L_w$ are the same as in Proposition \ref{reverse Holder w1}.
We fix $a := 2^{n+\frac{2\delta}{\delta - 1}}$
and $\alpha := 2^{n}a^{-1} = 2^{-\frac{2\delta}{\delta - 1}}$.
Combining this with both \eqref{QE a} and \eqref{wEQ EQ 3} yields
\begin{align}\label{akw 2}
\mathbb{W}(Q) &<C \left[ 1 - 2(2^na^{-1})^{1-\frac1\delta} \right]^{-p_+}  [w]_{\mathcal{A}_{p(\cdot),\infty}}^{A}
L_w\mathbb{W}(E)
 = C2^{p_+} [w]_{\mathcal{A}_{p(\cdot),\infty}}^{A}
L_w \mathbb{W}(E)
=: C_w \mathbb{W}(E),
\end{align}
where
$C_w := C 2^{p_+}  [w]_{\mathcal{A}_{p(\cdot),\infty}}^{A}L_w. $
Thus, by taking \eqref{akw 2} into \eqref{akw 1},
we find that
\begin{align*}
\sum_{k = 0}^\infty a^{k\epsilon}\mathbb{W}(\Omega_k)
&\leq \sum_{k = 0}^\infty \sum_{j = 0}^\infty \left[ \fint_{Q_{k,j}} \mathbb{W}(x)\,dx \right]^{\epsilon}\mathbb{W}\left(Q_{k,j}\right)
\leq C_w \sum_{k = 0}^\infty \sum_{j = 0}^\infty \left[ \fint_{Q_{k,j}} \mathbb{W}(x)\,dx \right]^{\epsilon}\mathbb{W}\left(E_{k,j}\right)\\
&\leq C_w \sum_{k = 0}^\infty \sum_{j = 0}^\infty \int_{E_{k,j}} \left[M_d\left(\mathbb{W}\mathbf{1}_Q\right)(x)\right]^{\epsilon}\mathbb{W}(x)\,dx
\leq C_w \int_{Q} \left[M_d\left(\mathbb{W}\mathbf{1}_Q\right)(x)\right]^{\epsilon}\mathbb{W}(x)\,dx.
\end{align*}
Combining this with \eqref{akw 3},
we obtain
\begin{align}\label{akw 4}
\fint_Q \left[M_d\left( \mathbb{W}\mathbf{1}_Q \right)(x)\right]^{\epsilon}\mathbb{W}(x)\,dx
\leq 1 + C_w \epsilon a^\epsilon \log a \int_{Q} \left[M_d\left(\mathbb{W}\mathbf{1}_Q\right)(x)\right]^{\epsilon}\mathbb{W}(x)\,dx,
\end{align}
where $\epsilon$ is a positive constant to be determined later.
Now, let
\begin{align*}
\epsilon_1 :=\,& \frac{1}{C_w  (n + \frac{2\delta}{\delta-1}) \log(2)
2^{n+1+\frac{2\delta}{\delta-1}}} \\
=\,& \frac{1}{C2^{p_+} L_w [w]_{\mathcal{A}_{p(\cdot),\infty}}^{A}(n + 2 + 2\tau_n C_{p(\cdot),n}[w]_{\mathcal{A}_{p(\cdot),\infty}} )\log(2) 2^{n + 2 + 2\tau_n C_{p(\cdot)}[w]_{\mathcal{A}_{p(\cdot),\infty}}}},
\end{align*}
where $\delta$ and $C_w$ are the same as in \eqref{wEQ EQ 3}.
Then
\begin{align*}
C_w \epsilon_1 a^{\epsilon_1} \log a
& = C_w \frac{1}{C_w  (n + \frac{2\delta}{\delta-1}) \log(2) 2^{n+\frac{2\delta}{\delta-1}}}
2^{\frac{n + \frac{2\delta}{\delta-1}}{C_w  (n + \frac{2\delta}{\delta-1}) \log(2) 2^{n+\frac{2\delta}{\delta-1}}}}
\left(n + \frac{2\delta}{\delta-1}\right)\log(2)\\
& = 2^{-(n+\frac{2\delta}{\delta-1})}
2^{\frac{1}{C_w   \log(2) 2^{n+\frac{2\delta}{\delta-1}}}}
< \frac12.
\end{align*}
From Lemma \ref{Ap Apinfty 1} and $w\in \mathcal{A}_{p(\cdot),\infty}$,
it follows that  $C_{p(\cdot),n}[w]_{\mathcal{A}_{p(\cdot),\infty}}\geq 1$
and hence, by letting $C_2 := (n + 2 +2\tau_n)C_{p(\cdot),n}$,
we obtain
$n + 2 + 2\tau_n C_{p(\cdot),n}[w]_{\mathcal{A}_{p(\cdot),\infty}} \leq C_2[w]_{\mathcal{A}_{p(\cdot),\infty}}$.
Thus, from this and the definition of $\epsilon_1$, we deduce that
$$
\epsilon := \frac{1}{C2^{p_+}\log(2) L_w [w]_{\mathcal{A}_{p(\cdot),\infty}}^{A}(C_2[w]_{\mathcal{A}_{p(\cdot),\infty}} ) 2^{C_2[w]_{\mathcal{A}_{p(\cdot),\infty}}}}
\leq \epsilon_1.
$$
and hence
$$ C_w \epsilon a^\epsilon \log a \leq C_w \epsilon_1 a^{\epsilon_1} \log a < \frac12. $$
Thus, using this with taking $\epsilon$ into \eqref{akw 4},
we find that
\begin{align*}
\fint_Q \left[M_d\left( \mathbb{W}\mathbf{1}_Q \right)(x)\right]^{\epsilon}\mathbb{W}(x)\,dx
\leq 1 + \frac12 \int_{Q} \left[M_d\left(\mathbb{W}\mathbf{1}_Q\right)(x)\right]^{\epsilon}\mathbb{W}(x)\,dx.
\end{align*}
Hence
\begin{align*}
\int_{Q} \left[M_d\left(\mathbb{W}\mathbf{1}_Q\right)(x)\right]^{\epsilon}\mathbb{W}(x)\,dx \leq 2,
\end{align*}
which, together with \eqref{akw 3},
further implies that
\begin{align*}
\fint_Q \mathbb{W}(x)^{1+\epsilon}
\leq \int_{Q} \left[M_d\left(\mathbb{W}\mathbf{1}_Q\right)(x)\right]^{\epsilon}\mathbb{W}(x)\,dx
\leq 2.
\end{align*}
This finishes the proof of Lemma \ref{reverse Holder w}.
\end{proof}
Next, we give the proof of Theorem \ref{reverse Holder Apinfty}.

\begin{proof}[Proof of Theorem \ref{reverse Holder Apinfty}]
We first claim that we only need to prove the present theorem
in the case $ p_- \geq 1 $.
Indeed, assume that \eqref{reverse Holder Apinfty 1} holds
for any $\mathcal{A}_{q(\cdot),\infty}$ weight $w$ with $q(\cdot) \in \mathcal{P}$ and $q(\cdot) \in LH$.
For any $w\in \mathcal{A}_{p(\cdot),\infty}$ with $p(\cdot) \in \mathcal{P}_0$ and $p(\cdot) \in LH$,
by Lemma \ref{wr}, we find that $w^{p_-} \in \mathcal{A}_{\frac{p(\cdot)}{p_-},\infty}$
and $[w^{p_-}]_{\mathcal{A}_{\frac{p(\cdot)}{p_-},\infty}} = [w]_{\mathcal{A}_{p(\cdot),\infty}}$,
where $\frac{p(\cdot)}{p_-} \in \mathcal{P}$ and $\frac{p(\cdot)}{p_-} \in LH$.
From this and the above assumption,
we infer that there exist positive constants $C$ and $A$,
depending only on $\frac{p(\cdot)}{p_-}$ and $n$,
such that, for any cube $Q$ in $\mathbb{R}^n$,
\begin{align*}
\frac{1}{\|\mathbf{1}_{Q}\|_{L^{\frac{rp(\cdot)}{p_-}}}} \left\|w^{p_-} \mathbf{1}_Q\right\|_{L^{\frac{rp(\cdot)}{p_-}}}
\leq C [w^{p_-}]_{\mathcal{A}_{\frac{p(\cdot)}{p_-},\infty}}^{A}
\frac{1}{\|\mathbf{1}_{Q}\|_{L^{\frac{p(\cdot)}{p_-}}}} \left\|w^{p_-}\mathbf{1}_Q\right\|_{L^{\frac{p(\cdot)}{p_-}}},
\end{align*}
which, combined with Lemma \ref{wr}, further implies that
\begin{align*}
\frac{1}{\|\mathbf{1}_{Q}\|_{L^{rp(\cdot)}}} \left\|w\mathbf{1}_Q\right\|_{L^{rp(\cdot)}}
&\leq C^{\frac{1}{p_-}} [w^{p_-}]_{\mathcal{A}_{\frac{p(\cdot)}{p_-},\infty}}^{\frac{A}{p_-}}
\frac{1}{\|\mathbf{1}_{Q}\|_{L^{p(\cdot)}}} \left\|w\mathbf{1}_Q\right\|_{L^{p(\cdot)}} \\
& = C^{\frac{1}{p_-}} [w]_{\mathcal{A}_{p(\cdot),\infty}}^{\frac{A}{p_-}}
\frac{1}{\|\mathbf{1}_{Q}\|_{L^{p(\cdot)}}} \left\|w\mathbf{1}_Q\right\|_{L^{p(\cdot)}}.
\end{align*}
This proves the above claim.

Now, let $p(\cdot) \in \mathcal{P}$ with $p(\cdot) \in LH$.
From Lemma \ref{reverse Holder w},
it follows that, if $r := r_w$, where $r_w$ is the same as in \eqref{def rw},
then, for any cube $Q$ in $\mathbb{R}^n$,
\begin{align}\label{reverse eq 1}
 \fint_Q \mathbb{W}(x)^r\,dx  \leq 2 \left[\fint_Q \mathbb{W}(x)\,dx\right]^r,
\end{align}
where $\mathbb{W} := [w(\cdot)]^{p(\cdot)}$.

We begin with the assumption $\|w\mathbf{1}_{Q_0}\|_{L^{p(\cdot)}} = 1$.
By this, we find that $L_w = 1$, where $L_w$ is the same as in \eqref{def Lw}.
Let $q(\cdot) := rp(\cdot)$.
Then, for any $x\in\mathbb{R}^n$,
\begin{align}\label{eq rq}
\frac{q(x)}{q_Q} = \frac{p(x)}{p_Q}
\ \  \text{and} \ \  \frac{r}{q_-} = \frac{1}{p_-}.
\end{align}
Next, fix a cube $Q$ in $\mathbb{R}^n$.
Applying Lemma \ref{est Q} and Remark \ref{rem est Q},
we find that $\| \mathbf{1}_Q \|_{L^{p(\cdot)}}^{-1} \sim |Q|^{-\frac{1}{p_Q}}$
and $\| \mathbf{1}_Q \|_{L^{q(\cdot)}}^{-1} \leq 6|Q|^{-\frac{1}{q_Q}}$.
Hence, to prove \eqref{reverse Holder Apinfty 1},
we only need to show that
\begin{align*}
\left\| |Q|^{-\frac{1}{q_Q}} w\mathbf{1}_Q \right\|_{L^{q(\cdot)}} \lesssim \left\| |Q|^{-\frac{1}{p_Q}} w\mathbf{1}_Q \right\|_{L^{p(\cdot)}}.
\end{align*}
By the homogeneity of $\|\cdot\|_{L^{q(\cdot)}}$,
we assume $\| |Q|^{-\frac{1}{p_Q}} w\mathbf{1}_Q \|_{L^{p(\cdot)}} = 1 $.
We split the proof into the following two cases: $|Q| \leq 1$ and $|Q| > 1$.

First, we consider the case $|Q| \leq 1$.
By Lemma \ref{rhof 1} with the assumption $\| |Q|^{-\frac{1}{p_Q}} w\mathbf{1}_Q \|_{L^{p(\cdot)}} = 1  $,
we obtain
$ \rho_{L^{p(\cdot)}}( |Q|^{-\frac{1}{p_Q}} w\mathbf{1}_Q ) = 1 $.
Hence, using this, \eqref{eq rq}, Lemma \ref{weight bound}{\rm (i)}, and \eqref{reverse eq 1},
we find that
\begin{align*}
\rho_{L^{q(\cdot)}} \left( |Q|^{-\frac{1}{q_Q}} w\mathbf{1}_Q \right)
& = \int_{Q} |Q|^{-\frac{q(x)}{q_Q}} \mathbb{W}(x)^r\,dx
 = \int_{Q} |Q|^{-\frac{p(x)}{p_Q}} \mathbb{W}(x)^r\,dx
 \leq C_D^{\frac{1}{p_-}} \fint_Q \mathbb{W}(x)^r\,dx\\
& \leq 2C_D^{\frac{1}{p_-}} \left[ \fint_Q \mathbb{W}(x)\,dx\right]^r
\leq 2C_D^{\frac{1+r}{p_-}} \left[ \int_Q |Q|^{-\frac{p(x)}{p_Q}}  \mathbb{W}(x)\,dx\right]^r
 = 2 C_D^{\frac{1+r}{p_-}},
\end{align*}
where $C_D$ is the same as in Lemma \ref{weight bound}{\rm (i)}.
Thus, from this and Lemma \ref{rhof 1},
we deduce that
\begin{align*}
\left\| |Q|^{-\frac{1}{q_Q}} w\mathbf{1}_Q \right\|_{L^{q(\cdot)}}
&\leq \max\left\{\left[\rho_{L^{q(\cdot)}} \left( |Q|^{-\frac{1}{q_Q}} w\mathbf{1}_Q \right)\right]^{\frac{1}{q_-}},
\left[\rho_{L^{q(\cdot)}} \left( |Q|^{-\frac{1}{q_Q}} w\mathbf{1}_Q \right)\right]^{\frac{1}{q_+}}\right\} \\
&\leq \left(2 C_D^{\frac{1+r}{p_-}}\right)^{\frac{1}{q_-}}
\leq  2^{\frac{1}{q_-}} C_D^{\frac{2r}{p_-q_-}} \leq 2 C_D^2,
\end{align*}
where the last two inequalities come from the facts
$C_D \geq 1$ and $\frac{r}{q_-} = \frac{1}{p_-}$.
This finishes the estimation of the case $|Q|\leq 1$.

Now, we consider the case $|Q| > 1$.
By \eqref{eq rq} and Lemmas \ref{est Q} and \ref{weight bound}{\rm (ii)},
we find that
\begin{align}\label{reverse eq 2}
\rho_{L^{q(\cdot)}} \left( |Q|^{-\frac{1}{q_Q}} w\mathbf{1}_Q \right)
& = \int_{Q} |Q|^{-\frac{q(x)}{q_Q}} \mathbb{W}(x)^r\,dx
 \leq \int_{Q} |Q|^{-\frac{p(x)}{p_Q}} \mathbb{W}(x)^r\,dx\nonumber\\
& \lesssim \int_{Q} \left\| \mathbf{1}_Q \right\|_{L^{p(\cdot)}}^{p(x)} \mathbb{W}(x)^r\,dx
 \lesssim \int_{Q} |Q|^{-\frac{p(x)}{p_\infty}} \mathbb{W}(x)^r\,dx.
\end{align}
Since $|Q|\geq 1$, we infer $|Q|^{-\frac{1}{p_\infty}} \leq 1$.
Hence, using this, Lemma \ref{rs f 1}, and \eqref{reverse eq 1}
with $f := |Q|^{-\frac{1}{p_\infty}}$ and $d\mu(x) := \mathbb{W}(x)^r\,dx$ yields
\begin{align}\label{reverse eq 3}
\int_{Q} |Q|^{-\frac{p(x)}{p_\infty}} \mathbb{W}(x)^r\,dx
& \leq e^{ntC_\infty} \fint_Q \mathbb{W}(x)^r\,dx
+ \int_Q \frac{\mathbb{W}(x)^r}{(e + |x|)^{ntp_-}}\,dx\nonumber\\
& \leq 2e^{ntC_\infty} \left[\fint_Q  \mathbb{W}(x)\,dx\right]^r
+ \int_Q \frac{\mathbb{W}(x)^r}{(e + |x|)^{ntp_-}}\,dx
=: 2e^{ntC_\infty}{\rm I}^r + {\rm II},
\end{align}
where $t \in (0,\infty)$ is a constant to be determined later.
We first give the estimation of ${\rm II}$.
Similar to \eqref{WWinfty eq 2} with
$Q_k := Q(\mathbf{0},2e^{k+1})$ for any $k\in\mathbb{Z}_+$,
by \eqref{WWinfty eq 11},
we conclude that
\begin{align}\label{reverse eq 4}
{\rm II} & \leq \int_{\mathbb{R}^n} \frac{\mathbb{W}(x)^r}{(e + |x|)^{ntp_-}}\,dx
\leq e^{-ntp_-} \int_{Q_0} \mathbb{W}(x)^r\,dx + \sum_{k = 1}^\infty e^{-kntp_-} \int_{Q_k} \mathbb{W}(x)^r\,dx.
\end{align}
From \eqref{reverse eq 1} and the facts $|Q_k| \geq 1$ and $1-r<0$,
it follows that, for any $k\in\mathbb{Z}_+$,
\begin{align*}
\int_{Q_k} \mathbb{W}(x)^r\,dx = |Q_k| \fint_{Q_k} \mathbb{W}(x)^r\,dx
\leq 2|Q_k| \left[ \fint_{Q_k} \mathbb{W}(x)\,dx\right]^r
= 2|Q_k|^{1-r} \mathbb{W}(Q_k)^r \leq 2 \mathbb{W}(Q_k)^r.
\end{align*}
Using this, \eqref{reverse eq 4}, and Lemma \ref{rhof 1}
with the fact that $\mathbb{W}(Q_k) \geq 1$ for any $k\in\mathbb{Z}_+$,
we obtain
\begin{align}\label{reverse eq 5}
{\rm II} \leq e^{-ntp_-} \mathbb{W}\left(Q_0\right)^r + \sum_{k = 1}^\infty e^{-kntp_-} \mathbb{W}\left(Q_k\right)^r
\lesssim  e^{-ntp_-} \left\|w\mathbf{1}_{Q_0}\right\|_{L^{p(\cdot)}}^{rp_+} + \sum_{k = 1}^\infty e^{-kntp_-} \left\|w\mathbf{1}_{Q_k}\right\|_{L^{p(\cdot)}}^{rp_+}.
\end{align}
From Lemma \ref{EB 1} and the assumption $\|w\mathbf{1}_{Q_0}\|_{L^{p(\cdot)}} = 1$,
we deduce that,
for any $k\in\mathbb{Z}_+$,
\begin{align*}
\left\|w\mathbf{1}_{Q_k}\right\|_{L^{p(\cdot)}} \leq C_w e^{\delta_w} \left\|w\mathbf{1}_{Q_0}\right\|_{L^{p(\cdot)}}
= C_w e^{\delta_w},
\end{align*}
where $C_w := 2^{2^n}C_{p(\cdot),n}^{2^n + 1} [w]_{\mathcal{A}_{p(\cdot),\infty}}^{2^n + 1}$
and $\delta_w := \frac{2^n}{n}[1 + \log_2 (C_{p(\cdot),n}[w]_{\mathcal{A}_{p(\cdot),\infty}})]$.
By this and \eqref{reverse eq 5},
we obtain
\begin{align}\label{reverse eq 6}
{\rm II} \leq e^{-ntp_-}
+ C_w^{rp_+} \sum_{k = 1}^\infty e^{-kntp_- + knrp_+\delta_w},
\end{align}
which further implies that, if $t > \frac{rp_+\delta_w}{p_-}$,
then the summation in \eqref{reverse eq 6} converges.

Next, let
$$t_1 := \frac{rp_+\delta_w}{p_-} + \frac{1}{np_-} \log\left( 2C_w^{rp_+} \right). $$
Using the fact $C_w > 1$,
we easily find that $t_1 > \frac{rp_+\delta_w}{p_-}$.
Hence, by this and the summation formula for geometric series
with taking $t := t_1$,
we obtain
\begin{align*}
C_w^{rp_+} \sum_{k = 1}^\infty e^{-kntp_- + knrp_+\delta_w}
\lesssim C_w^{rp_+} \sum_{k = 1}^\infty e^{-k\log( 2C_w^{rp_+} )}
= \frac{C_w^{rp_+} e^{-\log( 2C_w^{rp_+} )}}{1- e^{-\log( 2C_w^{rp_+} )}}
< 1.
\end{align*}
Consequently, from this, \eqref{reverse eq 6},
and the fact $C_w > 1$ with taking $t:= t_1$,
we infer that
\begin{align}\label{reverse eq 7}
{\rm II} \leq e^{-nrp_+\delta_w} \left(2C_w^{rp_+}\right)^{-1} + 1 <2.
\end{align}

Now, we give the estimation of {\rm I}.
Using the fact $|Q|^{-\frac{1}{p_\infty}} \leq 1$,
Lemma \ref{rs f 1} with $f := |Q|^{-\frac{1}{p_\infty}}$,
and $d\mu(x) := \mathbb{W}(x)\,dx$ yields
\begin{align}\label{reverse eq 8}
{\rm I} \leq e^{nsC_\infty} \int_{Q} |Q|^{-\frac{p(x)}{p_\infty}} \mathbb{W}(x)\,dx
+ \int_Q \frac{\mathbb{W}(x)}{(e + |x|)^{nsp_-}}\,dx,
\end{align}
where $s$ is any positive constant waiting to be fixed.
Similar to \eqref{WWinfty eq 3},
we find that, if choosing
$$ s_1 := \frac{p_+\delta_w}{p_-} + \frac{1}{np_-} \log\left( 2C_w^{p_+} \right), $$
then
\begin{align}\label{reverse eq 9}
 \int_Q \frac{\mathbb{W}(x)}{(e + |x|)^{ns_1p_-}}\,dx < 2.
\end{align}
By Lemma \ref{rhof 1} with the assumption $\||Q|^{-\frac{1}{p_Q}} w\mathbf{1}_Q\|_{L^{p(\cdot)}} = 1$,
we conclude that
$$ \int_{Q} |Q|^{-\frac{p(x)}{p_Q}} \mathbb{W}(x)\,dx =  \rho_{L^{p(\cdot)}}(|Q|^{-\frac{1}{p_Q}} w\mathbf{1}_Q) = 1.  $$
Hence, from this, \eqref{reverse eq 9}, \eqref{reverse eq 8},
and Lemmas \ref{weight bound}{\rm (ii)} and \ref{est Q},
we deduce that
\begin{align*}
{\rm I} &\leq e^{\frac{nC_\infty p_+\delta_w}{p_-}} \left( 2C_w^{p_+} \right)^{\frac{C_\infty}{p_-}}
\int_{Q} |Q|^{-\frac{p(x)}{p_\infty}} \mathbb{W}(x)\,dx + 2
 \lesssim e^{\frac{nC_\infty p_+\delta_w}{p_-}} \left( 2C_w^{p_+} \right)^{\frac{C_\infty}{p_-}}
\int_{Q} \|\mathbf{1}_Q\|_{L^{p(\cdot)}}^{-p(x)} \mathbb{W}(x)\,dx + 2\\
& \lesssim e^{\frac{nC_\infty p_+\delta_w}{p_-}} \left( 2C_w^{p_+} \right)^{\frac{C_\infty}{p_-}}
\int_{Q} |Q|^{-\frac{p(x)}{p_Q}} \mathbb{W}(x)\,dx
 = e^{\frac{nC_\infty p_+\delta_w}{p_-}} \left( 2C_w^{p_+} \right)^{\frac{C_\infty}{p_-}} + 2
 \lesssim  e^{\frac{nC_\infty p_+\delta_w}{p_-}} \left( 2C_w^{p_+} \right)^{\frac{C_\infty}{p_-}},
\end{align*}
where the final inequality comes from the fact that
$1 < e^{\frac{nC_\infty p_+\delta_w}{p_-}} \left( 2C_w^{p_+} \right)^{\frac{C_\infty}{p_-}}$.
Thus, by combining this with \eqref{reverse eq 2}, \eqref{reverse eq 3}, and \eqref{reverse eq 7}
with $t = t_1$,
we conclude that
\begin{align*}
\rho_{L^{q(\cdot)}} \left( |Q|^{-\frac{1}{q_Q}} w\mathbf{1}_Q \right)
&\lesssim e^{\frac{n r p_+ \delta_w C_\infty}{p_-}} \left( 2C_w^{rp_+} \right)^{\frac{C_\infty}{p_-}}\left[ e^{\frac{nC_\infty p_+\delta_w}{p_-}} \left( 2C_w^{p_+} \right)^{\frac{C_\infty}{p_-}}\right]^r
+ 2
\lesssim C_w^{\frac{rp_+C_\infty}{p_-}} e^{\frac{n r p_+ \delta_w C_\infty}{p_-}}.
\end{align*}
Hence, using this and Lemma \ref{rhof 1} yields
\begin{align*}
\left\| |Q|^{-\frac{1}{q_Q}} w\mathbf{1}_Q  \right\|_{L^{q(\cdot)}}
&\leq \max\left\{\left[\rho_{L^{q(\cdot)}} \left( |Q|^{-\frac{1}{q_Q}} w\mathbf{1}_Q \right)\right]^{\frac{1}{q_+}}, \left[\rho_{L^{q(\cdot)}} \left( |Q|^{-\frac{1}{q_Q}} w\mathbf{1}_Q \right)\right]^{\frac{1}{q_-}}  \right\}\\
& \lesssim \left[ C_w^{\frac{rp_+C_\infty}{p_-}} e^{\frac{n r p_+ \delta_w C_\infty}{p_-}}\right]^\frac{1}{q_-}
= C_w^{\frac{p_+C_\infty}{p_-p_-}} e^{\frac{n p_+ \delta_w C_\infty}{p_-p_-}}.
\end{align*}
Since $C_w = 2^{2^n}C_{p(\cdot),n}^{2^n + 1} [w]_{\mathcal{A}_{p(\cdot),\infty}}^{2^n + 1}$
and $\delta_w := \frac{2^n}{n}[1 + \log_2(C_{p(\cdot),n}[w]_{\mathcal{A}_{p(\cdot),\infty}})]$,
we infer that there exists a positive constant $A$,
depending only on $p(\cdot)$ and $n$, such that
$\| |Q|^{-\frac{1}{q_Q}} w\mathbf{1}_Q  \|_{L^{q(\cdot)}}  \lesssim [w]_{\mathcal{A}_{p(\cdot),\infty}}^A.$
This finishes the proof of \eqref{reverse Holder Apinfty 1}
in the case $\|w\mathbf{1}_{Q_0}\|_{L^{p(\cdot)}} = 1$.

Finally, we consider the case $\|w\mathbf{1}_{Q_0}\|_{L^{p(\cdot)}} \neq 1$.
Similar to the proof of Lemma \ref{Bp bound},
since $w \in \mathcal{A}_{p(\cdot),\infty}$,
it follows that $ \| w\mathbf{1}_{Q_0} \|_{L^{p(\cdot)}} \neq 0$.
Letting $v := \frac{w}{\|w\mathbf{1}_{Q_0}\|_{L^{p(\cdot)}}}$,
by the definition of $\mathcal{A}_{p(\cdot),\infty}$,
we have $\|v\mathbf{1}_{Q_0}\|_{L^p(\cdot)} = 1$
and $[w]_{\mathcal{A}_{p(\cdot),\infty}} = [v]_{\mathcal{A}_{p(\cdot),\infty}}$.
From this and the discussion on the case $\|w\mathbf{1}_{Q_0}\|_{L^{p(\cdot)}} \neq 1$,
we deduce that
\begin{align*}
\frac{1}{\|\mathbf{1}_Q\|_{L^{q(\cdot)}}}\left\| w\mathbf{1}_Q  \right\|_{L^{q(\cdot)}}
 &= \left\| w\mathbf{1}_{Q_0}  \right\|_{L^{p(\cdot)}} \frac{1}{\|\mathbf{1}_Q\|_{L^{q(\cdot)}}}\left\| v\mathbf{1}_Q  \right\|_{L^{q(\cdot)}}
\lesssim [v]_{\mathcal{A}_{p(\cdot),\infty}}^{A} \left\| w\mathbf{1}_{Q_0}  \right\|_{L^{p(\cdot)}}
 \frac{1}{\|\mathbf{1}_Q\|_{L^{p(\cdot)}}} \left\| v\mathbf{1}_Q  \right\|_{L^{p(\cdot)}}  \\
& = [w]_{\mathcal{A}_{p(\cdot),\infty}}^{A}
 \frac{1}{\|\mathbf{1}_Q\|_{L^{p(\cdot)}}} \left\| w\mathbf{1}_Q  \right\|_{L^{p(\cdot)}}.
\end{align*}
This finishes the proof of Theorem \ref{reverse Holder Apinfty}.
\end{proof}

The following result shows that the $\mathcal{A}_{p(\cdot),\infty}$ condition
is equivalent to the reverse H\"older's inequality in variable Lebesgue spaces,
where the necessity is guaranteed by Theorem \ref{reverse Holder Apinfty}.
\begin{theorem}\label{Apinfty Holder}
Let $p(\cdot)\in\mathcal{P}_0$ with $p(\cdot) \in LH$.
Then, for any scalar weight $w$, the following two statements are equivalent:
\begin{itemize}
\item[{\rm (i)}] $w \in \mathcal{A}_{p(\cdot),\infty}$;
\item[{\rm (ii)}] there exist positive constants $r\in (1,\infty)$ and $C$
such that, for any cube $Q$ in $\mathbb{R}^n$,
\begin{align}\label{Ap Holder eq 1}
\frac{1}{\|\mathbf{1}_Q\|_{L^{rp(\cdot)}}} \left\|w\mathbf{1}_Q\right\|_{L^{rp(\cdot)}}
\leq C\frac{1}{\|\mathbf{1}_Q\|_{L^{p(\cdot)}}} \left\|w\mathbf{1}_Q\right\|_{L^{p(\cdot)}}.
\end{align}
\end{itemize}
\end{theorem}
Before giving the proof of Theorem \ref{Apinfty Holder},
we first recall some useful tools.
\begin{lemma}\label{EQ 1}
Let $p(\cdot) \in \mathcal{P}$ with $p(\cdot) \in LH$ and let $w\in \mathscr{M}$.
If there exist $\alpha,\beta \in (0,1)$ such that, for any cube $Q$ in $\mathbb{R}^n$
and any measurable set $E\subset Q$ with $|E| \geq \alpha|Q|$,
$\|w\mathbf{1}_E\|_{L^{p(\cdot)}}\geq \beta \|w\mathbf{1}_Q\|_{L^{p(\cdot)}}$,
then there exist positive constants $C$ and $\delta$
such that, for any $\lambda \in (1,\infty)$ and any cube $Q$ in $\mathbb{R}^n$,
$\|w\mathbf{1}_{\lambda Q}\|_{L^{p(\cdot)}} \leq C \lambda^\delta
\|w\mathbf{1}_Q\|_{L^{p(\cdot)}}.$
\end{lemma}
\begin{proof}
Since $\alpha \in (0,1)$,
it follows that there exists $\mu \in (0,1)$ such that $\mu^n > \alpha$.
Thus, for any cube $Q$ in $\mathbb{R}^n$, $\mu Q \subset Q$ and $|\mu Q| = \mu^n |Q| > \alpha |Q|$.
By these, we obtain
$\|w\mathbf{1}_{\mu Q}\|_{L^{p(\cdot)}}\geq \beta \|w\mathbf{1}_Q\|_{L^{p(\cdot)}}.$
Using this with $Q$ replaced by $\frac{1}{\mu}Q$,
we find that
\begin{align}\label{EQ 2}
\left\|w\mathbf{1}_{\frac{1}{\mu}Q}\right\|_{L^{p(\cdot)}} \leq \frac{1}{\beta} \left\|w\mathbf{1}_Q\right\|_{L^{p(\cdot)}}.
\end{align}
Next, let $q$ be the least integer such that $\frac{1}{\mu^q} > 2 $.
From this and \eqref{EQ 2},
we infer that
$\|w\mathbf{1}_{2Q}\|_{L^{p(\cdot)}} \leq \frac{1}{\beta^q} \|w\mathbf{1}_Q\|_{L^{p(\cdot)}},$
which further implies that there exists $\delta \in (0,\infty)$ such that,
for any $\lambda \in (1,\infty)$,
$\|w\mathbf{1}_{\lambda Q}\|_{L^{p(\cdot)}} \lesssim \lambda^\delta \|w\mathbf{1}_Q\|_{L^{p(\cdot)}}.$
This finishes the proof of Lemma \ref{EQ 1}.
\end{proof}
The following lemma is an application of Lemma \ref{Holder},
which is exactly \cite[Corollary 2.28]{cf13}.
\begin{lemma}\label{Holder qr}
Let $p(\cdot),q(\cdot),r(\cdot) \in \mathcal{P}$ satisfy
$\frac{1}{p(\cdot)} = \frac{1}{q(\cdot)} + \frac{1}{r(\cdot)}$.
Then there exists a positive constant $C$ such that,
for any $f \in L^{q(\cdot)}$ and $g \in L^{r(\cdot)}$,
$fg \in L^{p(\cdot)}$ and
$\|fg\|_{L^{p(\cdot)}} \leq C \|f\|_{L^{q(\cdot)}} \|g\|_{L^{r(\cdot)}}.$
\end{lemma}
Now, we give the proof of Theorem \ref{Apinfty Holder}.
\begin{proof}[Proof of Theorem \ref{Apinfty Holder}]
By Theorem \ref{reverse Holder Apinfty},
we find that \eqref{Ap Holder eq 1} holds for any $w \in \mathcal{A}_{p(\cdot),\infty}$.
This finishes the proof that ${\rm (i)} \Rightarrow {\rm (ii)}$.
Hence, it only remains to prove ${\rm (ii)} \Rightarrow {\rm (i)}$.

First, we begin with the proof under the assumption $p(\cdot) \in \mathcal{P}$.
We claim that, for any cube $Q$ and any measurable set $E\subset Q$,
we have
\begin{align}\label{ww eq 7}
\left[\frac{\mathbb{W}(E)}{\mathbb{W}(Q)}\right]^{\frac{1}{p_-}} \lesssim \left( \frac{|E|}{|Q|} \right)^{\frac{1}{p_+}(1 - \frac{1}{r})},
\end{align}
where $\mathbb{W} := [w(\cdot)]^{p(\cdot)}$.
Indeed, if \eqref{ww eq 7} holds,
then, using this with \cite[Theorem 7.3.3]{g14},
we find $\mathbb{W}\in A_\infty$,
which, together with Theorem \ref{Apinfty A},
further implies $w\in\mathcal{A}_{p(\cdot),\infty}$.
Thus, we only need to prove the above claim.

We begin with the assumption $\|w\mathbf{1}_{Q_0}\|_{L^{p(\cdot)}} = 1$.
Indeed, by Lemma \ref{Holder qr} with taking $f := w\mathbf{1}_E$, $g := \mathbf{1}_E$,
$p(\cdot) := p(\cdot)$, $q(\cdot) := rp(\cdot)$,
and $r(\cdot) := \frac{r}{r-1}p(\cdot)$
and by Lemma \ref{con f} and \eqref{Ap Holder eq 1},
we find that, for any cube $Q$ in $\mathbb{R}^n$ and any measurable set $E\subset Q$,
\begin{align*}
\left\|w\mathbf{1}_E\right\|_{L^{p(\cdot)}}
&= \left\|\mathbf{1}_E\right\|_{L^{p(\cdot)}} \frac{\|w\mathbf{1}_E\|_{L^{p(\cdot)}}}{\|\mathbf{1}_E\|_{L^{p(\cdot)}}}
\lesssim \left\|\mathbf{1}_E\right\|_{L^{p(\cdot)}} \frac{\|w\mathbf{1}_E\|_{L^{rp(\cdot)}} \|\mathbf{1}_E\|_{L^{\frac{r}{r-1}p(\cdot)}}}{\|\mathbf{1}_E\|_{L^{p(\cdot)}}}\\
&= \left\|\mathbf{1}_E\right\|_{L^{p(\cdot)}} \frac{\|w\mathbf{1}_E\|_{L^{rp(\cdot)}}}{\|\mathbf{1}_E\|_{L^{rp(\cdot)}}}
 \leq \left\|\mathbf{1}_E\right\|_{L^{p(\cdot)}}^{1 - \frac1r} \left\|\mathbf{1}_Q\right\|_{L^{rp(\cdot)}} \frac{\|w\mathbf{1}_Q\|_{L^{rp(\cdot)}}}{\|\mathbf{1}_Q\|_{L^{rp(\cdot)}}}\\
&\lesssim \left\|\mathbf{1}_E\right\|_{L^{p(\cdot)}}^{1 - \frac1r} \left\|\mathbf{1}_Q\right\|_{L^{rp(\cdot)}} \frac{\|w\mathbf{1}_Q\|_{L^{p(\cdot)}}}{\|\mathbf{1}_Q\|_{L^{p(\cdot)}}}
 = \left[\frac{\|\mathbf{1}_E\|_{L^{p(\cdot)}}}{\|\mathbf{1}_Q\|_{L^{p(\cdot)}}}\right]^{1 - \frac1r} \left\|w\mathbf{1}_Q\right\|_{L^{p(\cdot)}},
\end{align*}
which further implies that
\begin{align}\label{ww eq 3}
\frac{\|w\mathbf{1}_E\|_{L^{p(\cdot)}}}{\|w\mathbf{1}_Q\|_{L^{p(\cdot)}}}
\lesssim \left[\frac{\|\mathbf{1}_E\|_{L^{p(\cdot)}}}{\|\mathbf{1}_Q\|_{L^{p(\cdot)}}}\right]^{1 - \frac1r}.
\end{align}
Using Lemma \ref{EB 3} and the fact $1\in \mathcal{A}_{p(\cdot),\infty}$,
we obtain $ \frac{\|\mathbf{1}_E\|_{L^{p(\cdot)}}}{\|\mathbf{1}_Q\|_{L^{p(\cdot)}}} \lesssim (\frac{|E|}{|Q|})^{\frac{1}{p_+}} $,
which, combined with \eqref{ww eq 3},
further yields
\begin{align}\label{ww eq 4}
\frac{\|w\mathbf{1}_E\|_{L^{p(\cdot)}}}{\|w\mathbf{1}_Q\|_{L^{p(\cdot)}}}
\lesssim \left(\frac{|E|}{|Q|}\right)^{\frac{1}{p_+}(1 - \frac1r)}.
\end{align}
Next, let $\alpha \in (0,1)$ be a fixed constant.
Then it follows from \eqref{ww eq 4} that there exists a positive constant $\beta \in (0,1)$
such that, for any cube $Q$ and any measurable set $E\subset Q$ with $|E| < \alpha |Q|$,
$\|w\mathbf{1}_E\|_{L^{p(\cdot)}} < \beta\|w\mathbf{1}_Q\|_{L^{p(\cdot)}}$.
Using this with $E$ replaced by $Q\setminus E$,
we conclude that, for any cube $Q$ and any measurable set $E\subset Q$ with $|E| \geq (1-\alpha)|Q|$,
$\|w\mathbf{1}_E\|_{L^{p(\cdot)}} \geq (1-\beta)\|w\mathbf{1}_Q\|_{L^{p(\cdot)}}$.
From this and Lemma \ref{EQ 1},
we deduce that there exists $\delta \in (0,\infty)$ such that,
for any cube $Q$ in $\mathbb{R}^n$ and any $\lambda \in [1,\infty)$,
\begin{align}\label{mu doubling}
\|w\mathbf{1}_{\lambda Q}\|_{L^{p(\cdot)}} \lesssim \lambda^\delta \|w\mathbf{1}_Q\|_{L^{p(\cdot)}}.
\end{align}
Recall that, in the proof of Lemma \ref{Bp bound},
we need to use Lemma \ref{EB 1}, which is the only place
using the properties of $\mathcal{A}_{p(\cdot),\infty}$ weights.
Thus, repeating the proof of Lemma \ref{Bp bound} with Lemma \ref{EB 1} replaced
by \eqref{mu doubling},
we find that, for any cube $Q$ in $\mathbb{R}^n$,
\begin{align}\label{ww eq 11}
\left\| w\mathbf{1}_Q \right\|_{L^{p(\cdot)}}^{p_-(Q) - p_+(Q)} \lesssim 1.
\end{align}
Now, let $Q$ be any cube in $\mathbb{R}^n$ and $E\subset Q$ be any measurable set.
If $\|w\mathbf{1}_Q\|_{L^{p(\cdot)}} \leq 1$,
then, from \eqref{ww eq 11} and Lemma \ref{rhof 1},
we infer that
\begin{align}\label{ww eq 5}
\frac{\|w\mathbf{1}_E\|_{L^{p(\cdot)}}}{\|w\mathbf{1}_Q\|_{L^{p(\cdot)}}}
= \|w\mathbf{1}_Q\|^{\frac{p_+(Q)}{p_-(Q)} - 1}_{L^{p(\cdot)}}
\frac{\|w\mathbf{1}_E\|_{L^{p(\cdot)}}}{\|w\mathbf{1}_Q\|^{\frac{p_+(Q)}{p_-(Q)}}_{L^{p(\cdot)}}}
\geq \|w\mathbf{1}_Q\|^{\frac{p_+(Q) - p_-(Q)}{p_-(Q)}}_{L^{p(\cdot)}} \left[\frac{\mathbb{W}(E)}{\mathbb{W}(Q)}\right]^\frac{1}{p_-(Q)}
\gtrsim \left[\frac{\mathbb{W}(E)}{\mathbb{W}(Q)}\right]^{\frac{1}{p_-}}.
\end{align}
If $\|w\mathbf{1}_E\|_{L^{p(\cdot)}} \leq 1 \leq \|w\mathbf{1}_Q\|_{L^{p(\cdot)}}$,
then, by Lemma \ref{rhof 1}, we conclude that
\begin{align}\label{ww eq 6}
\frac{\|w\mathbf{1}_E\|_{L^{p(\cdot)}}}{\|w\mathbf{1}_Q\|_{L^{p(\cdot)}}}
\geq \left[\frac{\mathbb{W}(E)}{\mathbb{W}(Q)}\right]^{\frac{1}{p_-(Q)}}
\geq \left[\frac{\mathbb{W}(E)}{\mathbb{W}(Q)}\right]^{\frac{1}{p_-}}.
\end{align}
If $\|w\mathbf{1}_E\|_{L^{p(\cdot)}} \geq 1$, then,
using Remark \ref{rem EB 2} and \eqref{mu doubling},
we find that $ \|w\mathbf{1}_E\|_{L^{p(\cdot)}} \sim \mathbb{W}(E)^{\frac{1}{p_\infty}} $
and $ \|w\mathbf{1}_Q\|_{L^{p(\cdot)}} \sim \mathbb{W}(Q)^{\frac{1}{p_\infty}} $,
which, together with the fact $p_- \leq p_\infty$, further implies that
\begin{align*}
\frac{\|w\mathbf{1}_E\|_{L^{p(\cdot)}}}{\|w\mathbf{1}_Q\|_{L^{p(\cdot)}}}
\sim \left[\frac{\mathbb{W}(E)}{\mathbb{W}(Q)}\right]^{\frac{1}{p_\infty}} \geq \left[\frac{\mathbb{W}(E)}{\mathbb{W}(Q)}\right]^{\frac{1}{p_-}}.
\end{align*}
Hence, from this, \eqref{ww eq 4}, \eqref{ww eq 5}, and \eqref{ww eq 6},
it follows that
\begin{align*}
\left[\frac{\mathbb{W}(E)}{\mathbb{W}(Q)}\right]^{\frac{1}{p_-}}
\lesssim \frac{\|w\mathbf{1}_E\|_{L^{p(\cdot)}}}{\|w\mathbf{1}_Q\|_{L^{p(\cdot)}}}
\lesssim \left( \frac{|E|}{|Q|} \right)^{\frac{1}{p_+}(1 - \frac{1}{r})}.
\end{align*}
This finishes ths proof of \eqref{ww eq 7} in the case $\|w\mathbf{1}_{Q_0}\|_{L^{p(\cdot)}} = 1$.

Next, we consider the case $\|w\mathbf{1}_{Q_0}\|_{L^{p(\cdot)}} \neq 1$.
Since $w$ is positive,
we deduce that $\|w\mathbf{1}_{Q_0}\|_{L^{p(\cdot)}} >0$.
Now, letting $v := \frac{w}{\|w\mathbf{1}_{Q_0}\|_{L^{p(\cdot)}}}$
and $\mathbb{W}_0 := [v(\cdot)]^{p(\cdot)}$, by the proof above,
we find that
\begin{align*}
\left[\frac{\mathbb{W}(E)}{\mathbb{W}(Q)}\right]^{\frac{1}{p_-}}
&\lesssim \max\left\{ \left\|w\mathbf{1}_{Q_0}\right\|^{\frac{p_+}{p_-} - 1}_{L^{p(\cdot)}}, \left\|w\mathbf{1}_{Q_0}\right\|^{1 - \frac{p_+}{p_-}}_{L^{p(\cdot)}} \right\}
\left[\frac{\mathbb{W}_0(E)}{\mathbb{W}_0(Q)}\right]^{\frac{1}{p_-}}
\lesssim \frac{\|v\mathbf{1}_E\|_{L^{p(\cdot)}}}{\|v\mathbf{1}_Q\|_{L^{p(\cdot)}}}
\lesssim \left( \frac{|E|}{|Q|} \right)^{\frac{1}{p_+}(1 - \frac1r)},
\end{align*}
which completes the proof of \eqref{ww eq 7} for the case $\|w\mathbf{1}_{Q_0}\|_{L^{p(\cdot)}} \neq 1$.

Finally, we consider the case $ p(\cdot)\in \mathcal{P}_0 $.
By Lemma \ref{con f},
we find that, for any cube $Q$,
\eqref{Ap Holder eq 1} is equivalent to
\begin{align*}
\left[\frac{1}{\|\mathbf{1}_Q\|_{L^{\frac{rp(\cdot)}{p_-}}}} \left\|w^{p_-}\mathbf{1}_Q\right\|_{L^{\frac{rp(\cdot)}{p_-}}}\right]^\frac{1}{p_-}
\lesssim \left[\frac{1}{\|\mathbf{1}_Q\|_{L^{\frac{p(\cdot)}{p_-}}}} \left\|w^{p_-}\mathbf{1}_Q\right\|_{L^{\frac{p(\cdot)}{p_-}}}\right]^\frac{1}{p_-},
\end{align*}
which further implies that
\begin{align*}
\frac{1}{\|\mathbf{1}_Q\|_{L^{\frac{rp(\cdot)}{p_-}}}} \left\|w^{p_-}\mathbf{1}_Q\right\|_{L^{\frac{rp(\cdot)}{p_-}}}
\lesssim \frac{1}{\|\mathbf{1}_Q\|_{L^{\frac{p(\cdot)}{p_-}}}} \left\|w^{p_-}\mathbf{1}_Q\right\|_{L^{\frac{p(\cdot)}{p_-}}}.
\end{align*}
From this, the facts that $\frac{p(\cdot)}{p_-} \in \mathcal{P}$ and $\frac{p(\cdot)}{p_-}\in LH$,
and the proved result for $p(\cdot) \in \mathcal{P}$ with $p(\cdot) \in LH$,
we infer that $ w^{p_-} \in \mathcal{A}_{\frac{p(\cdot)}{p_-},\infty}$.
Using this and Lemma \ref{wr},
we conclude that $w \in \mathcal{A}_{p(\cdot),\infty}$.
This finishes the proof of Theorem \ref{Apinfty Holder}.
\end{proof}

\section{Variable Matrix Weights}\label{sec matrix}
In this section, we first introduce matrix $\mathscr{A}_{p(\cdot),\infty}$ weights
and reducing operators related to matrix $\mathscr{A}_{p(\cdot),\infty}$ weights.
Then, in Subsection \ref{sec proApinfty},
we extend some results on scalar-valued weights to the matrix-valued case,
such as the proper inclusion (Proposition \ref{Ap Apinfty 2}),
the reverse H\"older's inequality (Theorem \ref{reverse Holder Apinfty}),
and the equivalent expressions (Theorem \ref{w reverse}).
Finally, in Subsection \ref{sec Apdimesnion},
for further applications of matrix $\mathscr{A}_{p(\cdot),\infty}$ weights,
we introduce and study the upper and the lower weight dimensions for matrix $\mathscr{A}_{p(\cdot),\infty}$ weights.

We first recall some basic concepts of matrices and matrix weights.
For any $m,n\in\mathbb{N}$, the set of all $m\times n$ complex-valued matrices
is denoted by the \emph{symbol $M_{m, n}$},
and $M_{m, m}$ is simply denoted by $M_m$.
For any $A\in M_m$, let
\begin{align}\label{def norm matrix}
\|A\| := \sup_{\vec{z}\in\mathbb{C}^m, |\vec{z}| = 1} \left| A\vec{z} \right|.
\end{align}
Then $(M_m, \|\cdot\|)$ is a Banach space.
Moreover, we have the following well-known result
(see, for instance, \cite[Lemma 2.3]{bhyy23}).
\begin{lemma}\label{norm matrix}
Let $A,B\in M_m$ be two nonnegative definite matrices.
Then $\|AB\| = \|BA\|$.
\end{lemma}
Next, we recall the concept of matrix weights
(see, for instance, \cite[Definition 2.7]{bhyy23}).
\begin{definition}\label{def matrix}
A matrix-valued function $W:  \mathbb{R}^n \to M_m$ is called a \emph{matrix weight}
if $W$ satisfies
\begin{itemize}
\item[{\rm (i)}] for almost every $x\in\mathbb{R}^n$, $W(x)$ is nonnegative definite,
\item[{\rm (ii)}] for almost every $x\in\mathbb{R}^n$, $W(x)$ is invertible,
\item[{\rm (iii)}] the entries of $W$ are all locally integrable.
\end{itemize}
\end{definition}
Now, we recall the definition of variable matrix $\mathscr{A}_{p(\cdot)}$ weights
(see \cite{cp23} for more details).
\begin{definition}\label{def Ap}
Let $p(\cdot)\in \mathcal{P}$.
A matrix weight $W$ on $\mathbb{R}^n$ is called a \emph{matrix $\mathscr{A}_{p(\cdot)}$ weight}
if
$$ \left[W\right]_{\mathscr{A}_{p(\cdot)}} := \sup_{Q} |Q|^{-1} \left\|\, \left\|\, \left\| W(x)W^{-1}(\cdot)  \right\| \mathbf{1}_Q(\cdot) \right\|_{L^{p'(\cdot)}} \mathbf{1}_{Q}(x) \right\|_{L^{p(\cdot)}_x} <\infty , $$
where the supremum is taken over all cubes $Q$ in $\mathbb{R}^n$
and $L^{p(\cdot)}_x$ indicates to take the norm  with respect to the variable $x$.
\end{definition}
\begin{remark}\label{rem cap}
\begin{itemize}
\item[{\rm (i)}] In the scalar-valued case (namely $m = 1$), $\mathscr{A}_{p(\cdot)}$
in Definition \ref{def Ap} reduces to $\mathcal{A}_{p(\cdot)}$.
\item[{\rm (ii)}] In Definition \ref{def Ap}, if $p(\cdot)=p$ is a constant exponent,
then, for any $W \in \mathscr{A}_{p}$,
the $p$-th power of $W$ is the classical matrix $A_p$ weight
(see, for example, \cite{v97} for more details about $A_p$ matrix weights
and the $p$-th power of matrices).
\item[{\rm (iii)}] By Definition \ref{def Ap},
it is obvious that, for any $W \in \mathscr{A}_{p(\cdot)}$,
$ W^{-1} \in \mathscr{A}_{p'(\cdot)}$.
\end{itemize}
\end{remark}

Then we recall the concept of  matrix $A_{p,\infty}$ weights
(see, for example, \cite{v97} and \cite{bhyy23} for more details).

\begin{definition}
Let $p\in (0,\infty)$. A matrix weight $W$ is called an \emph{$A_{p,\infty}$ matrix weight} if
\begin{align*}
[W]_{A_{p,\infty}} := \sup_{Q} \exp\left( \fint_Q \log \left( \fint_Q \left\| W^\frac1p (x) W^{-\frac1p}(y) \right\|^p\,dx \right)\,dy \right)< \infty,
\end{align*}
where the supremum is taken over all cubes $Q$ in $\mathbb{R}^n$.
\end{definition}

\begin{remark}\label{rem Apinfty}
\begin{itemize}
\item[{\rm (i)}] By the definitions of $A_{p,\infty}$ and $\mathscr{A}_{p(\cdot),\infty}$,
when $p(\cdot)=p$ is a constant exponent,
we easily find that, for any matrix weight $W$,
$W \in \mathscr{A}_{p,\infty}$ if and only if $W^p \in A_{p,\infty}$.
\item[{\rm (ii)}] When $p(\cdot) = 1$, it is obvious that $\mathscr{A}_{1,\infty} = A_{1,\infty}$.
\end{itemize}
\end{remark}

The following result gives an equivalent characterization of matrix $\mathscr{A}_{p(\cdot),\infty}$ weights.
\begin{proposition}\label{eq def infty lem}
Let $p(\cdot)\in \mathcal{P}_0$ with $p(\cdot) \in LH$
and let $W$ be a matrix weight.
If, for any cube $Q\subset \mathbb{R}^n$,
$$ \log_+\left( \left\|\,\left\| W(x)W^{-1}(\cdot) \right\|\mathbf{1}_Q \right\|_{L^{p(\cdot)}_x} \right) \in L^1(Q), $$
where $\log_+$ is the same as in \eqref{def log plus},
then
\begin{align}\label{eq def infty}
\left[W\right]_{\mathscr{A}_{p(\cdot),\infty}(\mathbb{R}^n)}
&\sim \sup_{Q}
\sup_{H\in \mathcal{F}_{Q,W}} \left[ \frac{1}{\|\mathbf{1}_Q\|_{L^{p(\cdot)}}} \left\|\,\left\| W(\cdot)H(\cdot) \right\|\mathbf{1}_Q \right\|_{L^{p(\cdot)}} \right]^{-1} \nonumber\\
&\quad \times \exp\left( \fint_{Q} \log\left( \frac{1}{\|\mathbf{1}_Q\|_{L^{p(\cdot)}}} \left\|\, \left\| W(\cdot)H(y)\right\| \mathbf{1}_Q \right\|_{L^{p(\cdot)}}
 \right)\,dy \right),
\end{align}
where the first supremum is taken over all cubes $Q$ in $\mathbb{R}^n$,
the positive equivalence constants  depend only on $p(\cdot)$ and $n$,
and
\begin{align*}
\mathcal{F}_{Q,W} := &\Bigg\{ H:  \mathbb{R}^n \to M_m(\mathbb{C})\ \text{measurable}:
\left\|\,\left\| W(\cdot)H(\cdot) \right\|\mathbf{1}_Q \right\|_{L^{p(\cdot)}} \neq 0,\\
&\quad \left.\log\left( \frac{1}{\|\mathbf{1}_Q\|_{L^{p(\cdot)}}} \left\|\, \left\| W(\cdot)H(y)\right\| \mathbf{1}_Q \right\|_{L^{p(\cdot)}}\right) \in L^1(Q) \right\}. \nonumber
\end{align*}
\end{proposition}
\begin{proof}
It is obvious that, if we take $H = W^{-1}$,
then $H\in \mathcal{F}_{Q,W}$ and the right-hand side of \eqref{eq def infty} is
at least as big as $[W]_{\mathscr{A}_{p(\cdot),\infty}}$.
Thus, the left-hand side of \eqref{eq def infty} is not more than the right-hand one.

Next, we show the converse inequality. Let $r := \min\{1,p_-\}$.
From the definition of $[W]_{\mathscr{A}_{p(\cdot),\infty}}$, Jensen's inequality,
and Lemmas \ref{est fQ} and \ref{con f} with $\frac{p(\cdot)}{r} \in \mathcal{P}$ and $\frac{p(\cdot)}{r} \in LH$,
it follows that, for any cube $Q$ in $\mathbb{R}^n$ and any $H\in \mathcal{F}_{Q,W}$,
\begin{align*}
&\exp\left( \fint_{Q} \log\left( \frac{1}{\|\mathbf{1}_Q\|_{L^{p(\cdot)}}} \left\|\, \left\| W(\cdot)H(y) \right\| \mathbf{1}_Q \right\|_{L^{p(\cdot)}} \right)\,dy \right)\\
&\quad \leq \exp\left( \fint_{Q} \log\left( \frac{1}{\|\mathbf{1}_Q\|_{L^{p(\cdot)}}} \left\|\, \left\| W(\cdot)W^{-1}(y) \right\| \mathbf{1}_Q \right\|_{L^{p(\cdot)}} \right)\,dy \right)\\
& \quad\quad \times \exp\left( \fint_{Q} \log\left( \frac{1}{\|\mathbf{1}_Q\|_{L^{p(\cdot)}}} \left\|\, \left\| W(y)H(y) \right\| \mathbf{1}_Q \right\|_{L^{p(\cdot)}} \right)\,dy \right)\\
&\quad \leq [W]_{\mathscr{A}_{p(\cdot),\infty}} \left[\exp\left( \fint_{Q} \log\left( \left\| W(y)H(y) \right\|^r \right)\,dy \right)\right]^\frac1r\\
&\quad \lesssim [W]_{\mathscr{A}_{p(\cdot),\infty}} \left[\frac{1}{\|\mathbf{1}_Q\|_{L^{\frac{p(\cdot)}{r}}}} \left\|\, \left\| W(y)H(y) \right\|^r \mathbf{1}_Q \right\|_{L^{\frac{p(\cdot)}{r}}}\right]^\frac1r\\
&\quad = [W]_{\mathscr{A}_{p(\cdot),\infty}} \frac{1}{\|\mathbf{1}_Q\|_{L^{p(\cdot)}}} \left\|\, \left\| W(y)H(y) \right\| \mathbf{1}_Q \right\|_{L^{p(\cdot)}},
\end{align*}
which
further implies that the left-hand side of \eqref{eq def infty} is not less than the right-hand one.
This finishes the proof of Proposition \ref{eq def infty lem}.
\end{proof}

Now, we recall the concept of the reducing operators for matrix weights
(see, for example, \cite{cp23} for reducing operators for $\mathscr{A}_{p(\cdot)}$
and \cite{bhyy23} for reducing operators for $A_{p,\infty}$).
\begin{definition}
Let $p(\cdot) \in \mathcal{P}_0$,
$W$ be a matrix weight,
and $Q$ any cube in $\mathbb{R}^n$.
The matrix $A_Q\in M_m$ is called a
\emph{reducing operator} of order $p(\cdot)$ for $W$
if $A_Q$ is positive definite and self-adjoint such that,
for any $\vec{z} \in \mathbb{C}^m$,
\begin{align}\label{eq redu}
\left| A_Q \vec{z} \right|
\sim \frac{1}{\|\mathbf{1}_Q\|_{L^{p(\cdot)}}} \left\|\, \left| W(\cdot) \vec{z} \right|  \mathbf{1}_{Q} \right\|_{L^{p(\cdot)}},
\end{align}
where the positive equivalence constants depend only on $m$ and $p(\cdot)$.
\end{definition}
The following proposition guarantees the existence of reducing operators of order $p(\cdot)$
for matrix weights.
\begin{proposition}\label{ext redu}
Let $p(\cdot) \in \mathcal{P}_0$.
Then, for any matrix weight $W$ and any cube $Q$ in $\mathbb{R}^n$,
the reducing operator $A_Q$ of order $p(\cdot)$ for $W$ exists.
\end{proposition}

\begin{proof}
The case $p_-\geq 1$ has been proved by Cruz--Uribe in \cite[p.\,1142]{cp23}
and hence we only need to show the case $p_- < 1$ here,
for which we borrow some ideas from \cite[p.\,1237]{fr04}.
To this end, for any cube $Q$ in $\mathbb{R}^n$ and any $\vec{z} \in \mathbb{C}^m$,
let
$$ \rho_{p(\cdot),Q}\left(\vec{z}\right) := \frac{1}{\|\mathbf{1}_Q\|_{L^{p(\cdot)}}} \left\|\,\left| W(\cdot)\vec{z} \right|\mathbf{1}_Q  \right\|_{L^{p(\cdot)}}. $$
Since $p_- \in (0,1)$, we deduce that $\rho_{p(\cdot),Q}$ is a quasi-norm.
Using Lemma \ref{con f} and the fact $\frac{p(\cdot)}{p_-} \in \mathcal{P}$,
we find that, for any $\vec{x},\vec{z} \in \mathbb{C}^m$,
\begin{align}\label{rho norm}
\left[\rho_{p(\cdot),Q}(\vec{x} + \vec{z})\right]^{p_-}
& = \frac{1}{\|\mathbf{1}_Q\|_{L^{\frac{p(\cdot)}{p_-}}}} \left\|\,\left| W(\cdot)\left(\vec{x} +\vec{z} \right) \right|^{p_-} \mathbf{1}_Q  \right\|_{L^{\frac{p(\cdot)}{p_-}}}\nonumber\\
&\leq \frac{1}{\|\mathbf{1}_Q\|_{L^{\frac{p(\cdot)}{p_-}}}} \left\|\,\left| W(\cdot)\vec{x}\right|^{p_-} \mathbf{1}_Q  \right\|_{L^{\frac{p(\cdot)}{p_-}}}
+ \frac{1}{\|\mathbf{1}_Q\|_{L^{\frac{p(\cdot)}{p_-}}}} \left\|\,\left| W(\cdot)\vec{z} \right|^{p_-} \mathbf{1}_Q  \right\|_{L^{\frac{p(\cdot)}{p_-}}}\nonumber\\
&= \left[\rho_{p(\cdot),Q}(\vec{x})\right]^{p_-} + \left[\rho_{p(\cdot),Q}(\vec{z})\right]^{p_-}.
\end{align}
Let $C := \{\vec{z} \in \mathbb{C}^m:  \rho_{p(\cdot),Q}(\vec{z}) < 1\}$
be the unit ball of $\rho_{p(\cdot),Q}$.
Since $\rho_{p(\cdot),Q}$ is a quasi-norm,
it follows that $C$ may not be convex.
Let $D$ be the convex hull of $C$
and $q(\vec{z}) := \inf\{t >0:  \frac{\vec{z}}{t} \in D\}$
be the Minkovski functional of $D$.
From the definition of $D$,
we infer that $q$ is a norm on $\mathbb{C}^m$.

To show the existence of reducing operators,
we claim that $D\subset (2m+1)^{\frac{1}{p_-}-1}C$.
If this claim holds,
then $C \subset D \subset (2m+1)^{\frac{1}{p_-}-1}C$
and hence
the quasi-norm $\rho_{p(\cdot),Q} \sim q$,
which, combined with \cite[Theorem 4.11]{bc22},
further implies that there exists a positive definite and invertible matrix $A_Q$
such that
$\rho_{p(\cdot),Q}(\vec{z}) \sim q(\vec{z}) \sim |A_Q \vec{z}|$ for any $\vec{z} \in \mathbb{C}^m$,
where the positive equivalence constants are independent of $Q$ and $\vec{z}$.

Finally, we prove the claim $D\subset (2m+1)^{\frac{1}{p_-}-1}C$.
For any $\vec{x} \in D$, by a classical theorem of Carath\'eodory (see \cite{c07} and also \cite[p.\,55]{w94}),
we conclude that there exist at most $2m+1$ points $\{\vec{x_i}\}_{i = 1}^{2m+1}$ of $C$
such that $\vec{x}$ is the convex combination of $\{\vec{x_i}\}_{i = 1}^{2m+1}$,
that is, $\vec{x} = \sum_{i = 1}^{2m+1} \alpha_i \vec{x_i}$,
where $\sum_{i = 1}^{2m+1}\alpha_i = 1$ and $\alpha_i \geq 0$ for any $i \in \{1,\dots,2m+1\}$.
Then, using \eqref{rho norm} and H\"older's inequality, we obtain
\begin{align*}
\left[\rho_{p(\cdot),Q}(\vec{x})\right]^{p_-}
&\leq \sum_{i = 1}^{2m+1} \left[\rho_{p(\cdot),Q}(\alpha_i\vec{x_i})\right]^{p_-}
= \sum_{i = 1}^{2m+1} \alpha^{p_-}_i\left[\rho_{p(\cdot),Q}(\vec{x_i})\right]^{p_-}
< \sum_{i = 1}^{2m+1} \alpha^{p_-}_i\\
&\leq (2m+1)^{\frac{1}{p_-}-1} \sum_{i = 1}^{2m+1} \alpha_i = (2m+1)^{\frac{1}{p_-}-1},
\end{align*}
which completes the proof of the claim above and hence Proposition \ref{ext redu}.
\end{proof}
Since the norm of matrices can be decomposed into the summation over the orthonormal basis
(see \cite[Lemma 3.2]{r03}),
it follows that we can extend \eqref{eq redu} from any vector
$\vec{z}$ \ to any matrix $M\in M_m$,
which is the following result.
\begin{lemma}\label{eq reduc M}
Let $p(\cdot) \in \mathcal{P}_0$, $W$ be a matrix weight,
and $Q$ any cube in $\mathbb{R}^n$.
If $A_Q$ is a reducing operator of order $p(\cdot)$ for $W$,
then, for any matrix $M\in M_m$,
$$ \left\| A_Q M \right\| \sim \frac{1}{\|\mathbf{1}_Q\|_{L^{p(\cdot)}}} \left\| \, \left\| W(\cdot) M  \right\|  \mathbf{1}_{Q} \right\|_{L^{p(\cdot)}}, $$
where the positive equivalence constants depend only on $m$ and $p(\cdot)$.
\end{lemma}
\begin{proof}
Let $\{\vec{e}_i\}_{i = 1}^m$ be any orthonormal basis of $\mathbb{C}^m$.
From \cite[Lemma 3.2]{r03},
we deduce that, for any matrix $M \in M_m$,
$\|M\|\sim \sum_{i = 1}^m  |M\vec{e}_i|, $
where the positive equivalence constants depend only on $m$.
By this and \eqref{eq redu},
we find that,
for any matrix $M\in M_m$,
\begin{align*}
\left\| A_Q M \right\| &\sim \sum_{i = 1}^m \left| A_Q M \vec{e}_i \right|
\sim \sum_{i = 1}^m \frac{1}{\|\mathbf{1}_Q\|_{L^{p(\cdot)}}}  \left\| \left| W(\cdot) M\vec{e}_i \right|  \mathbf{1}_{Q} \right\|_{L^{p(\cdot)}}\\
& \sim  \frac{1}{\|\mathbf{1}_Q\|_{L^{p(\cdot)}}}  \left\| \sum_{i = 1}^m \left| W(\cdot) M\vec{e}_i \right|  \mathbf{1}_{Q} \right\|_{L^{p(\cdot)}}
\sim \frac{1}{\|\mathbf{1}_Q\|_{L^{p(\cdot)}}}  \left\| \left\| W(\cdot) M \right\|  \mathbf{1}_{Q} \right\|_{L^{p(\cdot)}}.
\end{align*}
This finishes the proof of Lemma \ref{eq reduc M}.
\end{proof}

The following is the estimate on $A_Q^{-1}$
(see, for instance, \cite[Proposition 2.9]{bhyy24}
for a similar one about $A_{p,\infty}$ weights).
\begin{proposition}\label{AA-1}
Let $p(\cdot)\in \mathcal{P}_0$ with $p(\cdot)\in LH$, $W\in \mathscr{A}_{p(\cdot),\infty}$,
$Q$ be any cube in $\mathbb{R}^n$, and $A_Q$ the reducing operator of order $p(\cdot)$ for $W$.
Then, for any $\vec{z}\in\mathbb{C}^m$,
\begin{align*}
\left| A_Q^{-1} \vec{z} \right|
\sim \exp\left( \fint_{Q} \log \left| W^{-1}(y) \vec{z} \right| \,dy \right),
\end{align*}
where the positive equivalence constants depend only on $n$, $m$, $p(\cdot)$,
and $[W]_{\mathscr{A}_{p(\cdot)},\infty}$.
\end{proposition}
\begin{proof}
Letting $r := \min\{1,p_-\}$, by the assumption $p(\cdot) \in LH$,
we find that $\frac{p(\cdot)}{r} \in LH$ and $\frac{p(\cdot)}{r} \in \mathcal{P}$.
Then, from Jensen's inequality and Lemmas \ref{est fQ} and \ref{con f},
we infer that
\begin{align*}
\left| A_Q^{-1} \vec{z} \right|
& \leq \exp\left( \fint_{Q} \log \left(\left\| A_Q^{-1}W(y) \right\| \left| W^{-1}(y) \vec{z} \right|\right) \,dy \right)\\
& = \left[\exp\left( \fint_{Q} \log \left\| A_Q^{-1}W(y) \right\|^r  \,dy \right)\right]^{\frac{1}{r}}
\exp\left( \fint_{Q} \log  \left| W^{-1}(y) \vec{z} \right| \,dy \right)\\
& \leq \left[\fint_{Q}  \left\| A_Q^{-1}W(y) \right\|^r  \,dy \right]^{\frac{1}{r}}
\exp\left( \fint_{Q} \log  \left| W^{-1}(y) \vec{z} \right| \,dy \right)\\
&\lesssim \left[ \frac{1}{\|\mathbf{1}_{Q}\|_{L^{\frac{p(\cdot)}{r}}}} \left\|\, \left\| A_Q^{-1}W(y) \right\|^r\mathbf{1}_{Q} \right\|_{L^{\frac{p(\cdot)}{r}}} \right]^{\frac{1}{r}}
\exp\left( \fint_{Q} \log  \left| W^{-1}(y) \vec{z} \right| \,dy \right)\\
& = \frac{1}{\|\mathbf{1}_{Q}\|_{L^{p(\cdot)}}} \left\|\, \left\| A_Q^{-1}W(y) \right\| \mathbf{1}_{Q} \right\|_{L^{p(\cdot)}}
\exp\left( \fint_{Q} \log  \left| W^{-1}(y) \vec{z} \right| \,dy \right)\\
& \sim \exp\left( \fint_{Q} \log  \left| W^{-1}(y) \vec{z} \right| \,dy \right).
\end{align*}

Next, by the fact $\frac{p(\cdot)}{r} \in \mathcal{P}$, Jensen's inequality,
Lemmas \ref{est fQ} and \ref{con f}, and the definition of $\mathscr{A}_{p(\cdot),\infty}$,
we conclude that
\begin{align*}
&\exp\left( \fint_{Q} \log  \left| W^{-1}(y) \vec{z} \right| \,dy \right)\\
&\quad = \left[\exp\left( \fint_{Q} \log  \left| W^{-1}(y) \vec{z} \right|^r \,dy \right)\right]^\frac1r \\
&\quad \leq \left[\exp\left( \fint_{Q} \log  \left\| W^{-1}(y) A_Q \right\|^r \,dy\right) \right]^\frac1r
\exp\left( \fint_{Q} \log  \left| A_Q^{-1} \vec{z} \right|^r \,dy \right) \\
&\quad \lesssim \left| A_Q^{-1} \vec{z} \right|
\left[\exp\left( \fint_{Q} \log\left( \frac{1}{\|\mathbf{1}_Q\|_{L^{\frac{p(\cdot)}{r}}}} \left\|\, \left\| W(\cdot)W^{-1}(y) \right\|^r \mathbf{1}_Q\right\|_{L^{\frac{p(\cdot)}{r}}} \right) \,dy \right)\right]^\frac1r\\
&\quad = \left| A_Q^{-1} \vec{z} \right| \exp\left( \fint_{Q} \log\left( \frac{1}{\|\mathbf{1}_Q\|_{L^{p(\cdot)}}} \left\|\, \left\| W(\cdot)W^{-1}(y) \right\| \mathbf{1}_Q\right\|_{L^{p(\cdot)}} \right) \,dy \right)
\lesssim \left| A_Q^{-1} \vec{z} \right|,
\end{align*}
where the implicit constants depend only on $n$, $m$, $p(\cdot)$, and $[W]_{\mathscr{A}_{p(\cdot),\infty}}$.
This finishes the proof of Proposition \ref{AA-1}.
\end{proof}
Using Proposition \ref{AA-1},
we immediately obtain the following result.
Since its proof is similar to that of Lemma \ref{eq reduc M},
we omit the details here.
\begin{lemma}\label{AA-1 M}
Let $p(\cdot) \in \mathcal{P}_0$ with $p(\cdot) \in LH$, $W\in \mathscr{A}_{p(\cdot),\infty}$,
$Q$ be any cube in $\mathbb{R}^n$,
and $A_Q$ the reducing operator of order $p(\cdot)$ for $W$.
Then, for any $M\in M_m$,
$$ \left\| A_Q^{-1} M \right\|
\sim \exp\left( \fint_{Q} \log \left\| W^{-1}(y) M \right\| \,dy \right), $$
where the positive equivalence constants depend only on $n$, $m$, $p(\cdot)$,
and $[W]_{\mathscr{A}_{p(\cdot)},\infty}$.
\end{lemma}

\subsection{Properties of Matrix $\mathscr{A}_{p(\cdot),\infty}$ Weights}\label{sec proApinfty}

In this subsection, we prove that  some results for scalar-valued weights in
Section \ref{sec weight} can be established also in  the matrix-valued  case.
The following is the relationship among $\mathscr{A}_{p(\cdot)}$, $\mathscr{A}_{p(\cdot),\infty}$,
and $A_{1,\infty}$,
which is an  matrix-valued analogue of Proposition \ref{Ap Apinfty 2}.

\begin{proposition}\label{Ap Apinfty}
Let $p(\cdot) \in \mathcal{P}$ with $p(\cdot)\in LH$.
Then $\mathscr{A}_{p(\cdot)}\subset \mathscr{A}_{p(\cdot),\infty} \subset A_{1,\infty}$.
Moreover, if $p_- > 1$,
then $\mathscr{A}_{p(\cdot)} \subsetneqq \mathscr{A}_{p(\cdot),\infty} \subsetneqq A_{1,\infty}$.
\end{proposition}

To prove Proposition \ref{Ap Apinfty},
we need several elementary results.
Notice that, for any identity matrix $I_m$, we have $\|I_m\| = 1$.
Thus, by this and the definitions of $\mathscr{A}_{p(\cdot)}$ and $\mathscr{A}_{p(\cdot),\infty}$,
we obtain the following {\rm(i)} and {\rm (ii)} of Lemma \ref{W w 1},
while Lemma \ref{W w 1}{\rm(iii)} is precisely \cite[Lemma 3.2(ii)]{bhyy-a} with $p = 1$.
We omit the details here.

\begin{lemma}\label{W w 1}
Let $p(\cdot) \in \mathcal{P}$.
Then, for any matrix $W := wI_m$, where $w$ is a scalar-valued weight,
the following statements hold:
\begin{itemize}
\item[{\rm (i)}] $W \in \mathscr{A}_{p(\cdot)}$ if and only if $w \in \mathcal{A}_{p(\cdot)}$.
\item[{\rm (ii)}] $W \in \mathscr{A}_{p(\cdot),\infty}$ if and only if $w \in \mathcal{A}_{p(\cdot),\infty}$.
\item[{\rm (iii)}] $W \in A_{1,\infty}$ if and only if $w \in A_\infty$.
\end{itemize}
\end{lemma}

Now, we prove Proposition \ref{Ap Apinfty}.
\begin{proof}[Proof of Proposition \ref{Ap Apinfty}]
Let $W \in \mathscr{A}_{p(\cdot)}$.
Then, by Remark \ref{rem cap}{\rm (iii)} with $W\in\mathscr{A}_{p(\cdot)}$,
we find that $W^{-1} \in \mathscr{A}_{p'(\cdot)}$.
From this and Proposition \ref{ext redu},
we deduce that, for any cube $Q$ in $\mathbb{R}^n$,
there exists a reducing operator $\widetilde{A_Q}$ of order $p'(\cdot)$ for $W^{-1}$.
Using Jensen's inequality and Lemmas \ref{eq reduc M},
\ref{est fQ}, and \ref{norm matrix},
we find that, for any cube $Q$ in $\mathbb{R}^n$,
\begin{align*}
&\exp\left( \fint_{Q} \log\left( \frac{1}{\|\mathbf{1}_Q\|_{L^{p(\cdot)}}} \left\|\, \left\| W(\cdot)W^{-1}(y) \right\| \mathbf{1}_Q \right\|_{L^{p(\cdot)}}
 \right)\,dy \right)\\
&\quad \leq \fint_{Q}  \frac{1}{\|\mathbf{1}_Q\|_{L^{p(\cdot)}}} \left\|\, \left\| W(\cdot)W^{-1}(y) \right\| \mathbf{1}_Q \right\|_{L^{p(\cdot)}} \,dy\\
&\quad \sim  \fint_{Q}  \left\| A_Q W^{-1} \right\| \,dy
\lesssim \frac{1}{\|\mathbf{1}_Q\|_{L^{p(\cdot)}}} \left\|\, \left\| A_Q W^{-1}(y) \right\| \mathbf{1}_Q\right\|_{L^{p(\cdot)}}
\sim  \left\|\widetilde{A_Q} A_Q  \right\| = \left\| A_Q \widetilde{A_Q} \right\|\\
&\quad \sim |Q|^{-1} \bigg\|\,\left\|\, \left\| W(x)W^{-1}(\cdot) \right\|\mathbf{1}_Q \right\|_{L^{p'(\cdot)}} \mathbf{1}_{Q}(x)\bigg\|_{L^{p(\cdot)}_x},
\end{align*}
where $L^{p(\cdot)}_x$ indicates to take the norm with respect to the variable $x$.
This, together with the definitions of $\mathscr{A}_{p(\cdot),\infty}$ and $\mathscr{A}_{p(\cdot)}$,
further implies that $[W]_{\mathscr{A}_{p(\cdot),\infty}} \lesssim [W]_{\mathscr{A}_{p(\cdot)}}$
and hence $\mathscr{A}_{p(\cdot)} \subset \mathscr{A}_{p(\cdot),\infty}$.

Next, letting $W \in \mathscr{A}_{p(\cdot),\infty}$,
by Lemma \ref{est fQ},
we find that, for any cube $Q$ in $\mathbb{R}^n$,
\begin{align*}
&\exp\left( \fint_{Q} \log\left( \fint_Q \left\| W(\cdot)W^{-1}(y)\right\|\,dx  \right)\,dy \right)\\
&\quad \leq C_{p(\cdot),n} \exp\left( \fint_{Q} \log\left( \frac{1}{\|\mathbf{1}_Q\|_{L^{p(\cdot)}}} \left\|\, \left\| W(\cdot)W^{-1}(y) \right\| \mathbf{1}_Q \right\|_{L^{p(\cdot)}}\right)\,dy \right).
\end{align*}
Combining this with the definitions of $\mathscr{A}_{p(\cdot),\infty}$ and $A_{1,\infty}$,
we conclude that
$$[W]_{A_{1,\infty}} \leq C_{p(\cdot),n} [W]_{\mathscr{A}_{p(\cdot),\infty}}$$
and hence $\mathscr{A}_{p(\cdot),\infty} \subset A_{1,\infty}$.

Finally, if $p_- > 1$, then, by Proposition \ref{Ap Apinfty 2},
we find that there exist $w_1 \in \mathcal{A}_{p(\cdot),\infty} \setminus \mathcal{A}_{p(\cdot)}$
and $w_2 \in A_{\infty} \setminus \mathcal{A}_{p(\cdot),\infty}$.
Thus, using this and Lemma \ref{W w 1},
we obtain $W_1 := w_1I_{m} \in \mathscr{A}_{p(\cdot),\infty} \setminus \mathscr{A}_{p(\cdot)}$
and $W_2 := w_2I_{m} \in A_{1,\infty}\setminus \mathscr{A}_{p(\cdot),\infty}$.
This finishes the proof of Proposition \ref{Ap Apinfty}.
\end{proof}

As an application of Theorem \ref{reverse Holder Apinfty},
we  establish  the following reverse H\"older's inequality
for matrix-valued weights.

\begin{theorem}\label{reverse Holder M}
Let $p(\cdot)\in\mathcal{P}_0$ with $p(\cdot) \in LH$.
Then, for any $W \in \mathscr{A}_{p(\cdot),\infty}$,
there exist positive constants $C$, $C_1$, $C_2$,
$A$, and $A_1$,
depending only on $p(\cdot)$ and $n$,
such that, for any $r \in (1,r_w]$ with
$$ r_w := 1+ \frac{1}{C_1 [W]_{\mathscr{A}_{p(\cdot),\infty}}^{A_1} 2^{C_2[W]_{\mathscr{A}_{p(\cdot),\infty}}}}, $$
any cube $Q$ in $\mathbb{R}^n$, and any matrix $M \in M_{m}$,
\begin{align*}
\frac{1}{\|\mathbf{1}_{Q}\|_{L^{rp(\cdot)}}} \left\|\left\| W(\cdot)M \right\|\mathbf{1}_Q\right\|_{L^{rp(\cdot)}}
\leq C [W]_{\mathscr{A}_{p(\cdot),\infty}}^{A}
\frac{1}{\|\mathbf{1}_{Q}\|_{L^{p(\cdot)}}} \left\| \left\| W(\cdot)M \right\|\mathbf{1}_Q\right\|_{L^{p(\cdot)}}.
\end{align*}
\end{theorem}
Before giving the proof of Theorem \ref{reverse Holder M},
we present a key lemma, which connects matrix $\mathscr{A}_{p(\cdot),\infty}$ weights
with scalar $\mathcal{A}_{p(\cdot),\infty}$ weights.
\begin{lemma}\label{wz 1}
Let $p(\cdot) \in \mathcal{P}_0$ with $p(\cdot) \in LH$
and let $W\in \mathscr{A}_{p(\cdot),\infty}$.
Then, for any nonzero vector $\vec{z}\in \mathbb{C}^m$,
one has $w_{\vec{z}} := |W\vec{z}| \in \mathcal{A}_{p(\cdot),\infty}$ and
$[w_{\vec{z}}]_{\mathcal{A}_{p(\cdot),\infty}}
\leq [W]_{\mathscr{A}_{p(\cdot),\infty}}. $
\end{lemma}
\begin{proof}
By the definition of $w_{\vec{z}}$,
we obtain, for any cube $Q$ and any $x,y\in Q$,
\begin{align*}
w_{\vec{z}}(x) \leq \left\|W(x) W^{-1}(y)\right\| \left|W(y) \vec{z} \right|.
\end{align*}
Using this, we find that, for any $y\in Q$,
\begin{align*}
\frac{1}{\|\mathbf{1}_Q\|_{L^{p(\cdot)}}} \left\| w_{\vec{z}} \mathbf{1}_Q \right\|_{L^{p(\cdot)}}
\leq \frac{1}{\|\mathbf{1}_Q\|_{L^{p(\cdot)}}} \left\| \left\|W(x) W^{-1}(y)\right\| \mathbf{1}_Q \right\|_{L^{p(\cdot)}}w_{\vec{z}}(y),
\end{align*}
which further implies that
\begin{align*}
\log\left( \frac{1}{\|\mathbf{1}_Q\|_{L^{p(\cdot)}}} \left\| w_{\vec{z}} \mathbf{1}_Q \right\|_{L^{p(\cdot)}} \right) + \log\left(w_{\vec{z}}^{-1}(y)\right)
&\leq \log\left(\frac{1}{\|\mathbf{1}_Q\|_{L^{p(\cdot)}}} \left\| \left\|W(x) W^{-1}(y)\right\| \mathbf{1}_Q \right\|_{L^{p(\cdot)}} \right).
\end{align*}
From this, we infer that
\begin{align*}
&\frac{1}{\|\mathbf{1}_Q\|_{L^{p(\cdot)}}} \left\| w_{\vec{z}} \mathbf{1}_Q \right\|_{L^{p(\cdot)}} \exp\left(\fint_Q \log(w_{\vec{z}}^{-1}(y))\,dy\right)\\
&\quad \leq \exp\left( \fint_Q \log\left(\frac{1}{\|\mathbf{1}_Q\|_{L^{p(\cdot)}}}\left\| \, \left\|W(\cdot) W^{-1}(y)\right\|\mathbf{1}_Q \right\|_{L^{p(\cdot)}}\,dx\right) \, dy\right).
\end{align*}
Thus, using this with the definitions of $\mathcal{A}_{p(\cdot),\infty}$ and $\mathscr{A}_{p(\cdot),\infty}$,
we conclude that
$[w_{\vec{z}}]_{\mathcal{A}_{p(\cdot),\infty}} \leq [W]_{\mathscr{A}_{p(\cdot),\infty}}.$
This finishes the proof of Lemma \ref{wM 3}.
\end{proof}
Now, we begin to prove Theorem \ref{reverse Holder M}.
\begin{proof}[Proof of Theorem \ref{reverse Holder M}]
Let $\{\vec{e_j}\}_{j = 1}^m$ be an orthonormal basis of $\mathbb{C}^m$.
Next, let $w_j := |WM\vec{e_j}|$ for any $j\in \{1,\dots,m\}$.
Then, by Lemma \ref{wz 1},
we find that, for any $ j\in\{1,\dots,m\} $,
\begin{align}\label{reverse j 1}
[w_j]_{\mathcal{A}_{p(\cdot),\infty}} \leq [W]_{\mathscr{A}_{p(\cdot),\infty}} .
\end{align}
From this and Theorem \ref{reverse Holder Apinfty},
it follows that there exist positive constants $C_1$, $C_2$,
$A$, and $A_1$, depending only on $p(\cdot)$ and $n$,
such that, for any $ j\in\{1,\dots,m\} $ and $r \in [1,r_j ]$ with
$$ r_j := 1+ \frac{1}{C_1 [w_j]_{\mathcal{A}_{p(\cdot),\infty}}^{A_1} 2^{C_2[w_j]_{\mathcal{A}_{p(\cdot),\infty}}}} $$
and for any cube $Q$,
\begin{align}\label{reverse j}
\frac{1}{\|\mathbf{1}_{Q}\|_{L^{rp(\cdot)}}} \left\| w_j \mathbf{1}_Q\right\|_{L^{rp(\cdot)}}
\lesssim [w_j]_{\mathcal{A}_{p(\cdot),\infty}}^{A}
\frac{1}{\|\mathbf{1}_{Q}\|_{L^{p(\cdot)}}} \left\| w_j \mathbf{1}_Q\right\|_{L^{p(\cdot)}}.
\end{align}
Now, let
$$ r_W := 1+ \frac{1}{C_1 [W]_{\mathscr{A}_{p(\cdot),\infty}}^{A_1} 2^{C_2[W]_{\mathscr{A}_{p(\cdot),\infty}}}}, $$
which, combined with \eqref{reverse j 1}, is less than $ r_j $ for any $j\in\{1,\dots,m\}$.
Using this, \cite[Lemma 3.2]{r03}, \eqref{reverse j}, and \eqref{reverse j 1},
we conclude that,
for any $r \in [1,r_W]$ and any cube $Q$ in $\mathbb{R}^n$,
\begin{align*}
&\frac{1}{\|\mathbf{1}_{Q}\|_{L^{rp(\cdot)}}} \left\|\left\| W(\cdot) M \right\|\mathbf{1}_Q\right\|_{L^{rp(\cdot)}}\\
&\quad \sim \sum_{j = 1}^m \frac{1}{\|\mathbf{1}_{Q}\|_{L^{rp(\cdot)}}} \left\|\left\| W(\cdot) M \vec{e_j} \right\|\mathbf{1}_Q\right\|_{L^{rp(\cdot)}}
= \sum_{j = 1}^m \frac{1}{\|\mathbf{1}_{Q}\|_{L^{rp(\cdot)}}} \left\| w_j \mathbf{1}_Q\right\|_{L^{rp(\cdot)}} \\
&\quad\lesssim \sum_{j = 1}^m [w_j]_{\mathcal{A}_{p(\cdot),\infty}}^{A}
\frac{1}{\|\mathbf{1}_{Q}\|_{L^{p(\cdot)}}} \left\| w_j \mathbf{1}_Q\right\|_{L^{p(\cdot)}}
\lesssim [W]_{\mathscr{A}_{p(\cdot),\infty}}^{A}
\frac{1}{\|\mathbf{1}_{Q}\|_{L^{p(\cdot)}}} \left\| \left\| W(\cdot)M \right\|\mathbf{1}_Q\right\|_{L^{p(\cdot)}}.
\end{align*}
This finishes the proof of Theorem \ref{reverse Holder M}.
\end{proof}

The following result is an extension of \eqref{wM 2}
from scalar-valued weights to matrix-valued weights.

\begin{proposition}\label{WM reverse 1}
Let $p(\cdot) \in \mathcal{P}$ with $p_- > 1$ and $p(\cdot) \in LH$.
Then there exists a positive constant $C$ such that,
for any $W \in \mathscr{A}_{p(\cdot),\infty}$, any matrix $M \in M_m$, and any cube $Q$ in $\mathbb{R}^n$
\begin{align}\label{WM reverse eq}
\frac{1}{\|\mathbf{1}_Q\|_{L^{p(\cdot)}}} \left\| \,\left\|W(\cdot)M\right\| \mathbf{1}_Q \right\|_{L^{p(\cdot)}}
\leq C [W]_{\mathscr{A}_{p(\cdot),\infty}} \fint_Q \left\|W(x)M\right\|\,dx.
\end{align}
\end{proposition}

To prove Proposition \ref{WM reverse 1},
we need the following result,
which is exactly Lemma \ref{wz 1} with vector $\vec{z}$ replaced by matrix $M$.
We omit the details here.

\begin{lemma}\label{wM 1}
Let $p(\cdot) \in \mathcal{P}_0$ with $p(\cdot) \in LH$
and let $W\in \mathscr{A}_{p(\cdot),\infty}$.
Then, for any nonzero matrix $M\in M_m$,
one has $w_M := \|WM\| \in A_\infty$ and
$[w_M]_{\mathcal{A}_{p(\cdot),\infty}}
\leq [W]_{\mathscr{A}_{p(\cdot),\infty}}. $
\end{lemma}
Next, we give the proof of Proposition \ref{WM reverse 1}.
\begin{proof}[Proof of Proposition \ref{WM reverse 1}]
By Lemma \ref{wM 1}, we find that,
for any $M\in M_m$ with $w_M := \|WM\|$,
$ w_M \in \mathcal{A}_{p(\cdot),\infty} $ and $[w_M]_{\mathcal{A}_{p(\cdot),\infty}} \leq [W]_{\mathscr{A}_{p(\cdot),\infty}}$.
Using this and Theorem \ref{w reverse},
we obtain, for any cube $Q$ in $\mathbb{R}^n$,
\begin{align*}
\frac{1}{\|\mathbf{1}_Q\|_{L^{p(\cdot)}}} \left\|\,\left\| W(\cdot)M \right\| \mathbf{1}_Q \right\|_{L^{p(\cdot)}}
& = \frac{1}{\|\mathbf{1}_Q\|_{L^{p(\cdot)}}} \left\| w_M \mathbf{1}_Q \right\|_{L^{p(\cdot)}}
\leq [w_M]_{\mathcal{A}_{p(\cdot),\infty}} \fint_{Q} w_M(x)\,dx\\
&\leq [W]_{\mathscr{A}_{p(\cdot),\infty}} \fint_{Q} w_M(x)\,dx.
\end{align*}
This finishes the proof of Proposition \ref{WM reverse 1}.
\end{proof}
Using this, we obtain the following characterization of $\mathscr{A}_{p(\cdot),\infty}$.
\begin{proposition}\label{WM reverse 2}
Let $p(\cdot) \in \mathcal{P}$ with $p_- > 1$ and $p(\cdot) \in LH$.
Then, for any matrix weight $W$,
$W \in \mathscr{A}_{p(\cdot),\infty}$ if and only if
$W\in \mathscr{A}_{1,\infty}$ and \eqref{WM reverse eq} holds.
\end{proposition}

\begin{remark}
Proposition \ref{WM reverse 2}
implies
$\mathscr{A}_{1,\infty} = \bigcup_{p\in(1,\infty)} \mathscr{A}_{p,\infty}.$
\end{remark}

Now, we give the proof of Proposition \ref{WM reverse 2}.
\begin{proof}[Proof of Proposition \ref{WM reverse 2}]
We first show the necessity.
Indeed, by Lemma \ref{est fQ},
we find that, for any $W \in \mathscr{A}_{p(\cdot),\infty}$ and any cube $Q$ in $\mathbb{R}^n$,
\begin{align*}
\exp\left( \fint_Q \log\left( \fint_Q \left\| W(x)W^{-1}(y) \right\|\,dx \right)\,dy \right)
\lesssim \exp\left( \fint_Q \log\left( \frac{1}{\|\mathbf{1}_Q\|_{L^{p(\cdot)}}} \left\|\,\left\| W(\cdot)W^{-1}(y) \right\|\mathbf{1}_Q\right\|_{L^{p(\cdot)}} \right)\,dy \right),
\end{align*}
which further implies that
$ [W]_{\mathscr{A}_{1,\infty}} \lesssim [W]_{\mathscr{A}_{p(\cdot),\infty}} $.
Thus, using this and Proposition \ref{WM reverse 1},
we conclude that, for any $W \in \mathscr{A}_{p(\cdot),\infty}$,
$W\in \mathscr{A}_{1,\infty}$ and \eqref{WM reverse eq} holds.
This finishes the proof of the necessity.

Next, we consider the sufficiency.
From Proposition \ref{WM reverse 1},
it follows immediately that, for any cube $Q$ in $\mathbb{R}^n$,
\begin{align*}
&\exp\left( \fint_Q \log\left( \frac{1}{\|\mathbf{1}_Q\|_{L^{p(\cdot)}}} \left\|\,\left\| W(\cdot)W^{-1}(y) \right\|\mathbf{1}_Q\right\|_{L^{p(\cdot)}} \right)\,dy \right) \\
&\quad \lesssim \exp\left( \fint_Q \log\left( \fint_Q \left\| W(x)W^{-1}(y) \right\|\,dx \right)\,dy \right) < \infty,
\end{align*}
which further implies that $ W \in \mathscr{A}_{p(\cdot),\infty} $.
This finishes the proof of the sufficiency
and hence Proposition \ref{WM reverse 2}.
\end{proof}

\subsection{Upper and Lower $\mathscr{A}_{p(\cdot),\infty}$ Dimensions}\label{sec Apdimesnion}
In this subsection, we introduce and study  the upper and the lower weight dimensions of
$\mathscr{A}_{p(\cdot),\infty}$ weights
(see \cite{bhyy23} for weight  dimensions of matrix $A_p$ weights
and \cite{bhyy-a} for weight  dimensions of matrix $A_{p,\infty}$ weights).
Recall that the concept of weight dimensions of matrix $A_p$ weights
was first introduced by Bu et al. in \cite{bhyy23} and has proved to be a key tool
in the  study of  matrix-weighted Besov-type and Triebel--Lizorkin-type spaces.
Thus, weight dimensions of
$\mathscr{A}_{p(\cdot),\infty}$ weights play a key role in the
study of function spaces with matrix $\mathscr{A}_{p(\cdot),\infty}$ weights (see \cite{yyz25}).

\begin{definition}
Let $p(\cdot) \in \mathcal{P}_0$ and $d\in\mathbb{R}$.
A matrix weight $W$ is said to have \emph{$\mathscr{A}_{p(\cdot),\infty}$-lower dimension} $d$,
denoted by $W \in \mathbb{D}^{\rm lower}_{p(\cdot),\infty,d}$,
if there exists a positive constant $C$ such that,
for any $\lambda \in [1,\infty)$ and any cube $Q \subset \mathbb{R}^n$,
\begin{align*}
\exp\left( \fint_{\lambda Q} \log\left( \frac{1}{\|\mathbf{1}_Q\|_{L^{p(\cdot)}}} \left\|\, \left\| W(\cdot)W^{-1}(y) \right\| \mathbf{1}_Q \right\|_{L^{p(\cdot)}}
 \right)\,dy \right) \leq C \lambda^d.
\end{align*}
A matrix weight $W$ is said to have \emph{$\mathscr{A}_{p(\cdot),\infty}$-upper dimension} $d$,
denoted by $W \in \mathbb{D}^{\rm upper}_{p(\cdot),\infty,d}$,
if there exists a positive constant $C$ such that,
for any $\lambda \in [1,\infty)$ and any cube $Q \subset \mathbb{R}^n$,
\begin{align*}
\exp\left( \fint_{Q} \log\left( \frac{1}{\|\mathbf{1}_{\lambda Q}\|_{L^{p(\cdot)}}} \left\|\, \left\| W(\cdot)W^{-1}(y) \right\| \mathbf{1}_{\lambda Q} \right\|_{L^{p(\cdot)}}
 \right)\,dy \right) \leq C \lambda^d.
\end{align*}
\end{definition}

We have the following basic properties
(see \cite{bhyy-a} for the corresponding ones of $A_{p,\infty}$ weights).
\begin{proposition}\label{dim ext}
Let $p(\cdot) \in \mathcal{P}_0$ with $p(\cdot) \in LH$.
Then the following statements hold:
\begin{itemize}
\item[{\rm (i)}] For any $d \in (-\infty,0)$, $\mathbb{D}^{\rm lower}_{p(\cdot),\infty,d} = \emptyset$
and $\mathbb{D}^{\rm upper}_{p(\cdot),\infty,d} = \emptyset$.
\item[{\rm (ii)}] For any $W \in \mathscr{A}_{p(\cdot),\infty}$,
there exists $d_1 \in [0,\frac{n}{p_-})$
such that $W \in \mathbb{D}^{\rm lower}_{p(\cdot),\infty,d_1}$.
\item[{\rm (iii)}] For any $W \in \mathscr{A}_{p(\cdot),\infty}$,
there exists $d_2 \in [0,\infty)$
such that $W \in \mathbb{D}^{\rm upper}_{p(\cdot),\infty,d_2}$.
\end{itemize}
\end{proposition}

\begin{remark}\label{rem dim ext}
If $p(\cdot)=p$ is a constant exponent,
then Proposition \ref{dim ext}{\rm (ii)}
shows that, for any $W\in\mathscr{A}_{p,\infty}$,
$ d^{\rm lower}_{p(\cdot),\infty}(W) \in [0,\frac{n}{p}) $.
Notice that, from Remark \ref{rem Apinfty}{\rm (i)},
it follows that $W \in \mathscr{A}_{p,\infty}$ if and only if $\widetilde{W} := W^p \in A_{p,\infty}$.
Hence, by this and Proposition \ref{dim ext}{\rm (ii)},
we find that, if $\widetilde{W} \in A_{p,\infty}$ and $d \in \mathbb{R}$ satisfies
\begin{align*}
\exp\left( \fint_{\lambda Q} \log\left( \fint_Q \left\| \widetilde{W}^{\frac1p}(x)\widetilde{W}^{-\frac1p}(y) \right\|^p \,dx \right)\,dy \right) \lesssim \lambda^d,
\end{align*}
then $d \in (0,n)$,
which is precisely \cite[Proposition 6.3{\rm (ii)}]{bhyy23}.
\end{remark}
For the proof of Proposition \ref{dim ext},
we first show several technical lemmas, which give some necessary tools.
The following lemma is an extension of \cite[Lemma 2.11(ii)]{d15}
with the condition $p(\cdot) \in \mathcal{P}$ replaced by $p(\cdot)\in \mathcal{P}_0$,
which is immediately obtained by using Lemma \ref{con f};
we omit the details here.

\begin{lemma}\label{QP1}
Let $p(\cdot)\in \mathcal{P}_0$ with $p(\cdot) \in LH$.
Then  there exist positive constants $C$ and $\widetilde{C}$,
depending only on $p(\cdot)$ and $n$,
such that, for any cubes $Q$ and $R$ with $R\subset Q$,
$$ C \left( \frac{|Q|}{|R|} \right)^\frac{1}{p_+} \leq \frac{\|\mathbf{1}_{Q}\|_{L^{p(\cdot)}}}{\|\mathbf{1}_{R}\|_{L^{p(\cdot)}}}
\leq \widetilde{C} \left( \frac{|Q|}{|R|} \right)^\frac{1}{p_-}. $$
\end{lemma}

The following lemma gives an estimate of $A_QW^{-1}$ used later;
see \cite[Corollary 3.9]{bhyy23} for the corresponding result for $A_{p,\infty}$ matrix weights.

\begin{lemma}\label{Wpinfty}
Let $p(\cdot) \in \mathcal{P}_0$ with $p(\cdot)\in LH$.
Then there exists a positive constant $C$
such that, for any $W\in\mathscr{A}_{p(\cdot),\infty}$, any cube $Q$ in $\mathbb{R}^n$, and any $M\in (0,\infty)$,
\begin{align}\label{Wpinfty eq}
\left|\left\{ y\in Q:  \left\|A_QW^{-1}(y)\right\| \geq e^M \right\}\right| \leq \frac{\log(C[W]_{\mathscr{A}_{p(\cdot),\infty}})}{M}|Q|,
\end{align}
where $A_Q$ is the reducing operator of order $p(\cdot)$ for $W$ on $Q$.
\end{lemma}
\begin{proof}
Let $E_Q := \{y\in Q:  \|A_Q W^{-1}(y)\| \geq e^M\}$.
From Proposition \ref{eq def infty lem} with
taking $H := W^{-1}\mathbf{1}_{Q\cap E_Q} + A_Q^{-1}\mathbf{1}_{Q\setminus E_Q}$,
we deduce that
\begin{align}\label{Wpinfty 0}
[W]_{\mathscr{A}_{p(\cdot),\infty}}
& \gtrsim \left[ \frac{1}{\|\mathbf{1}_Q\|_{L^{p(\cdot)}}} \left\|\,\left\| W(\cdot)H(\cdot) \right\|\mathbf{1}_Q \right\|_{L^{p(\cdot)}} \right]^{-1}\nonumber\\
&\quad \times \exp\left( \fint_{Q} \log\left( \frac{1}{\|\mathbf{1}_Q\|_{L^{p(\cdot)}}} \left\|\, \left\| W(\cdot)H(y)\right\| \mathbf{1}_Q \right\|_{L^{p(\cdot)}}
 \right)\,dy \right)\nonumber\\
& =: {\rm I}(Q)\times {\rm II}(Q),
\end{align}
where
$$ {\rm I}(Q) := \left[ \frac{1}{\|\mathbf{1}_Q\|_{L^{p(\cdot)}}} \left\|\,\left\| W(\cdot)H(\cdot) \right\|\mathbf{1}_Q \right\|_{L^{p(\cdot)}} \right]^{-1} $$
and
$$ {\rm II}(Q) := \exp\left( \fint_{Q} \log\left( \frac{1}{\|\mathbf{1}_Q\|_{L^{p(\cdot)}}} \left\|\, \left\| W(\cdot)H(y)\right\| \mathbf{1}_Q \right\|_{L^{p(\cdot)}} \right)\,dy \right). $$
Now, using Lemma \ref{eq reduc M} and the fact that $\|\cdot\|_{L^{p(\cdot)}}$ is a quasi-norm,
we obtain
\begin{align}\label{Wpinfty 1}
\left[{\rm I}(Q)\right]^{-1}
&\lesssim \frac{1}{\|\mathbf{1}_Q\|_{L^{p(\cdot)}}} \left\|\mathbf{1}_{Q\cap E_Q} \right\|_{L^{p(\cdot)}}
+ \frac{1}{\|\mathbf{1}_Q\|_{L^{p(\cdot)}}} \left\|\,\left\| W(\cdot)A_Q^{-1} \right\|\mathbf{1}_{Q\setminus E_Q} \right\|_{L^{p(\cdot)}}
\lesssim 1 + \left\| A_Q A_Q^{-1} \right\| \lesssim 1.
\end{align}
Then, by the definition of $H$, we find that, for any $y\in Q$,
\begin{align*}
&\log\left( \frac{1}{\|\mathbf{1}_Q\|_{L^{p(\cdot)}}}\left\|\, \left\| W(\cdot)H(y)\right\| \mathbf{1}_Q \right\|_{L^{p(\cdot)}} \right)\mathbf{1}_Q(y)\\
&\quad = \log\left(\frac{1}{\|\mathbf{1}_Q\|_{L^{p(\cdot)}}}\left\|\, \left\| W(\cdot)W^{-1}(y)\right\| \mathbf{1}_Q \right\|_{L^{p(\cdot)}}\right) \mathbf{1}_{Q \cap E_Q}(y)\\
&\quad\quad + \log\left(\frac{1}{\|\mathbf{1}_Q\|_{L^{p(\cdot)}}}\left\|\, \left\| W(\cdot)A_Q^{-1}\right\| \mathbf{1}_Q \right\|_{L^{p(\cdot)}}\right) \mathbf{1}_{Q\setminus E_Q}(y).
\end{align*}
Hence, using this, Lemma \ref{eq reduc M}, and the definition of $E_Q$,
we obtain
\begin{align*}
{\rm II}(Q)
&= \exp\left( |Q|^{-1} \int_{E_Q} \log\left( \frac{1}{\|\mathbf{1}_Q\|_{L^{p(\cdot)}}} \left\|\, \left\| W(\cdot)W^{-1}(y)\right\| \mathbf{1}_Q \right\|_{L^{p(\cdot)}} \right)\,dy \right)\\
&\quad\times \exp\left( |Q|^{-1} \int_{Q\setminus E_Q} \log\left( \frac{1}{\|\mathbf{1}_Q\|_{L^{p(\cdot)}}} \left\|\, \left\| W(\cdot)A_Q^{-1}\right\| \mathbf{1}_Q \right\|_{L^{p(\cdot)}} \right)\,dy\right)\\
&\sim \exp\left( |Q|^{-1} \int_{E_Q} \log\left( \left\| A_Q W^{-1}(y)\right\|
\right)\,dy \right)
\exp\left( |Q|^{-1} \int_{Q\setminus E_Q} \log\left( \left\| A_Q A_Q^{-1}\right\|
\right)\,dy \right)\\
&\geq \exp\left(\frac{M|E_Q|}{|Q|}\right).
\end{align*}
From this, \eqref{Wpinfty 0}, and \eqref{Wpinfty 1},
it follows immediately that
$\exp(\frac{M|E_Q|}{|Q|}) \lesssim [W]_{\mathscr{A}_{p(\cdot),\infty}},$
which further implies that \eqref{Wpinfty eq} holds.
This finishes the proof of Lemma \ref{Wpinfty}.
\end{proof}

The following result is a  matrix-valued analogue of  Lemma \ref{wM 1}
(see \cite[Lemma 5.3]{bhyy23} for a corresponding result of $A_{p,\infty}$ matrix weights).
Notice that, by \cite[Proposition 7.3.2(4)]{g14},
$[w]_{A_\infty} \geq 1$ for any $w\in A_\infty$.
Hence, the remaining proof of this result is similar  to the proof of Lemma \ref{wz 1}
with $\mathcal{A}_{p(\cdot),\infty}$ replaced by $A_\infty$.
We omit the details here.

\begin{lemma}\label{wM 3}
Let $p(\cdot) \in \mathcal{P}_0$ with $p(\cdot) \in LH$
and let $W\in \mathscr{A}_{p(\cdot),\infty}$.
Let $r := \min\{1,p_-\}$.
Then, for any nonzero matrix $M\in M_m$,
one has $w_M^r := \|WM\|^r \in A_\infty$ and
there exists a positive constant $C$, independent of $W$ and $M$,
such that
$1 \leq [w_M^r]^{\frac1r}_{A_\infty} \leq C[W]_{\mathscr{A}_{p(\cdot),\infty}}.$
\end{lemma}

Next, we begin to prove Proposition \ref{dim ext}.
\begin{proof}[Proof of Proposition \ref{dim ext}]
We begin with the proof of {\rm (i)}.
We first prove that $ \mathbb{D}^{\rm lower}_{p(\cdot),\infty,d} = \emptyset $ for any $d\in(-\infty,0)$.
We do this by contradiction.
Assume that there exists $d \in (-\infty,0)$
such that $ \mathbb{D}^{\rm lower}_{p(\cdot),\infty,d} \neq \emptyset $,
that is, there exists $W \in \mathbb{D}^{\rm lower}_{p(\cdot),\infty,d}$.
Let $Q_{\mathbf{0},0} := Q(\mathbf{0},1)$.
From Definition \ref{def matrix}{\rm (iii)},
we infer that $\|W(\cdot) A_{Q_{\mathbf{0},0}}^{-1}\|^r$ is locally integrable
with $r := \min\{1,p_-\}$,
which, together with the Lebesgue differentiation theorem
(see, for instance, \cite[Corollary 2.1.16]{g14}),
further implies that
there exists an interior point $x_0 \in Q_{\mathbf{0},0}$
such that $W(x_0)$ is invertible
and there exists a sequence of cubes $\{Q_k\}_{k\in\mathbb{Z}_+}$
satisfying that $x_0$ is the center of each $Q_k$,
$\lim_{k\to \infty} l(Q_k) = 0$,
and
$$ \left\|W(x_0) A_{Q_{\mathbf{0},0}}^{-1} \right\|^r
= \lim\limits_{k\to \infty} \fint_{Q_k} \left\|W(x) A_{Q_{\mathbf{0},0}}^{-1} \right\|^r \,dx. $$
Using this, Lemmas \ref{eq reduc M} and \ref{est fQ} with $ \frac{p(\cdot)}{r} \in \mathcal{P} $,
and the assumption $W \in \mathbb{D}^{\rm lower}_{p(\cdot),\infty,d}(\mathbb{R}^n)$,
we conclude that
\begin{align*}
\left\|W(x_0) A_{Q_{\mathbf{0},0}}^{-1} \right\|^r
&= \lim\limits_{k\to \infty} \fint_{Q_k} \left\|W(x) A_{Q_{\mathbf{0},0}}^{-1} \right\|^r\,dx
\lesssim \lim\limits_{k\to \infty} \frac{1}{\|\mathbf{1}_{Q_k}\|_{L^{\frac{p(\cdot)}{r}}}} \left\|\, \left\|W(x) A_{Q_{\mathbf{0},0}}^{-1} \right\|^r\mathbf{1}_{Q_k}\right\|_{L^{\frac{p(\cdot)}{r}}}\\
& = \lim\limits_{k\to \infty} \left[\frac{1}{\|\mathbf{1}_{Q_k}\|_{L^{p(\cdot)}}} \left\|\, \left\|W(x) A_{Q_{\mathbf{0},0}}^{-1} \right\|\mathbf{1}_{Q_k}\right\|_{L^{p(\cdot)}}\right]^r
\sim \lim\limits_{k\to \infty} \left\| A_{Q_k} A_{Q_{\mathbf{0},0}}^{-1} \right\|^r\\
&\sim \lim\limits_{k\to \infty} \left[\exp\left( \fint_{Q_{\mathbf{0},0}} \log\left(\frac{1}{\|\mathbf{1}_{Q_k}\|_{L^{p(\cdot)}}}\left\| \, \left\|W(\cdot) W^{-1}(y)\right\|\mathbf{1}_{Q_k} \right\|_{L^{p(\cdot)}}\,dx\right) \, dy\right)\right]^r\\
&\lesssim \lim\limits_{k\to \infty} \left[ \frac{l(Q_{\mathbf{0},0})}{l(Q_k)} \right]^{rd} = 0.
\end{align*}
Hence, by the fact that $A_{Q_{\mathbf{0},0}}$ is invertible,
we find that $W(x_0) = 0$, which contradicts the invertibility of $W(x_0)$.
Thus, $\mathbb{D}^{\rm lower}_{p(\cdot),\infty,d} = \emptyset$.

Now, we prove that $ \mathbb{D}^{\rm upper}_{p(\cdot),\infty,d} = \emptyset$ for any $d \in (-\infty,0)$.
We do this again by contradiction.
To this end, we assume that there exists $d \in (-\infty,0)$
such that $\mathbb{D}^{\rm upper}_{p(\cdot),\infty,d} \neq \emptyset$,
that is, there exists $W \in \mathbb{D}^{\rm upper}_{p(\cdot),\infty,d}$.
By the definition of $\mathscr{A}_{p(\cdot),\infty}$ and \eqref{eq reduc M},
we find that
\begin{align*}
&\exp\left( \fint_{Q_{\mathbf{0},0}} \log\left( \left\| W^{-1}(y) A_{Q_{\mathbf{0},0}} \right\| \right)\,dy \right)\\
&\quad \sim \exp\left( \fint_{Q_{\mathbf{0},0}} \log\left( \frac{1}{\|\mathbf{1}_{Q_{\mathbf{0},0}}\|_{L^{p(\cdot)}}} \left\|\, \left\| W(\cdot)W^{-1}(y) \right\| \mathbf{1}_{Q_{\mathbf{0},0}} \right\|_{L^{p(\cdot)}} \right)\,dy \right)
 < \infty,
\end{align*}
which further implies that
$ \log( \| W^{-1}(y) A_{Q_{\mathbf{0},0}} \| ) \in L^1(Q_{\mathbf{0},0})$.
Using this and the Lebesgue differentiation theorem,
we conclude that there exists an interior point $x_0 \in Q_{\mathbf{0},0}$
such that $W(x_0)$ is invertible
and there exists a sequence of cubes $\{Q_k\}_{k\in\mathbb{Z}_+}$
such that $x_0$ is the center of every $Q_k$,
$\lim_{k\to \infty} l(Q_k) = 0$,
and
$$ \left\|W^{-1}(x_0) A_{Q_{\mathbf{0},0}} \right\|
= \lim\limits_{k\to \infty} \exp \left[ \fint_{Q_k} \log\left( \left\|W^{-1}(y) A_{Q_{\mathbf{0},0}} \right\|\right) \,dy\right]. $$
Let $Q_{x_0} := Q(x_0, 2)$ be a fixed cube.
Then $Q_{\mathbf{0},0} \subset Q_{x_0}$.
From this, Lemma \ref{QP1}, and the definition of $\mathbb{D}^{\rm upper}_{p(\cdot),\infty,d}$,
it follows that
\begin{align*}
\left\|W^{-1}(x_0) A_{Q_{\mathbf{0},0}} \right\|
&= \lim\limits_{k\to \infty} \exp\left( \fint_{Q_k} \log\left(\left\|W(y)^{-1} A_{Q_{\mathbf{0},0}} \right\|\right) \,dy \right)\\
&\sim \lim\limits_{k\to \infty} \exp\left( \fint_{Q_k} \log\left( \frac{1}{\|\mathbf{1}_{Q_{\mathbf{0},0}}\|_{L^{p(\cdot)}}} \left\|\, \left\| W(\cdot)W^{-1}(y) \right\| \mathbf{1}_{Q_{\mathbf{0},0}} \right\|_{L^{p(\cdot)}} \right)\,dy \right)\\
&\lesssim \lim\limits_{k\to \infty} \exp\left( \fint_{Q_k} \log\left( \frac{1}{\|\mathbf{1}_{Q_{x_0}}\|_{L^{p(\cdot)}}} \left\|\, \left\| W(\cdot)W^{-1}(y) \right\| \mathbf{1}_{Q_{x_0}} \right\|_{L^{p(\cdot)}} \right)\,dy \right)\\
&\lesssim \lim\limits_{k\to \infty} \left[ l(Q_k) \right]^{-d} = 0,
\end{align*}
which further implies that $W^{-1}(x_0) = 0$
and hence contradicts the fact that $ W^{-1}(x_0) $ is invertible.
Thus, $\mathbb{D}^{\rm upper}_{p(\cdot),\infty,d} = \emptyset$.
This finishes the proof of {\rm (i)}.

Secondly, we prove {\rm (ii)}.
Let $\lambda \in [1,\infty)$ and $Q$ be a cube in $\mathbb{R}^n$.
From Theorem \ref{reverse Holder M},
we deduce that there exists $r \in (1, r_W)$,
where $r_W$ is the same as in Theorem \ref{reverse Holder M},
such that, for any matrix $M \in M_m$ and any cube $Q$ in $\mathbb{R}^n$,
we have
\begin{align*}
\frac{1}{\|\mathbf{1}_{Q}\|_{L^{rp(\cdot)}}} \left\|\left\| W(\cdot)M \right\|\mathbf{1}_Q\right\|_{L^{rp(\cdot)}}
\lesssim \frac{1}{\|\mathbf{1}_{Q}\|_{L^{p(\cdot)}}} \left\| \left\| W(\cdot)M \right\|\mathbf{1}_Q\right\|_{L^{p(\cdot)}}.
\end{align*}
Then, by this with $M:= W^{-1}(y)$ and by Lemmas \ref{Holder} and \ref{QP1},
we find that
\begin{align*}
&\exp\left( \fint_{\lambda Q} \log\left( \frac{1}{\|\mathbf{1}_Q\|_{L^{p(\cdot)}}} \left\|\, \left\| W(\cdot)W^{-1}(y) \right\| \mathbf{1}_Q \right\|_{L^{p(\cdot)}} \right)\,dy \right)\\
&\quad \lesssim \exp\left( \fint_{\lambda Q} \log\left( \frac{1}{\|\mathbf{1}_Q\|_{L^{r p(\cdot)}}} \left\|\, \left\| W(\cdot)W^{-1}(y) \right\| \mathbf{1}_Q \right\|_{L^{r p(\cdot)}} \right)\,dy \right)\\
&\quad = \frac{\|\mathbf{1}_{\lambda Q}\|_{L^{r p(\cdot)}}}{\|\mathbf{1}_Q\|_{L^{r p(\cdot)}}} \exp\left( \fint_{\lambda Q} \log\left( \frac{1}{\|\mathbf{1}_{\lambda Q}\|_{L^{r p(\cdot)}}} \left\|\, \left\| W(\cdot)W^{-1}(y) \right\| \mathbf{1}_{\lambda Q} \right\|_{L^{r p(\cdot)}} \right)\,dy \right)\\
&\quad \lesssim \left(\frac{|\lambda Q|}{|Q|}\right)^{\frac{1}{r p_-}}
\exp\left( \fint_{\lambda Q} \log\left( \frac{1}{\|\mathbf{1}_{\lambda Q}\|_{L^{p(\cdot)}}} \left\|\, \left\| W(\cdot)W^{-1}(y) \right\| \mathbf{1}_{\lambda Q} \right\|_{L^{p(\cdot)}} \right)\,dy \right)
\lesssim  \lambda^{\frac{n}{r p_-}},
\end{align*}
which, combined with the definitions of $\mathscr{A}_{p(\cdot),\infty}$-lower dimensions,
further implies that $W$ has $\frac{n}{r p_-}$-$\mathscr{A}_{p(\cdot),\infty}$-lower dimension.
Since $r > 1$, we infer that $ \frac{n}{r p_-} < \frac{n}{p_-} $.
This finishes the proof of {\rm (ii)}.

Finally, we prove {\rm (iii)}.
Let $M \in (0,\infty)$ be a constant determined later
and, for any cube $Q$ in $\mathbb{R}^n$ and any $x\in\mathbb{R}^n$,
let
\begin{align*}
h_Q(x) :=
\begin{cases}
\displaystyle M &\text{if}\ x\in Q,\\
\displaystyle 1 &\text{if}\ x\in\mathbb{R}^n\setminus Q.
\end{cases}
\end{align*}
From Proposition \ref{eq def infty lem} with
replacing $Q$ and $H$, respectively, by $2Q$ and $h_QA_Q^{-1}$,
we deduce that
\begin{align}\label{Wpinfty 3}
[W]_{\mathscr{A}_{p(\cdot),\infty}}
& \gtrsim \left[ \frac{1}{\|\mathbf{1}_{2Q}\|_{L^{p(\cdot)}}} \left\|\,\left\| W(\cdot)A_Q^{-1} \right\|h_Q(\cdot)\mathbf{1}_{2Q} \right\|_{L^{p(\cdot)}} \right]^{-1} \nonumber\\
&\quad \times \exp\left( \fint_{2Q} \log\left( \frac{1}{\|\mathbf{1}_{2Q}\|_{L^{p(\cdot)}}} \left\|\, \left\| W(\cdot)A_Q^{-1}\right\| \mathbf{1}_{2Q} \right\|_{L^{p(\cdot)}} h_Q(y) \right)\,dy \right)\nonumber\\
&=: {\rm I}(Q)\times {\rm II}(Q),
\end{align}
where
$$ {\rm I}(Q) := \left[ \frac{1}{\|\mathbf{1}_{2Q}\|_{L^{p(\cdot)}}} \left\|\,\left\| W(\cdot)A_Q^{-1} \right\|h_Q(\cdot)\mathbf{1}_{2Q} \right\|_{L^{p(\cdot)}} \right]^{-1} $$
and
$$ {\rm II}(Q) := \exp\left( \fint_{2Q} \log\left( \frac{1}{\|\mathbf{1}_{2Q}\|_{L^{p(\cdot)}}} \left\|\, \left\| W(\cdot)A_Q^{-1}\right\| \mathbf{1}_{2Q} \right\|_{L^{p(\cdot)}} h_Q(y) \right)\,dy \right). $$
Next, by Lemmas \ref{eq reduc M} and \ref{QP1},
we obtain
\begin{align}\label{Wpinfty 4}
\left[{\rm I}(Q)\right]^{-1}
&\lesssim \frac{1}{\|\mathbf{1}_{2Q}\|_{L^{p(\cdot)}}} \left\| \,\left\| W(\cdot)A_Q^{-1} \right\|\mathbf{1}_{2Q} \right\|_{L^{p(\cdot)}}
+ \frac{|M-1|}{\|\mathbf{1}_{2Q}\|_{L^{p(\cdot)}}} \left\|\,\left\| W(\cdot)A_Q^{-1} \right\|\mathbf{1}_{Q} \right\|_{L^{p(\cdot)}}\nonumber\\
&\sim \left\| A_{2Q}A_{Q}^{-1} \right\| + \left| M-1 \right| \frac{\|\mathbf{1}_{Q}\|_{L^{p(\cdot)}}}{\|\mathbf{1}_{2Q}\|_{L^{p(\cdot)}}}
 \lesssim \left\| A_{2Q}A_{Q}^{-1} \right\| + 2^{-\frac{n}{p_+}}\left| M-1 \right|,
\end{align}
where the implicit constant is independent of $W$.
Meanwhile,
\begin{align*}
{\rm II}(Q)
&=  \frac{1}{\|\mathbf{1}_{2Q}\|_{L^{p(\cdot)}}} \left\|\, \left\| W(\cdot)A_Q^{-1}\right\| \mathbf{1}_{2Q} \right\|_{L^{p(\cdot)}}
\exp\left(\fint_{2Q} \log\left( h_Q(y) \right)\,dy \right)
 \sim \left\| A_{2Q} A_Q^{-1} \right\| \exp\left( \frac{ \log M }{2^n} \right).
\end{align*}
Thus, combining this, \eqref{Wpinfty 3}, and \eqref{Wpinfty 4},
we conclude that
there exists a positive constant $C$, independent of $W$,
such that
\begin{align*}
[W]_{\mathscr{A}_{p(\cdot),\infty}} \geq C \frac{\| A_{2Q} A_Q^{-1} \|}{\| A_{2Q}A_{Q}^{-1} \| + 2^{-\frac{n}{p_+}}| M-1 | } \exp\left( \frac{ \log M }{2^n} \right).
\end{align*}
Now, fixing $M$ satisfying $CM^{\frac{1}{2^n}} = 2[W]_{\mathscr{A}_{p(\cdot),\infty}}$,
we then have
\begin{align}\label{Wpinfty 7}
\left\| A_{2Q} A_Q^{-1} \right\| \leq  2^{-\frac{n}{p_+}}\left| M-1 \right| .
\end{align}
Thus, by this,
if $M \in [1,\infty)$, then $|M-1| = M-1 $
and hence, for any cube $Q$ in $\mathbb{R}^n$,
one has
\begin{align*}
\left\|A_{2Q}A_Q^{-1}\right\| \leq (M-1) 2^{-\frac{n}{p_+}} < M2^{-\frac{n}{p_+}} = 2^{-\frac{n}{p_+}} \left\{ \frac{2[W]_{\mathscr{A}_{p(\cdot),\infty}}}{C} \right\}^{2^n}.
\end{align*}
Conversely, if $M \in (0,1)$, then, from Lemma \ref{wM 3},
it follows that there exists a positive constant $\widetilde{C}$
such that $1\leq \widetilde{C}[W]^{2^n}_{\mathscr{A}_{p(\cdot),\infty}}$.
Notice that, in this case, $|M-1| \leq 1 \leq \widetilde{C}[W]_{\mathscr{A}_{p(\cdot),\infty}}$.
Using this and \eqref{Wpinfty 7}, we conclude that
\begin{align*}
\left\|A_{2Q}A_Q^{-1}\right\| \leq |M-1| 2^{-\frac{n}{p_+}} \leq \widetilde{C}[W]^{2^n}_{\mathscr{A}_{p(\cdot),\infty}}.
\end{align*}
Thus, there exists a positive constant $C$
such that, for any cube $Q$ in $\mathbb{R}^n$,
\begin{align}\label{Wpinfty 6}
\left\|A_{2Q}A_Q^{-1}\right\| \leq \left( C^{-1} 2 [W]^{2^n}_{\mathscr{A}_{p(\cdot),\infty}} \right)^{2^n}.
\end{align}

Next, by using \eqref{Wpinfty 6},
we conclude that, for any $\lambda \in (1,\infty)$,
\begin{align*}
&\exp\left( \fint_{Q} \log\left( \frac{1}{\|\mathbf{1}_{\lambda Q}\|_{L^{p(\cdot)}}} \left\|\, \left\| W(\cdot)W^{-1}(y) \right\| \mathbf{1}_{\lambda Q} \right\|_{L^{p(\cdot)}} \right)\,dy \right)\\
&\quad \lesssim \exp\left( \fint_{Q} \log\left( \frac{1}{\|\mathbf{1}_{ 2^{\lceil \log_2 \lambda \rceil} Q}\|_{L^{p(\cdot)}}} \left\|\, \left\| W(\cdot)W^{-1}(y) \right\| \mathbf{1}_{2^{\lceil \log_2 \lambda \rceil} Q} \right\|_{L^{p(\cdot)}} \right)\,dy \right)\\
&\quad \sim \left\| A_{2^{\lceil \log_2 \lambda \rceil} Q} A_Q^{-1} \right\|
\leq \left( C^{-1} 2 [W]^{2^n}_{\mathscr{A}_{p(\cdot),\infty}} \right)^{2^n \lceil \log_2 \lambda \rceil} \\
&\quad \leq \left( C^{-1} 2 [W]^{2^n}_{\mathscr{A}_{p(\cdot),\infty}} \right)^{2^n(1 + \log_2 \lambda)}
\sim \lambda^{2^n(1-\log_2 C + \log_2 [W]_{\mathscr{A}_{p(\cdot),\infty}})},
\end{align*}
which further implies that $W$ has $\mathscr{A}_{p(\cdot),\infty}$-upper dimension.
This finishes the proof of {\rm (iii)}
and hence Proposition \ref{dim ext}.
\end{proof}

Let $p(\cdot)\in \mathcal{P}_0$ with $p(\cdot) \in LH$.
For any matrix weight $W \in \mathscr{A}_{p(\cdot),\infty}$,
let
\begin{align*}
d^{\rm lower}_{p(\cdot),\infty}(W) := \inf\left\{ d\in \left(0,\frac{n}{p_-}\right):  W\ \text{has}\ \mathscr{A}_{p(\cdot),\infty}\text{-lower dimension}\ d \right\}
\end{align*}
and
\begin{align*}
d^{\rm upper}_{p(\cdot),\infty}(W) := \inf\left\{ d\in (0,\infty):  W\ \text{has}\ \mathscr{A}_{p(\cdot),\infty}\text{-upper dimension}\ d \right\}.
\end{align*}
Let
\begin{align*}
[\![ d^{\rm lower}_{p(\cdot),\infty}(W), \infty) :=
\begin{cases}
\displaystyle \left[d^{\rm lower}_{p(\cdot),\infty}(W), \frac{n}{p_-}\right) &\text{if } d^{\rm lower}_{p(\cdot),\infty}(W)\ \text{is} \ \mathscr{A}_{p(\cdot),\infty}\text{-lower dimension of }W, \\
\displaystyle \left(d^{\rm lower}_{p(\cdot),\infty}(W), \frac{n}{p_-}\right) &\text{otherwise}
\end{cases}
\end{align*}
and
\begin{align*}
[\![ d^{\rm upper}_{p(\cdot),\infty}(W), \infty) :=
\begin{cases}
\displaystyle [d^{\rm upper}_{p(\cdot),\infty}(W), \infty) &\text{if } d^{\rm upper}_{p(\cdot),\infty}(W)\ \text{is} \ \mathscr{A}_{p(\cdot),\infty}\text{-upper dimension of }W, \\
\displaystyle (d^{\rm upper}_{p(\cdot),\infty}(W), \infty) &\text{otherwise}.
\end{cases}
\end{align*}

Using the concept of $\mathscr{A}_{p(\cdot),\infty}$-dimensions,
we obtain the following   estimate. A corresponding estimate for matrix $A_p$ weights
has proved to be vital in the study of matrix-weighted Besov-type and
Triebel--Lizorkin-type spaces; see \cite{bhyy23, bhyy23 2, bhyy23 3}.

\begin{proposition}\label{QP3}
Let $p(\cdot) \in \mathcal{P}_0$ with $p(\cdot) \in LH$, $W \in \mathscr{A}_{p(\cdot),\infty}$,
$d_1 \in [\![ d^{\rm lower}_{p(\cdot),\infty}(W) ,\frac{n}{p_-})$,
$d_2 \in [\![ d^{\rm upper}_{p(\cdot),\infty}(W) ,\infty)$,
and $\{A_Q\}_{Q}$ be a family of reducing operators of order $p(\cdot)$ for $W$.
Then there exists a positive constant $C$ such that,
for any cubes $Q$ and $R$ in $\mathbb{R}^n$,
\begin{align}\label{AQAR-1}
\left\| A_Q A_R^{-1} \right\| \leq C \max\left\{ \left[ \frac{l(R)}{l(Q)} \right]^{d_1},
\left[ \frac{l(Q)}{l(R)} \right]^{d_2} \right\}
\left[ 1+ \frac{|c_Q - c_R|}{l(Q)\vee l(R)} \right]^{\Delta},
\end{align}
where $\Delta := d_1 + d_2$.
\end{proposition}
\begin{remark}
If $p(\cdot) \equiv p$ is a constant exponent,
then it follows from Remark \ref{rem dim ext} that,
for any $W \in \mathscr{A}_{p,\infty}$, $\widetilde{W} := W^p \in A_{p,\infty}$.
For any cube $Q$ in $\mathbb{R}^n$, let $\widetilde{A}_Q$ be the reducing operator
of $\widetilde{W}$ as in \cite[(3.1)]{v97}.
Using this, \cite[Lemmas 2.9 and 3.5]{bhyy-a},
the definition of reducing operators, and Lemma \ref{AA-1 M},
we find that, for any cubes $Q, R$ in $\mathbb{R}^n$,
\begin{align*}
\left\|\widetilde{A}_Q \widetilde{A}_R^{-1}\right\|
\sim \exp\left( \fint_R \log \left( \fint_Q \left\| \widetilde{W}^{\frac1p}(x) \widetilde{W}^{-\frac1p}(y) \right\|^p\,dx \right) \,dy \right)
\sim \left\| A_Q A_R^{-1} \right\|^p.
\end{align*}
Thus, by this and Proposition \ref{QP3},
we conclude that, for any cubes $Q, R$ in $\mathbb{R}^n$,
\begin{align*}
\left\|\widetilde{A}_Q \widetilde{A}_R^{-1}\right\|
\sim \left\| A_Q A_R^{-1} \right\|^p
\lesssim \max\left\{ \left[ \frac{l(R)}{l(Q)} \right]^{pd_1},
\left[ \frac{l(Q)}{l(R)} \right]^{pd_2} \right\}
\left[ 1+ \frac{|c_Q - c_R|}{l(Q)\vee l(R)} \right]^{p\Delta},
\end{align*}
which, together with Remark \ref{rem dim ext},
is precisely \cite[Proposition 6.5]{bhyy-a}.
Recall that Bu et al. \cite[Section 7]{bhyy-a} has showed
that \cite[Proposition 6.5]{bhyy-a} is sharp.
In this sense, Proposition \ref{QP3} is also sharp.
\end{remark}
\begin{proof}[Proof of Proposition \ref{QP3}]
By the definition of reducing operators and Lemma \ref{AA-1 M},
we find that, for any cubes $Q, R$ in $\mathbb{R}^n$,
\begin{align}\label{QP2}
\left\| A_Q A^{-1}_R \right\|
\lesssim \exp\left( \fint_{R} \log\left( \frac{1}{\|\mathbf{1}_Q\|_{L^{p(\cdot)}}} \left\|\, \left\| W(\cdot)W^{-1}(y) \right\| \mathbf{1}_Q \right\|_{L^{p(\cdot)}}
 \right)\,dy \right).
\end{align}

Now, we first consider the special case $Q\subset R$ with the same center.
In this case, by \eqref{QP2} and the definition of $ \mathbb{D}^{\rm lower}_{p(\cdot),\infty,d} $,
we obtain
\begin{align}\label{QP eq 1}
\left\| A_Q A_R^{-1} \right\| & \lesssim \exp\left( \fint_{R} \log\left( \frac{1}{\|\mathbf{1}_Q\|_{L^{p(\cdot)}}} \left\|\, \left\| W(\cdot)W^{-1}(y) \right\| \mathbf{1}_Q \right\|_{L^{p(\cdot)}} \right)\,dy \right)
\lesssim \left[ \frac{l(R)}{l(Q)} \right]^{d_1}.
\end{align}

Next, we consider another special case $R \subset Q$.
In this case, it is obvious that $l(R) \leq l(Q)$
and hence there exists $\lambda \in [1,\infty)$
such that $\lambda \sim \frac{l(Q)}{l(R)}$
and $Q\subset \lambda R$.
Using this, \eqref{QP2}, and the definition of $ \mathbb{D}^{\rm upper}_{p(\cdot),\infty,d} $,
we find that
\begin{align}\label{QP eq 2}
\left\| A_Q A_R^{-1} \right\|
& \lesssim \exp\left( \fint_{R} \log\left( \frac{1}{\|\mathbf{1}_Q\|_{L^{p(\cdot)}}} \left\|\, \left\| W(\cdot)W^{-1}(y) \right\| \mathbf{1}_Q \right\|_{L^{p(\cdot)}} \right)\,dy \right) \nonumber\\
& \lesssim \exp\left( \fint_{R} \log\left( \frac{1}{\|\mathbf{1}_{\lambda R}\|_{L^{p(\cdot)}}} \left\|\, \left\| W(\cdot)W^{-1}(y) \right\| \mathbf{1}_{\lambda R} \right\|_{L^{p(\cdot)}} \right)\,dy \right)
 \lesssim \lambda^{d_2} \sim \left[ \frac{l(Q)}{l(R)} \right]^{d_2}.
\end{align}

In the general case,
we choose a third cube $S$ in $\mathbb{R}^n$, which has the same center as $Q$,
such that $Q\cup R \subset S$ and
$$l(S) \sim l(Q) + l(R) + |c_Q - c_R|.$$
By the geometrical observation,
this cube obviously exists.
From this, \eqref{QP eq 1}, and \eqref{QP eq 2},
it follows that
\begin{align*}
\left\| A_Q A_R^{-1} \right\| &\leq \left\| A_Q A_S^{-1} \right\|\left\| A_S A_R^{-1} \right\|
 \lesssim \left[ \frac{l(S)}{l(Q)} \right]^{d_1} \left[ \frac{l(S)}{l(R)} \right]^{d_2}\\
&\sim \left[ \frac{l(Q)\vee l(R)}{l(Q)} \right]^{d_1} \left[ \frac{l(Q)\vee l(R)}{l(R)} \right]^{d_2}
\left[ \frac{l(S)}{l(Q)\vee l(R)} \right]^{d_1 + d_2}\\
&\sim \max\left\{ \left[ \frac{l(R)}{l(Q)} \right]^{d_1}, \left[ \frac{l(Q)}{l(R)} \right]^{d_2} \right\}
\left[ 1+ \frac{|c_Q - c_R|}{l(Q)\vee l(R)} \right]^{\Delta}.
\end{align*}
This finishes the proof of Proposition \ref{QP3}.
\end{proof}
The following lemma is exactly \cite[Lemma 2.31]{bhyy23}.
\begin{lemma}\label{QP4}
For any cubes $Q,R\subset \mathbb{R}^n$,
any $x,x' \in Q$, and any $y,y' \in R$,
$$ 1+\frac{|x-y|}{l(Q)\vee l(R)} \sim 1+\frac{|x'-y'|}{l(Q)\vee l(R)}, $$
where the positive equivalence constants depend only on $n$.
\end{lemma}
Combining Proposition \ref{QP3} and Lemma \ref{QP4},
we immediately obtain the following corollary.
\begin{corollary}
Let $p(\cdot) \in \mathcal{P}_0$ with $p(\cdot) \in LH$, $W \in \mathscr{A}_{p(\cdot),\infty}$,
$d_1 \in [\![ d^{\rm lower}_{p(\cdot),\infty}(W) ,\frac{n}{p_-})$,
$d_2 \in [\![ d^{\rm upper}_{p(\cdot),\infty}(W) ,\infty)$,
and $\{A_Q\}_Q$ be a family of reducing operators of order $p(\cdot)$ for $W$.
Then there exists a positive constant $C$ such that,
for any cubes $Q$ and $R$ in $\mathbb{R}^n$,
$$ \left\| A_Q A_R^{-1} \right\| \leq C \max\left\{ \left[ \frac{l(R)}{l(Q)} \right]^{d_1},
\left[ \frac{l(Q)}{l(R)} \right]^{d_2} \right\}\left[ 1+ \frac{|x_Q - x_R|}{l(Q)\vee l(R)} \right]^{\Delta}, $$
where $x_Q$ and $x_R$ are respectively any points of $Q$ and $R$ and $\Delta$ is the same as in \eqref{AQAR-1}.
\end{corollary}

%

\bigskip

\noindent\textbf{Author Contributions}\quad All authors shared equally
the writing and the reviewing of the main manuscript text.

\bigskip

\noindent\textbf{Funding}\quad This project is partially supported by
the National Natural Science Foundation of China
(Grant Nos. 12431006 and 12371093),
the Beijing Natural Science Foundation
(Grant No. 1262011),
and the Fundamental Research Funds for the Central Universities
(Grant No. 2253200028).

\bigskip

\noindent\textbf{Data Availability}\quad No datasets were generated
or analysed during the current study.

\section*{Declarations}

\noindent\textbf{Competing interests}\quad The authors declare no competing interests.

\bigskip

\noindent\textbf{Ethical Approval and Consent to participate}\quad
The authors declare ethical approval and consent to participate.

\bigskip

\noindent\textbf{Consent for Publication}\quad The authors declare consent for publication.

\bigskip

\noindent\textbf{Human and Animal Ethics}\quad Not applicable.

\bigskip

\noindent Dachun Yang (Corresponding author), Wen Yuan and Zongze Zeng

\medskip

\noindent  Laboratory of Mathematics and Complex Systems
(Ministry of Education of China),
School of Mathematical Sciences, Institute for Advanced Study,
Beijing Normal University,
Beijing 100875, The People's Republic of China

\smallskip

\noindent{\it E-mails}:
\texttt{dcyang@bnu.edu.cn} (D. Yang)

\noindent\phantom{{\it E-mails:}}
\texttt{wenyuan@bnu.edu.cn} (W. Yuan)

\noindent \phantom{{\it E-mails:}}
\texttt{zzzeng@mail.bnu.edu.cn} (Z. Zeng)

\begin{thebibliography}{99}

\bibitem{am02}
E. Acerbi and G. Mingione,
Regularity results for electrorheological fluids: the stationary case,
C. R. Math. Acad. Sci. Paris 334 (2002), 817--822.

\vspace{-0.3cm}

\bibitem{am02 1}
E. Acerbi and G. Mingione,
Regularity results for stationary electro-rheological fluids,
Arch. Ration. Mech. Anal. 164 (2022), 213--259.

\vspace{-0.3cm}

\bibitem{am05}
E. Acerbi and G. Mingione,
Gradient estimates for the $p(x)$-Laplacean system,
J. Reine Angew. Math. 584 (2005), 117--148.

\vspace{-0.3cm}

\bibitem{bgx25} T. Bai, P. Guo and J. Xu, On matrix weighted
Bourgain--Morrey Triebel--Lizorkin spaces, arXiv: 2508.11981.

\vspace{-0.3cm}

\bibitem{bx24a}	
T. Bai and J. Xu,
Pseudo-differential operators on matrix weighted
Besov--Triebel--Lizorkin spaces,
Bull. Iranian Math. Soc. 50 (2024), Paper No. 31, 26 pp.

\vspace{-0.3cm}

\bibitem{bx24b}	
T. Bai and J. Xu,
Non-regular pseudo-differential operators on
matrix weighted Besov--Triebel--Lizorkin spaces,
J. Math. Study 57 (2024), 84--100.

\vspace{-0.3cm}

\bibitem{bx24c}	
T. Bai and J. Xu,
Precompactness in matrix weighted Bourgain-Morrey spaces,
Filomat 39 (2025), 6261--6280.

\vspace{-0.3cm}

\bibitem{bhs11}
B. Bongioanni, E. Harboure and O. Salinas,
Classes of weights related to Schr\"odinger operators,
J. Math. Anal. Appl. 373 (2011), 563--579.

\vspace{-0.3cm}

\bibitem{b01}
M. Bownik,
Inverse volume inequalities for matrix weights,
Indiana Univ. Math. J. 50 (2001), 383--410.

\vspace{-0.3cm}

\bibitem{bc22}
M. Bownik and D. Cruz-Uribe,
Extrapolation and factorization of matrix weights,
Math. Ann. 395 (2026), Page No. 55, 84 pp.

\vspace{-0.3cm}

\bibitem{bcyy24}
F. Bu, Y. Chen, D. Yang and W. Yuan,
Maximal function and atomic characterizations of matrix-weighted Hardy spaces
with their applications to boundedness of Calder\'on--Zygmund operators,
arXiv: 2501.18800.

\vspace{-0.3cm}

\bibitem{bhyy23}
F. Bu, T. Hyt\"onen, D. Yang and W. Yuan,
Matrix-weighted Besov-type and Triebel--Lizorkin-type spaces I:
$A_p$-dimensions of matrix weights and $\psi$-transform characterizations,
Math. Ann. 391 (2025), 6105--6185.

\vspace{-0.3cm}

\bibitem{bhyy23 2}
F. Bu, T. Hyt\"onen, D. Yang and W. Yuan,
Matrix-weighted Besov-type and Triebel--Lizorkin-type spaces II:
Sharp boundedness of almost diagonal operators,
J. Lond. Math. Soc. (2) 111 (2025), Paper No. e70094, 59 pp.

\vspace{-0.3cm}

\bibitem{bhyy23 3}
F. Bu, T. Hyt\"onen, D. Yang and W. Yuan,
Matrix-weighted Besov-type and Triebel--Lizorkin-type spaces III:
Characterizations of molecules and wavelets, trace theorems,
and boundedness of pseudo-differential operators and Calder\'on--Zygmund operators,
Math. Z. 308 (2024), Paper No. 32, 67 pp.

\vspace{-0.3cm}

\bibitem{bhyy-a}
F. Bu, T. Hyt\"{o}nen, D. Yang and W. Yuan,
New characterizations and properties of matrix $A_\infty$ weights,
Acta Math. Sin. (Engl. Ser.) 42 (2026), 687--722.

\vspace{-0.3cm}

\bibitem{bhyy24}
F. Bu, T. Hyt\"onen, D. Yang and W. Yuan,
Besov--Triebel--Lizorkin-type spaces with matrix $A_\infty$ weights,
Sci. China Math. 69 (2026), 383--460.

\vspace{-0.3cm}

\bibitem{byyz25}
F. Bu, D. Yang, W. Yuan and M. Zhang,
Matrix-weighted Besov--Triebel--Lizorkin spaces of optimal scale:
Real-variable characterizations,
invariance on integrable index, and Sobolev-type embedding,
J. Differential Equations (2026), 463 (2026), Paper No. 114140, 101 pp.

\vspace{-0.3cm}

\bibitem{byyyz25}
F. Bu, D. Yang, W. Yuan and Y. Zhao,
Matrix weights, maximal operators, Calder\'on--Zygmund operators,
and Besov--Triebel--Lizorkin-type spaces --- A survey,
Anal. Theory Appl. 41 (2025), 371--468.

\vspace{-0.3cm}

\bibitem{c07}
C. Carath\'eodory,
\"Uber den Variabilit\"atsbereich der Koeffizienten von Potenzreihen,
die gegebene Werte nicht annehmen,
Math. Ann. 64 (1907), 95--115.

\vspace{-0.3cm}

\bibitem{cyy25} Y. Chen, D. Yang and W. Yuan,
Matrix-weighted Campanato spaces:
Duality and Calder\'on--Zygmund operators,
Acta Math. Sci. Ser. B (Engl. Ed.) (to appear).

\vspace{-0.3cm}

\bibitem{cg01}
M. Christ and M. Goldberg,
Vector $A_2$ weights and a Hardy--Littlewood maximal function,
Trans. Amer. Math. Soc. 353 (2001), 1995--2002.

\vspace{-0.3cm}

\bibitem{cf74}
R. R. Coifman and C. Fefferman,
Weighted norm inequalities for maximal functions and singular integrals,
Studia Math. 51 (1974), 241--250.

\vspace{-0.3cm}

\bibitem{c25} D. Cruz-Uribe, Recent results on matrix weighted
norm inequalities, arXiv: 2508.13352.

\vspace{-0.3cm}

\bibitem{cc22}
D. Cruz-Uribe and J. Cummings,
Weighted norm inequalities for the maximal operator on $L^{p(\cdot)}$ over spaces of homogeneous type,
Ann. Fenn. Math. 47 (2022), 457--488.


\vspace{-0.3cm}

\bibitem{cdh11}
D. Cruz-Uribe, L. Diening and P. H\"ast\"o,
The maximal operator on weighted variable Lebesgue spaces,
Fract. Calc. Appl. Anal. 14 (2011), 361--374.

\vspace{-0.3cm}

\bibitem{cf13}
D. Cruz-Uribe and A. Fiorenza,
Variable Lebesgue Space. Foundations and Harmonic Analysis,
Appl. Number. Harmon. Anal., Birkh\"auser/Springer, Heidelberg, 2013.

\vspace{-0.3cm}

\bibitem{cfn12}
D. Cruz-Uribe, A. Fiorenza and C. J. Neugebauer,
Weighted norm inequalities for the maximal operator on variable Lebesgue spaces,
J. Math. Anal. Appl. 394 (2012), 744--760.

\vspace{-0.3cm}

\bibitem{cn95}
D. Cruz-Uribe and C. J. Neugebauer,
The structure of the reverse H\"older classes,
Trans. Amer. Math. Soc. 347 (1995), 2941--2960.


\vspace{-0.3cm}

\bibitem{cp23}
D. Cruz-Uribe and M. Penrod,
Convolution operators in matrix weighted, variable Lebesgue spaces,
Anal. Appl. (Singap.) 22 (2024), 1133--1157.

\vspace{-0.3cm}

\bibitem{cp24}
D. Cruz-Uribe and M. Penrod,
The reverse H\"older inequality for $\mathcal{A}_{p(\cdot)}$
weights with applications to matrix weights,
arXiv: 2411.12849.

\vspace{-0.3cm}

\bibitem{cr24}
D. Cruz-Uribe and T. Roberts,
Necessary conditions for the boundedness of fractional operators on variable Lebesgue spaces,
arXiv: 2408.12745.

\vspace{-0.3cm}

\bibitem{cs25}
D. Cruz-Uribe and F. \c{S}irin,
Off-diagonal matrix extrapolation for Muckenhoupt bases,
Ann. Fenn. Math. 50 (2025), 577--609.

\vspace{-0.3cm}

\bibitem{cs23}
D. Cruz-Uribe and B. Sweeting,
Weighted weak-type inequalities for maximal operators and singular integrals,
Rev. Mat. Complut. 38 (2025), 183--205.

\vspace{-0.3cm}

\bibitem{cw17}
D. Cruz-Uribe and L. D. Wang,
Extrapolation and weighted norm inequalities in the variable Lebesgue spaces,
Trans. Amer. Math. Soc. 369 (2017), 1205--1235.

\vspace{-0.3cm}

\bibitem{dhl20}
F. Di Plinio, T. Hyt\"{o}nen and K. Li,
Sparse bounds for maximal rough singular
integrals via the Fourier transform,
Ann. Inst. Fourier (Grenoble) 70 (2020), 1871--1902.

\vspace{-0.3cm}

\bibitem{dhr17}
L. Diening, P. Harjulehto, P. H\"ast\"o and M. R\r{u}{\v z}i{\v c}ka,
Lebesgue and Sobolev Spaces with Variable Exponents,
Lecture Notes in Mathematics 2017, Springer, Heidelberg, 2011.

\vspace{-0.3cm}

\bibitem{dh08}
L. Diening and P. H\"ast\"o,
Muckenhoupt weights in variable exponent spaces,
https://www. researchgate.net/publication/228779582, 2008.

\vspace{-0.3cm}

\bibitem{dptv24}
K. Domelevo, S. Petermichl, S. Treil and A. Volberg,
The matrix $A_2$ conjecture fails, i.e. $3/2>1$,
arXiv: 2402.06961.

\vspace{-0.3cm}

\bibitem{dk18}
H. Dong and D. Kim,
On $L_p$-estimates for elliptic and parabolic equations with $A_p$ weights,
Trans. Amer. Math. Soc. 370 (2018), 5081--5130.

\vspace{-0.3cm}

\bibitem{d15}
D. Drihem,
Some properties of variable Besov-type spaces,
Funct. Approx. Comment. Math. 52 (2015), 193--221.

\vspace{-0.3cm}

\bibitem{dly21} X. T. Duong, J. Li and D. Yang, Variation of
Calder\'on--Zygmund operators with matrix weight, Commun. Contemp.
Math. 23 (2021), Paper No. 2050062, 30 pp.

\vspace{-0.3cm}

\bibitem{fr04}
M. Frazier and S. Roudenko,
Matrix-weighted Besov spaces and conditions of $A_p$ type for
$0<p\leq 1$,
Indiana Univ. Math. J. 53 (2004), 1225--1254.

\vspace{-0.3cm}

\bibitem{fr21}
M. Frazier and S. Roudenko,
Littlewood--Paley theory for matrix-weighted function spaces,
Math. Ann. 380 (2021), 487--537.

\vspace{-0.3cm}

\bibitem{gl86}
N. Garofalo and F.-H. Lin,
Monotonicity properties of variational integrals, $A_p$ weights and unique continuation,
Indiana Univ. Math. J. 35 (1986), 245--268.

\vspace{-0.3cm}

\bibitem{g03}
M. Goldberg, Matrix $A_p$ weights via maximal functions,
Pacific J. Math. 211 (2003), 201--220.

\vspace{-0.3cm}

\bibitem{g14}
L. Grafakos,
Classical Fourier Analysis,
Third edition, Grad. Texts in Math. 249,
Springer, New York, 2014.

\vspace{-0.3cm}

\bibitem{g14 1}
L. Grafakos,
Modern Fourier Analysis,
Third edition, Grad. Texts in Math. 250,
Springer, New York, 2014.

\vspace{-0.3cm}

\bibitem{hmw73}
R. Hunt, B. Muckenhoupt and R. Wheeden,
Weighted norm inequalities for the conjugate function and Hilbert transform,
Trans. Amer. Math. Soc. 176 (1973), 227--251.

\vspace{-0.3cm}

\bibitem{h12}
T. Hyt\"onen,
The sharp weighted bound for general Calder\'on--Zygmund operators,
Ann. of Math. (2) 175 (2012), 1473--1506.

\vspace{-0.3cm}

\bibitem{hp13}
T. Hyt\"onen and C. P\'erez,
Sharp weighted bounds involving $A_\infty$,
Anal. PDE 6 (2013), 777--818.

\vspace{-0.3cm}

\bibitem{kn24}
S. Kakaroumpas and Z. Nieraeth,
Multilinear matrix weights,
Adv. Math. 486 (2026), Paper No. 110744, 85 pp.

\vspace{-0.3cm}

\bibitem{llor23}
A. Lerner, K. Li, S. Ombrosi and I. Rivera-R\'{\i}os,
On the sharpness of some quantitative
Muckenhoupt--Wheeden inequalities,
C. R. Math. Acad. Sci. Paris 362 (2024), 1253--1261.

\vspace{-0.3cm}

\bibitem{llor24}
A. Lerner, K. Li, S. Ombrosi and I. Rivera-R\'{\i}os,
On some improved weighted weak type inequalities,
Ann. Sc. Norm. Super. Pisa Cl. Sci. (5) (2024),
https://doi.org/10.2422/2036-2145.202407\_012.

\vspace{-0.3cm}

\bibitem{lyy24a}
Z. Li, D. Yang and W. Yuan,
Matrix-weighted Besov--Triebel--Lizorkin spaces
with logarithmic smoothness,
Bull. Sci. Math. 193 (2024), Paper No. 103445, 54 pp.

\vspace{-0.3cm}

\bibitem{lyy24b}
Z. Li, D. Yang and W. Yuan,
Matrix-weighted Poincar\'{e}-type inequalities
with applications to logarithmic
Haj\l asz--Besov spaces on spaces of homogeneous type,
arXiv:2602.14361.

\vspace{-0.3cm}

\bibitem{n12} M. Nielsen, On transference of multipliers on
matrix weighted $L^p$-spaces, J. Geom. Anal. 22 (2012),
12--22.

\vspace{-0.3cm}

\bibitem{n25} M. Nielsen, Matrix weighted $\alpha$-modulation spaces,
Monatsh. Math. 206 (2025), 419--448.

\vspace{-0.3cm}

\bibitem{n25-2} M. Nielsen, Bandlimited multipliers on
matrix-weighted $L^p$-spaces, J. Fourier Anal. Appl. 31 (2025),
Paper No. 3, 10 pp.

\vspace{-0.3cm}

\bibitem{nr18} M. Nielsen and M. G. Rasmussen, Projection operators on matrix
weighted $L^p$ and a simple sufficient Muckenhoupt condition,
Math. Scand. 123 (2018), 72--84.


\vspace{-0.3cm}

\bibitem{m72}
B. Muckenhoupt,
Weighted norm inequalities for the Hardy maximal function,
Trans. Amer. Math. Soc. 165 (1972), 207--226.

\vspace{-0.3cm}

\bibitem{nptv17}
F. Nazarov, S. Petermichl, S. Treil and A. Volberg,
Convex body domination and weighted estimates with matrix weights,
Adv. Math. 318 (2017), 279--306.

\vspace{-0.3cm}

\bibitem{nt96}
F. L. Nazarov and S. R. Treil,
The hunt for a Bellman function: applications to estimates for singular integral operators and to other classical problems of harmonic analysis, (Russian),
translated from Algebra i Analiz 8 (1996), 32--162, St. Petersburg Math. J. 8 (1997), 721--824.

\vspace{-0.3cm}

\bibitem{n24}
Z. Nieraeth,
A lattice approach to matrix weights,
Math. Ann. 393 (2025), 993--1072.

\vspace{-0.3cm}

\bibitem{np25}
Z. Nieraeth and M. Penrod,
Matrix-weighted bounds in variable Lebesgue spaces,
Ann. Fenn. Math. 50 (2025), 519--548.

\vspace{-0.3cm}

\bibitem{o31}
W. Orlicz, {\"U}ber konjugierte Exponentenfolgen,
Studia Math. 3 (1931) 200--212.

\vspace{-0.3cm}

\bibitem{r03}
S. Roudenko,
Matrix-weighted Besov spaces,
Trans. Amer. Math. Soc. 355 (2003), 273--314.

\vspace{-0.3cm}

\bibitem{rou04}
S. Roudenko,
Duality of matrix-weighted Besov spaces,
Studia Math. 160 (2004), 129--156.

\vspace{-0.3cm}

\bibitem{r00}
M. R\r{u}{\v z}i{\v c}ka,
Electrorheological Fluids: Modeling and Mathematical Theory,
Lecture Notes in Mathematics 1748, Springer-Verlag, Berlin, 2000.

\vspace{-0.3cm}

\bibitem{r04}
M. R\r{u}{\v z}i{\v c}ka,
Modeling, mathematical and numerical analysis of electrorheological fluids,
Appl. Math. 49 (2004), 565--609.

\vspace{-0.3cm}

\bibitem{tv97}
S. Treil and A. Volberg,
Wavelets and the angle between past and future, J. Funct. Anal.
143 (1997), 269--308.

\vspace{-0.3cm}

\bibitem{v97}
A. Volberg,
Matrix $A_p$ weights via $\mathcal S$-functions,
J. Amer. Math. Soc. 10 (1997), 445--466.

\vspace{-0.3cm}

\bibitem{wyy23}
Q. Wang, D. Yang and Y. Zhang,
Real-variable characterizations and their
applications of matrix-weighted
Triebel--Lizorkin spaces,
J. Math. Anal. Appl. 529 (2024),
Paper No. 127629, 37 pp.

\vspace{-0.3cm}

\bibitem{wgx24}
S. Wang, P. Guo and J. Xu,
Characterizations of weighted Besov spaces with variable exponents,
Acta Math. Sin. (Engl. Ser.) 40 (2024), 2855--2878.

\vspace{-0.3cm}

\bibitem{wgx25a}
S. Wang, P. Guo and J. Xu,
Embedding and duality of matrix-weighted
modulation spaces,
Taiwanese J. Math. 29 (2025), 171--187.

\vspace{-0.3cm}

\bibitem{wgx25b}
S. Wang, P. Guo and J. Xu,
Precompact sets in matrix weighted Lebesgue spaces with variable exponent,
Georgian Math. J. 32 (2025), 1071--1083.

\vspace{-0.3cm}

\bibitem{wx22}
S. Wang and J. Xu,
Weighted Besov spaces with variable exponents,
J. Math. Anal. Appl. 505 (2022), Paper No. 125478, 27 pp.

\vspace{-0.3cm}

\bibitem{w94}
R. Webster, Convexity,
Oxford Science Publications, The Clarendon Press, Oxford University Press, New York, 1994.

\vspace{-0.3cm}

\bibitem{wm58}
N. Wiener and P. Masani,
The prediction theory of multivariate stochastic processes. II. The
linear predictor, Acta Math. 99 (1958), 93--137.

\vspace{-0.3cm}

\bibitem{yyz25} D. Yang, W. Yuan and Z. Zeng, Variable matrix-weighted
Besov spaces, Submitted or arXiv:2509.07786.

\vspace{-0.3cm}

\bibitem{yymz25}
D. Yang, W. Yuan and M. Zhang,
Matrix-weighted Besov--Triebel--Lizorkin
spaces of optimal scale:
boundedness of pseudo-differential, trace, and Calder\'on--Zygmund
operators, Proc. Steklov Inst. Math. 331 (2025), 297--333.

\end{thebibliography}
\end{document}